\documentclass[11pt,a4paper]{amsart}
\usepackage{amssymb,amsfonts,amsmath,mathrsfs}
\usepackage[left=2cm,right=2cm,top=2cm,height=250mm]{geometry}
\usepackage{color}
\usepackage{comment}
\newtheorem{prop}{Proposition}[section]
\newtheorem{theorem}[prop]{Theorem}
\newtheorem{cor}[prop]{Corollary}
\newtheorem{lemma}[prop]{Lemma}

\theoremstyle{definition}
\newtheorem{Def}[prop]{Definition}

\newtheorem{rem}[prop]{Remark}
\newtheorem{example}[prop]{Example}

\numberwithin{equation}{section}

\def\R{\Bbb R}
\def\T{\Bbb T}
\def\N{\Bbb N}
\def\Dx{\Delta_x}
\def\Nx{\nabla_x}
\def\Dt{\partial_t}
\def\({\left(}
\def\){\right)}
\def\eb{\varepsilon}

\def\Cal{\mathcal}
\def\l{\lambda}

\def\E{\mathcal{E}}

\def\Var{\operatorname{Var}}
\def\supp{\operatorname{supp}}

\def\l2l2{\Cal{L}(L^2(\Omega))}

\def\Tau{\Upsilon}
\def\Bbb{\mathbb}
\begin{document}
\title[Measure driven wave equation]{Uniform attractors for measure-driven quintic  wave equation with periodic boundary conditions}
\author[Savostianov and Zelik] {Anton Savostianov${}^1$ and  Sergey Zelik${}^2$}

\begin{abstract} We give a detailed study of  attractors for measure driven quintic damped wave equations with periodic boundary conditions. This includes uniform energy-to-Strichartz estimates, the existence of uniform attractors in a weak or strong topology in the energy phase space, the possibility to present them as a union of all complete trajectories, further regularity, etc.
\end{abstract}


\thanks{
 This work is partially supported by  the grants  14-41-00044 and 14-21-00025 of RSF as well as  grants 14-01-00346  and 15-01-03587 of RFBR and the EPSRC grant EP/P024920/1.
    Part of this work was done while the first author held a postdoctoral position in the University of Cergy-Pontoise. He  would like to thank N. Tzvetkov and  A. Shirikyan for hospitality, interesting and stimulating discussions.
  The authors are also would like to thank V. Chepyzhov, A. Mielke and O. Smolyanov for stimulating discussions.}

\address{${}^1$ Durham University, Department of Mathematical Sciences, Durham, DH1 3LE, United Kingdom.}
\email{anton.savostianov@durham.ac.uk}
\address{${}^2$ University of Surrey, Department of Mathematics, Guildford, GU2 7XH, United Kingdom.}
 \email{s.zelik@surrey.ac.uk}

\maketitle
\tableofcontents
\section{Introduction}\label{s.i}
We study the following nonautonomous damped wave equation:
\begin{equation}\label{eq.qdw}
%\begin{cases}
\Dt^2 u+\gamma\Dt u+(1-\Dx)u+f(u)=\mu(t),\ \ t\ge\tau,\ \
\{u,\Dt u\}\big|_{t=\tau}=\{u_\tau,u_\tau'\}
%\end{cases}
\end{equation}
%$$
in a bounded domain $\T^3:=(-\pi,\pi)^3$ of $\R^3$ endowed with periodic boundary conditions.
Here $u(t,x)$ is the unknown function, $\Dx$ is the  Laplacian with respect to variable $x$, $\gamma$ is a  positive constant, $f:\R\to\R$ is a given non-linearity which is assumed to be of \emph{quintic} growth ($f(u)\sim u^5$) and to satisfy some natural conditions (stated in \eqref{4.f}) and $\mu$ is a given external force which is a $L^2$-valued measure of finite total variation which is assumed to be  uniformly bounded on bounded time intervals: $\mu\in M_b(\R,H)$, see Section \ref{s.p} for definitions of key functional spaces.
\par
Dispersive or/and dissipative semilinear wave equations of the form \eqref{eq.qdw}  model various oscillatory processes
in many areas of  modern mathematical physics including  electrodynamics, quantum mechanics, nonlinear elasticity,  etc. and are of a big permanent interest, see \cite{lions,BV,Te,CV,straus,tao,sogge} and references therein.
\par
The basic property of this equation is the so-called energy identity:
%$$
\begin{equation}\label{0.energy}
\mathcal E(\xi_u(t))-\mathcal E(\xi_u(\tau))=-\gamma\|\Dt u\|^2_{L^2}+\int_\tau^t(\Dt u(s),\mu(ds)), \ \ \xi_u(t):=\{u(t),\Dt u(t)\}
\end{equation}
%$$
which can be formally obtained by multiplying equation \eqref{eq.qdw} by $\Dt u$ and integrating over $t$ and $x$. Here
$$
\E(\xi_u):=\frac12\(\|\Dt u\|^2_{L^2}+\|\Nx u\|^2_{L^2}+\|u\|^2_{L^2}+2(F(u),1)\),\ \ F(u):=\int_0^u f(z)\,dz
$$
and $(u,v):=\int_{\T^3}u(x)v(x)\,dx$. This identity motivates the natural choice of the energy phase space and the class of energy solutions (as the solutions for which the energy functional is finite) and also gives the control of the energy norm of the solution. Namely, if the non-linearity has a sub-quintic or quintic growth rate, due to the Sobolev embedding theorem $H^1\subset L^6$, the energy space is given by $\E:=H^1(\T^3)\times L^2(\T^3)$ and in the supercritical case $f(u)\sim u|u|^q$ with $q>4$, we need to take $\E:=(H^1(\T^3)\cap L^{q+2}(\T^3))\times L^2(\T^3)$ in order to guarantee the finiteness of the energy functional.
\par
It is believed that the analytic properties and the dynamics as $t\to\infty$ of solutions for damped wave equations \eqref{eq.qdw} strongly depend on the growth rate of the non-linearity $f(u)$ as $u\to\infty$. Indeed, in the most studied case of cubic and sub-cubic growth rate, the control of the energy norm is sufficient to get the well-posedness, dissipativity and further regularity of solutions as well as to develop the corresponding attractors theory in both autonomous and non-autonomous cases, see \cite{HR,BV,CV,Lad,lions,MirZel,Te,ZCPAA2004} and references therein.
\par
The case of super-cubic but sub-quintic growth rate ($2<q<4$) is a bit more complicated since the well-posedness of energy solutions is still an open problem here. However, this problem can be overcome using slightly more regular solutions than the energy ones for which, say, the mixed $L^4(\tau,T;L^{12}(\T^3))$ space-time norm is finite for every $T>\tau$. These are  the so-called Shatah-Struwe (or Strichartz) solutions. The existence of such solutions is strongly based on the Strichartz estimates for the {\it linear} wave equation (see Theorem \ref{th.linstr} below) which are now available not only for the whole space $\R^3$ or the torus $\T^3$, but also for bounded domains with Dirichlet or Neumann boundary conditions, see \cite{Chem,Sogge2009,plan1,plan2,straus,tao}. Moreover, crucial for the attractor theory is the following energy-to-Strichartz estimate for such solutions
%$$
\begin{equation}\label{0.es}
\|u\|_{L^4(T,T+1;L^{12})}\le Q(\|\xi_u(T)\|_{\E})+Q(\|\mu\|_{L^1(T,T+1;L^2)}),
\end{equation}
%$$
where $Q$ is monotone increasing function which is independent of $T$ and the solution $u$. In the sub-quintic case this estimate is a straightforward corollary of the linear Strichartz estimate and perturbation arguments. Energy-to-Strichartz estimate \eqref{0.es} allows us to deduce the control and establish the dissipativity of $u$ in the Strichartz norm based on the standard energy estimate. Since the control of this norm is enough for the uniqueness, the obtained control gives the well-posedness, dissipativity and the existence of global/uniform attractors in the way which is similar to the clasical cubic case, see \cite{feireisl},\cite{kap4} and \cite{KSZ} for the case of $\R^3$, $\T^3$ and a bounded domain endowed with the Dirichlet boundary conditions respectively (see also \cite{SADE} for the case of damped wave equations with fractional damping).
\par
In contrast to this, very few is known about the solutions of \eqref{eq.qdw} in the supercritical (superquintic) growth rate of the non-linearity $f$. In this case the situation is somehow close to 3D Navier-Stokes problem, namely, we have global existence of weak energy solutions for which we do not know the uniqueness theorem and the local existence of more regular solutions for which we do not know the global existence. It is expected that smooth solutions may blow up in finite time even in the defocusing case, but to the best of our knowledge there are no such examples. In this case the existing attractor theory  is related to multilavued semigroups or/and the so-called trajectory dynamical  systems and trajectory attractors, see \cite{CV,CV1,MirZel,ZelDCDS} (see also references therein).
\par
We now turn to the most interesting  borderline case of critical quintic non-linearity $f$ which is our main object of our study in this paper. In this case, the energy-to-Strichartz estimate \eqref{0.es} does not follow any more from the Strichartz estimate for the linear equation (at least in a straightforward way), so the proof of global existence for Shatah-Struwe solutions is usually based on the so-called non-concentration arguments and Pohozhaev-Morawetz equality, see \cite{Chem,Grill,kap1,kap2,kap3,SS1,SS,sogge,tao} (see also \cite{plan1,plan2} for the case of bounded domains with Dirichlet or Neumann boundary conditions). This approach allows us to construct a Shatah-Struwe solution $u$ such that the $L^4(\tau,T;L^{12})$-norm is finite for all $T$, but does not allow to get any control of this norm through the energy norm or to verify that the Strichartz norm does not grow as $T\to\infty$. This is clearly not sufficient for the attractors. Indeed, without the uniform control of the Strichartz norm as $T\to\infty$, this extra regularity may a priori be lost in the limit and the attractor may contain the solutions which are less regular that the Shatah-Struwe ones (for which we do not have the uniqueness theorem). Thus,  the uniform control of the Strichartz norm is crucial for the attractor theory.
\par
This problem has been overcome in \cite{KSZ} where the asymptotic regularity and existence of global attractors for {\it autonomous} quintic wave equiation in bounded domains of $\R^3$ has been established. The method sugested there is heavily based one existence of global Lyapunov function and on the related convergence of the trajectories to the set of equilibria and, by this reason cannot be extended to the non-autonomous case. Moreover, to the best of our knowledge, up to the moment there were no results on the attractor theory for quintic wave equations in the non-autonomous case.
\par
The main aim of the present paper is to give a comprehensive study of the {\it non-autonomous} quintic wave equation in the case of periodic boundary conditions. In order to do so, we first prove the energy-to-Strichartz estimate \eqref{0.es} for the Shatah-Struwe solutions of \eqref{eq.qdw} for the quintic case as well. Therefore, the following theorem can be considered as our first main result.
\begin{theorem}\label{Th0.estr} Let the non-linearity $f$ satisfy the assumptions \eqref{4.f}. Then, problem \eqref{eq.qdw} is globally well-posed in the class of Shatah-Struwe solutions,  any such soultion $u(t)$ satisfies the energy-to-Strichartz estimate \eqref{0.es} and the following dissipative estimate holds:
%$$
\begin{equation}\label{0.enst}
\|\xi_u(t)\|_\E+\|u\|_{L^4(t,t+1;L^{12})}\le Q(\|\xi_u(\tau)\|_\E)e^{-\delta (t-\tau)}+Q(\|\mu\|_{M_b(\R,L^2)}),
\end{equation}
%$$
where the positive constant $\delta$ and monotone function $Q$ are independent of $t$, $\tau$, and the solution $u$.
\end{theorem}
The non-trivial part here is exactly to establish the energy-to-Strichartz estimate \eqref{0.es} (the rest is a standard corollary of this estimate and the classical dissipative energy estimate). To do so we start with the analogous energy-to-Strichartz estimate for the Shatah-Struwe solutions for the quintic wave equation in the whole space $\R^3$:
%$$
\begin{equation}\label{0.quin}
\Dt^2v-\Dx v+v^5=0
\end{equation}
%$$
proved in \cite{BG1999} (see also \cite{tao1} for the explicit bounds for the function $Q$) and extend it to the non-autonomous case
$$
\Dt^2 v-\Dx v+v^5=\mu(t).
$$
This extension uses the approximation of the external force $\mu(t)$ by sums of Dirac $\delta$-measures and  presentation of the solution $v$ for such external forces via the solutions of the autonomous equation. This approach can be interpreted as the analogue of the Duhamel formula for the non-linear equation and has an independent interest. We would like to emphasize that this method requires to consider {\it measure-driven} equations of the form \eqref{0.es} as an intermediate step even if we finally want to verify estimate \eqref{0.enst} for regular external forces $\mu\in L^1_b(\R,L^2)$, see Section \ref{s.strna} for details. This is one of the sources of motivation for us to consider measure driven damped wave equations. Of course, measure driven equations are interesing and important by themselves, we metion here only that they are widely used in the theory of stochastic PDEs, see \cite{kuksin} and references therein. Note also that the analogue of the energy-to-Strichartz estimate for the case of equation \eqref{0.quin} in bounded domains (with Dirichlet or Neumann boundary conditions) is not known so far and this is the main reason for our choice of periodic boundary conditions.
\par
We now turn to the attractor theory. We first note that the dissipative estimate \eqref{0.enst} implies in a standard way the existence of a {\it uniform} attractor $\mathcal A_{un}$ for equation \eqref{eq.qdw} in a weak topology of the energy space $\E$, see Section \ref{s.ua}. However, new features arise when we try to describe the uniform attractor in terms of bounded complete trajectories related to equation \eqref{eq.qdw}. We recall that, following the general theory developed in \cite{CV,CV1}, in order to obtain such a description we need to study not only equation \eqref{eq.qdw}, but also all its time shifts as well as their limits in the proper topology. In our case it is natural to take the closure of all time shifts of the initial measure $\mu$ in a weak star toplogy, generated by the duality
$$
M_{loc}(\R,L^2)=[C_{00}(\R,L^2)]^*,
$$
where $C_{00}$ stands for continuous functions with compact support. Namely, we introduce the group of time shifts $T(h):M_b(\R,L^2)\to M_b(\R,L^2)$ via $(T(h)\mu)(t)=\mu(t+h)$ and define the hull of the given measure $\mu$ as follows:
$$
\mathcal H(\mu):=[T(h)\mu,\, h\in\R]_{M_{loc}^{w^*}(\R,L^2)},
$$
see Section \ref{s.ua} for more details. Then, the general theory predicts the representation
%$$
\begin{equation}\label{0.astr}
\mathcal A_{un}=\cup_{z\in\mathcal H(\mu)}\mathcal K_z\big|_{t=0},
\end{equation}
%$$
where $\mathcal K_z$ is a set of complete (defined for all $t\in\R$) bounded (in $\E$) solutions of equation \eqref{eq.qdw} with the right-hand side $z$. Again, according to the general theory, this reperesentation will hold if the solution operators $U_z(t,\tau)$ (which map the initial data $\xi_\tau$ to the Shatah-Struwe solution $\xi_u(t)$ of problem \eqref{eq.qdw} with the right-hand side $z\in\mathcal H(\mu)$) are weakly star continuous as maps from $\mathcal E\times\mathcal H(\mu)$ to $\E$.
\par
Unfortunately, in contrast to the standard situations, considered in \cite{CV}, the map $z\to U_z(t,\tau)$ may be {\it discontinuous} for the case of measure driven equations. As shown in Section \ref{s.ua} this may destroy (and destroys in concrete examples given there) the representation formula \eqref{0.astr}. Actually the attractor $\mathcal A_{un}$ may become {\it larger} than the union of all bounded complete trajectories. In order to avoid this pathology, we found necessary and sufficient conditions for the measure $\mu$ which guarantee the continuity of the map $z\to U_z(t,\tau)$. Particulalrly, these restrictions forbid the measures $z\in \mathcal H(\mu)$ to have non-zero discrete parts. By this reason, we refer to these measures as to weakly uniformly non-atomic, see Section \ref{s.ua} for the details. Thus, we have proved the following result.

\begin{theorem}\label{Th0.nona} Let the assumptions of Theorem \ref{Th0.estr} hold and let the measure $\mu$ be weakly uniformly non-atomic. Then the weak uniform attractor $\mathcal A_{un}$ possesses the representation formula \eqref{0.astr}.
\end{theorem}
We would like to recall that the representation formula \eqref{0.astr} is one of the key tools for further study of the attractor (and is crucial for our study of the compactness of weak attractors in stronger topologies, see Section \ref{s.str}). Unfortunately, this formula fails for generic measures $\mu\in M_b(\R,L^2)$ which makes the constructed theory not entirely satisfactory. We expect that the problem may be resolved using the trajectory approach and will return to this question in the forthcoming paper. We also would like to mention that measure driven equations naturally appears in the attractor theory even if we start from the regular external force $\mu\in L^1_b(\R,L^2)$ (the natural class of external forces from the point of view of Strichartz estimates). Indeed,
we cannot guarantee in general that the hull $\mathcal H(\mu)$ will be a subset of $L^1_b(\R,L^2)$ and the appearance of  Borel measures which are not absolutely continuous with respect to the Lebesgue measure in the hull $\mathcal H(\mu)$ looks unavoidable. This is the second source of motivation for us to consider measure driven damped wave equations from the very beginning.
\par
As the next step, we study existence of a uniform attractor for equation \eqref{eq.qdw} in the {\it strong} topology of the energy space $\E$. Clearly, the only assumption $\mu\in M_b(\R,L^2)$ is not enough for this, see examples given in \cite{Zntrc}, so we need to impose some extra conditions for the measure $\mu$ to get this result. In this paper we introduce, following \cite{Zntrc}, two classes of right-hand sides, the so-called space regular and time regular measures. Roughly speaking, these classes consist of measures which can be approximated (in $M_b(\R,L^2)$) by measures which are smooth in space or time respectively, see Definition \ref{Def7.ext}. The intersection of these classes coincide with class of translation compact external forces introduced in \cite{CV}. On the other hand, the following result is verified in Section \ref{s.str}.
\begin{theorem}\label{Th0.sa} Let the assumptions of Theorem \ref{Th0.nona} hold and let, in addition,
the measure $\mu$ be space or time regular. Then there exists a uniform attractor for equation \eqref{eq.qdw} in a strong topology of the energy space $\E$ which coincides with the weak attractor $\mathcal A_{un}$ constructed above.
\end{theorem}
Analogously to \cite{Zntrc}, we utilize the energy equality and the so-called energy method (see also \cite{ball,rosa}) to verify the asymptotic compactness.
\par
Furthermore, we also verify that the uniform attractor is more smooth if the external forces are more smooth. As usual, in order to do so it is enough to verify that $\mathcal A_{un}$ belongs to the higher energy space $\E^\alpha$ for some small positive $\alpha$. The further regularity can be obtained by standard bootstrapping arguments. To get this higher regularity, we follow mainly \cite{ZCPAA2004} and use the following corollary of the Kato-Ponce inequality:
$$
\|f(v+w)-f(v)\|_{H^\alpha}\le C(1+\|v\|_{L^{12}}+\|w\|_{L^{12}})^{4-\alpha}(1+\|v\|_{H^1}+\|w\|_{H^1})^{\alpha}
\|w\|_{H^{1+\alpha}}^{1-\alpha}\|w\|_{H^{\alpha,12}}^{\alpha}
$$
which holds for $\alpha\in[0,\frac25]$, see Section \ref{s.ap}. This allows us to prove (in Section \ref{s.sm}) the following result.
\begin{theorem}\label{Th0.sm} Let the assumptions of Theorem \ref{Th0.estr} holds and let, in addition,
$$
\mu\in M_b(\R,H^\alpha)
$$
for some $\alpha\in(0,\frac25]$. Then the attractor $\mathcal A_{un}$ is a bounded set in the higher energy space $\E^\alpha$. Moreover, the analogous result holds also if $\mu$ is sufficiently smooth in time.
\end{theorem}
Finally, for the convenience of the reader, we collect in Section \ref{s.bv} some standard facts and concepts of the theory of vector valued measures and related functions of bounded variation.

\section{Function spaces and preliminaries}\label{s.p}

In this Section, we introduce some notations which will be used throughout  the paper and state some classical results for the solutions of linear wave equations. We start with functional spaces.
\par
Let $\Omega$ be a domain of $\R^3$ with a smooth boundary. As usual, the Lebesgue spaces of $p$-integrable functions in $\Omega$ are denoted by $L^p(\Omega)$, $1\le p\le\infty$. In the particular case $p=2$ we will use the notation $H:=L^2(\Omega)$. For every $l\in\Bbb N$, we denote by $H^{l,p}(\Omega)=W^{l,p}(\Omega)$ the classical Sobolev space of distributions whose derivative up to order $l$ belong to $L^p(\Omega)$. The closure of $C_0^\infty(\Omega)$ in the space $H^{l,p}(\Omega)$ is denoted by $H^{l,p}_0(\Omega)$. In the case $p=2$, we will write $H^l$ instead of $H^{l,2}$ in order to simplify the notations. The negative Sobolev spaces $H^{-l,p}(\Omega)$ are defined as dual spaces:
$$
H^{-l,p}(\Omega)=\(H^{l,q}_0(\Omega)\)^*,\ \ \frac1p+\frac1q=1.
$$
For the case $l>0$ and $l\notin\N$, we define the fractional space $H^{l,p}(\Omega)$, $1<p<\infty$ as the restriction of the Bessel potentials space $H^{l,p}(\R^3)$ to the domain $\Omega$. We recall that the norm in the space $H^{l,p}(\R^3)$ is defined by
$$
\|u\|_{H^{l,p}}:=\|(1+|\xi|^2)^{l/2}\widehat u(\xi)\|_{L^p},
$$
where $\widehat u$ stands for the Fourier transform of $u$, see e.g., \cite{Tri} for more details. In particular, the fractional Laplacian gives an isomorphism between spaces $H^{l,p}(\R^3)$ and $L^p(\R^3)$:
%$$
\begin{equation}\label{1.i}
(1-\Dx)^{l/2}H^{l,p}(\R^3)=L^p(\R^3).
\end{equation}
%$$
Note that this formula remains true in the spatially periodic case when $\Omega=\T^3$. In the general case where $\Omega$ is a bounded domain some restrictions appear due to the boundary conditions, see \cite{Tri}. We will also widely use in the sequel the classical Sobolev embedding theorem:
$$
H^{\alpha,p}(\T^3)\subset L^q(\T^3),\ \ \frac1q\ge \frac1p-\frac\alpha3,\ \ \alpha\ge0,\ q\ne\infty
$$
and the interpolation inequality:
$$
\|u\|_{H^{\alpha,p}}\le C\|u\|_{H^{\alpha_1,p_1}}^s\|u\|_{H^{\alpha_2,p_2}}^{1-s},
$$
where $\alpha_1,\alpha_2\in\R$, $1<p_1,p_2<\infty$, $s\in[0,1]$ and
$$
\alpha=s\alpha_1+(1-s)\alpha_2,\ \ \frac1p=\frac s{p_1}+\frac{1-s}{p_2}.
$$
We will also need the spaces of functions of mixed space-time regularity. For instance, the natural norms in the spaces $L^p(a,b;\,H^{\alpha,q})$ and $H^{1,p}(a,b; H^{\alpha,q})$ are given by
$$
\|u\|_{L^p(a,b;H^{\alpha,q})}^p:=\int_a^b\|u(t,\cdot)\|_{H^{\alpha,q}}^p\,dt\ \ \text{and }\
\|u\|_{H^{1,p}(a,b;H^{\alpha,q})}^p:=\int_a^b\|u(t,\cdot)\|_{H^{\alpha,q}}^p+\|\Dt u(t,\cdot)\|_{H^{\alpha,q}}^p\,dt
$$
respectively. The index "loc" or "b" will stand for the local or uniformly local topology respectively. For instance,
$$
L^p_{loc}(\R, H^{\alpha,q})=\big\{u:\R\to H^{\alpha,q},\ \|u\|_{L^p(a,b;H^{\alpha,q})}<\infty, \ \forall [a,b]\subset\R\big\}
$$
and
$$
L^p_b(\R, H^{\alpha,q}):=\big\{u\in L^p_{loc}(\R,H^{\alpha,q})\,:\, \|u\|_{L^p_b(\R,H^{\alpha,q})}:=\sup_{a\in\R}\|u\|_{L^p_b(a,a+1;H^{\alpha,q})}<\infty\big\}.
$$
Finally, to treat the external forces, we will need the space $M(a,b;H)$ of vector measures with values in $H$ and with finite total variation and the associated spaces $BV(a,b;H)$ of functions of bounded variation, see Section \ref{s.bv} for more details. Namely, the locally convex space of $H$-valued Borel measures $\nu$ on $\R$ such that the restrictions of $\nu$ to every finite segment $[s,t]$ belong to $M(s,t;H)$ is denoted by $M_{loc}(\R,H)$. Analogously
$$
M_b(\R,H)=\big\{\nu\in M_{loc}(\R,H)\,:\, \|\nu\|_{M_b(\R,H)}:=\sup_{t\in\R}\|\nu\|_{M(t,t+1;H)}<\infty\big\}.
$$
The  spaces  $BV_{loc}(\R,H)$ and $BV_b(\R,H)$ are also defined analogously.
\par
We now recall the standard results about the solutions of the linear wave equation
%$$
\begin{equation}\label{eq.lw}
\Dt^2 v-\Dx v=g(t),\ \ \xi_v\big|_{t=0}=\xi_0
\end{equation}
%$$
in the energy phase spaces
$$
\xi_v(t):=\{v(t),\Dt v(t)\}\in \E^\alpha:=H^{1+\alpha}(\T^3)\times H^\alpha(\T^3).
$$
For simplicity, we state the results for the spatially periodic case $\Omega=\T^3$ although most part of the results stated below remain true for the case of bounded domains as well.

\begin{theorem}\label{th.linstr} Let the initial data $\xi_0\in\mathcal E^\alpha$, $T>0$ and
  $g(t)\in L^1([0,T];H^\alpha)$ for some $\alpha\in\R$. Then, there is a unique solution $\xi_v\in C(0,T;\E^\alpha)$ of problem \eqref{eq.lw}. In addition, the solution $v$
belongs to the space $L^4(0,T;H^{\alpha,12}(\T^3))$  and the following estimate holds:
%$$
\begin{equation}\label{est.strlw.R}
\|\xi_v\|_{C(0,T;\E^\alpha)}+\|v\|_{L^4(0,T;H^{\alpha,12}(\Omega))}\leq C_T\left(\|\xi_0\|_{\E^\alpha}+\|g\|_{L^1(0,T;H^\alpha)}\right),
\end{equation}
%$$
where constant $C_T$ does not depend on $\xi_0\in\E^\alpha$ and $g(t)\in L^1([0,T];H^\alpha)$.
\end{theorem}
The proof of this theorem can be found, e.g., \cite{Chem,sogge,tao}.
\par
To conclude this Section, we state the analogue of the above estimate for the damped linear wave equation:
%$$
\begin{equation}\label{1.damp}
\Dt^2 v+\gamma\Dt v+(1-\Dx)v=g(t),\ \ \xi_v\big|_{t=\tau}=\xi_\tau.
\end{equation}
%$$
where $\gamma>0$ and obtain an estimate which will be crucially used for later in order to obtain the further regularity of uniform attractors.
\begin{cor}\label{Cor1.damp} Let $\xi_\tau\in \E^\alpha$ and $g\in L^1_{loc}(\R,H^\alpha)$ for some $\alpha\in\R$. Then, the solution $\xi_v(t)$ of problem \eqref{1.damp} possesses the following estimate:
%$$
\begin{equation}\label{1.main}
\|\xi_v(t)\|_{\E^\alpha}+\(\int_\tau^t e^{-4\delta(t-s)}\|v(s)\|^4_{H^{\alpha,12}}\,ds\)^{1/4}\le C\|\xi_\tau\|_{\E^{\alpha}}e^{-\delta(t-\tau)}+C\int_\tau^te^{-\delta(t-s)}\|g(s)\|_{H^\alpha}\,ds,
\end{equation}
where the positive constants $C$ and  $\delta=\delta(\gamma)$ are independent of $t\ge\tau$ and $\xi_\tau$ and $g$.
\end{cor}
\begin{proof}
Indeed, due to isomorphism \eqref{1.i}, it is sufficient to verify \eqref{1.main} for $\alpha=0$ only.
For simplicity we also assume that $\tau=0$. Multiplying equation \eqref{1.damp} by $\Dt v+\beta v$, where $\beta>0$ is small enough and arguing in a standard way (see e.g., \cite{CV}), we arrive at
%$$
\begin{equation}\label{1.en}
\|\xi_v(t)\|_{\E}\le C\|\xi_\tau\|_{\E}e^{-\delta t}+C\int_0^te^{-\delta(t-s)}\|g(s)\|_H\,ds
 \end{equation}
%$$
for some positive constants $C$ and $\delta$. After that, we rewrite equation \eqref{1.damp} in the form of equation \eqref{eq.lw} with the right-hand side $\tilde g(t)=g(t)-\gamma\Dt v(t)-v(t)$ and apply estimate \eqref{est.strlw.R} for the Strichartz norm on the time interval $[t,t+1]$, $t\ge0$, to get
%$$
\begin{multline}\label{1.int}
\|v\|_{L^4(t,t+1;L^{12})}\le C\(\sup_{s\in[t,t+1]}\|\xi_v(s)\|_{\mathcal E}+\|g\|_{L^1(t,t+1;H)}\)\le\\\le
C\|\xi_0\|_{\E}e^{-\delta t}+C\int_0^{t+1}e^{-\delta(t-s)}\|g(s)\|_H\,ds.
\end{multline}
%$$
We claim that \eqref{1.int} implies \eqref{1.main} for the Strichartz norm. Indeed, we may assume without loss of generality that $t=n\in\N$ (if this condition is not satisfied, we always can increase $t$ by the proper $\kappa\in(0,1)$  to satisfy this assumption and put $g(s)=0$ for $s\ge t$). In this case, using the concavity of the function $z^{1/4}$  and \eqref{1.int}, we obtain that, for $0<\delta'<\delta$,
%$$
\begin{multline}
\(\int_0^te^{-4\delta'(t-s)}\|v\|^4_{L^{12}}\,ds\)^{1/4}\le C\(\sum_{k=0}^{n-1}e^{-4\delta'(n-k)}\int_{k}^{k+1}\|v(s)\|_{L^{12}}^4\,ds\)^{1/4}\le\\\le C\sum_{k=0}^{n-1}e^{-\delta'(n-k)}\|v\|_{L^4(k,k+1;L^{12})}\le C\sum_{k=0}^{n-1}e^{-\delta'(n-k)}e^{-\delta k}\|\xi_v(0)\|_{\E}+
\\+C\sum_{k=0}^{n-1}e^{-\delta'(n-k)}\int_0^{k+1}e^{-\delta(k+1-s)}\|g(s)\|_H\,ds\le
Ce^{-\delta't}\|\xi_v(0)\|_\E+\\+C\int_0^te^{-\delta'(t-s)}
\sum_{k=1}^{n-1}e^{-(\delta-\delta')|k-s|}\|g(s)\|_H\,ds\le C\|\xi_v(0)\|_{\E} e^{-\delta't}+\\+C(\delta-\delta')^{-1}\int_0^te^{-\delta'(t-s)}\|g(s)\|_H\,ds.
\end{multline}
%$$
Finally, replacing $\delta$ by $\delta'$ we get the desired estimate for the Strichartz norm and finish the proof of the corollary.
\end{proof}
\section{Measure driven damped wave equation: the linear case}\label{s.lwmf}

In this Section we consider the following linear wave equation:
%$$
\begin{equation}\label{lwmf.pr}
%\begin{cases}
\Dt^2 w+\gamma\Dt w+(-\Dx+1)w=\mu(t),\ \ \
\xi_w\big|_{t=\tau}=\xi_\tau:=\{w_\tau,w'_\tau\}
%\end{cases}
\end{equation}
%$$
on a three dimensional torus $x\in\T^3$  where damping parameter $\gamma\geq 0$ and, in contrast to the previous Section, $\mu$ is a {\it measure}. All of the results of this Section are actually valid not only for the case of periodic boundary conditions, but also for the case of Dirichlet or Neumann boundary conditions when $\Omega\subset \R^3$ is a smooth bounded domain (although this result is not necessary for our purposes). We suppose here that
%$$
\begin{equation}
\mu \in M(\tau,T;H),
\end{equation}
%$$
where $M(\tau,T;H)$ is the space of $H$-valued Borel vector measures on $[\tau,T]$ with values in $H$ and with bounded total variation  (see Section \ref{s.bv} for more details).
\par
We start with the definition of an energy solution for equation \eqref{lwmf.pr} which is a bit more delicate since in contrast to the usual case, the time derivative $\Dt w(t)$ may have jumps produced by the atoms of the measure $\mu$.

\begin{Def}\label{def.sol.lwm} A function $w(t)$ such that  $\xi_w(t)\in L^\infty(\tau,T;\E)$ (where $\xi_w(t):=\{w(t),\Dt w(t)\}$) is an  \emph{energy} solution of problem \eqref{lwmf.pr} on $[\tau,T]$ if
\par
1) It satisfies the equation in the sense of distributions, i.e.,
 for any test function $\phi\in C^\infty_0((\tau,T)\times \T^3)$, the following equality holds
%$$
\begin{multline}\label{lwmf.w.sol}
-\int_\tau^T(\Dt w(t),\Dt \phi(t))\,dt+\int_\tau^T(\nabla w(t),\nabla \phi(t))\,dt+\int_\tau^T(w(t),\phi(t))\,dt+\\+
\gamma\int_\tau^T(\Dt w(t),\phi(t))\,dt=\int_\tau^T(\phi(t),\mu(dt)),
\end{multline}
%$$
\par
2) It is left-weakly semicontinuous at every point $t\in[\tau,T]$ as $\mathcal E$-valued function.
\par
3) The initial conditions are satisfied in the following sense:
$$
\xi_w(\tau)=\Big\{w_\tau,w'_\tau\Big\},\ \ \xi_w(\tau+0):={\rm w}\!-\!\lim_{s\to\tau+0}\xi_w(s)=\Big\{w_\tau,w_\tau+\mu(\{\tau\})\Big\}.
$$
\end{Def}

\begin{rem}\label{Rem3.1} Since $w\in L^\infty(\tau,T;H^1)$ and $\Dt w\in L^\infty(\tau,T;H)$, the function $w(t)$ is weakly continuous as a  function with values in $H^1$: $w\in C_w(\tau,T;H^1)$, so the initial data for $w(t)$ is well-defined. The situation with the derivative $\Dt w$ is a bit more delicate since it may be discontinuous. Namely, from Definition \ref{def.sol.lwm} we see that the distributional derivative $\Dt^2w$ satisfies
%$$
\begin{multline}
\langle\Dt^2 w,\phi\rangle=-\langle\Dt w,\Dt\phi\rangle=-\int_\tau^T(\Dt w(t),\Dt \phi(t))\,dt=\\-
\int_\tau^T(\nabla w(t),\nabla \phi(t))\,dt-\int_\tau^T(w(t),\phi(t))\,dt-
\gamma\int_\tau^T(\Dt w(t),\phi(t))\,dt+\int_\tau^T(\phi(t),\mu(dt))
\end{multline}
%$$
and this functional clearly  can be extended by continuity to any $\phi\in L^1(\tau,T;H^1)\cap C_0(\tau,T;H)$. By this reason,
$$
\Dt^2 w\in L^\infty(\tau,T;H^{-1})+ M(\tau,T;H).
$$
 This, together with the fact that $\Dt w\in L^\infty(\tau,T;H)$, implies
 $$
 \Dt w\in C_w(\tau,T;H)+BV(\tau,T;H).
 $$
Since any BV function has left and right limits at every point (see Section \ref{s.bv}), the function $t\to\Dt w(t)$ also possesses left and right limits $\Dt w(t+0)$ and $\Dt w(t-0)$ at any point $t\in(\tau,T)$ (in a weak topology of $H:=L^2$) as well as the limits $\Dt w(\tau+0)$ and $\Dt w(T-0)$. Thus, assumption 2) of the definition makes sense and the second part of the initial conditions 3)  for $\Dt w$ is also well-defined. However, since $C_0(\tau,T;H)$ is not dense in $C(\tau,T;H)$ the values $\Dt w(\tau)$ and $\Dt w(T)$ remain undefined (as well as the values of $\Dt w$ at jump points).
\par
In order to avoid this ambiguity and to be able to define the dynamical process associated with our problem (see Section \ref{s.ua}), we choose weakly-left semicontinuous representative on $[\tau,T]$ from the class of equivalence of $\Dt w$ by default. Then, the value $\Dt w(T)=\Dt w(T-0)$ is also well-defined and the value $\Dt w(\tau)$ is determined by the first part of initial conditions.
\end{rem}

\begin{rem}\label{Rem3.2} Note that  energy solution $w$ of problem \eqref{lwmf.pr} possesses the following property:
%$$
\begin{equation}\label{lwmf.j}
\Dt w(t+0)-\Dt w(t)=\mu(\{t\}),\ \ t\in[\tau,T]
\end{equation}
%$$
(in the case $t=T$ we just assume that $\xi_w(T+0):=\{w(T),\Dt w(T)+\mu(\{T\})\}$).
Indeed, integrating by parts in \eqref{lwmf.w.sol} and  using \eqref{bv.ByParts} to handle with the most complicated term which involves measures, we get
%$$
\begin{equation}
\int_\tau^T\left(W(t),\Dt \phi(t)\right)\,dt=0,\ \text{for all }\phi\in C^\infty_0((\tau,T)\times\T^3),
\end{equation}
%$$
where $W(t):=-\Dt w(t)-\int_\tau^t(-\Dx+1)w(s)\,ds-\gamma w(t)+\mu([\tau,t))$. Therefore, $W(t)=\Psi$ almost everywhere for some $\Psi\in H^{-1}$. Using now the assumption that $\Dt w$ is left-continuous together with the obvious fact that $t\to\mu([\tau,t))$ is also left-semicontinuous and taking into the account the  initial data, we conclude that
%$$
\begin{equation}\label{eq-sol}
\Dt w(t)=-\int_\tau^t(-\Dx+1)w(s)\,ds-\gamma w(t)+\mu([\tau,t))+w_\tau'+\gamma w_\tau,\ t\in[\tau,T].
\end{equation}
%$$
The desired formula \eqref{lwmf.j} is an immediate corollary of \eqref{eq-sol}.
\par
The proved formula shows, in particular, that the function $\Dt w$ will be weak-continuous as a function with values in $H$ if the measure $\mu$ is non-atomic. Moreover, multiplying equation \eqref{eq-sol} by $\Dt\phi$ integrating over $t\in\R$, performing the integration by parts back and using the initial conditions, we  return to the distributional formulation \eqref{lwmf.w.sol}. Thus, identities \eqref{lwmf.w.sol} and \eqref{eq-sol} are equivalent and we may check \eqref{eq-sol} instead of \eqref{lwmf.w.sol}. We will essentially use this observation later.
\end{rem}
At the next step we write out the explicit formula for the solution of equation \eqref{lwmf.pr}. We start with the homogeneous case $\mu=0$. Then, the solution $w(t)$ is given by
%$$
\begin{equation}
w(t)=C_{A}(t-\tau)w_\tau+S_{A}(t-\tau)w'_\tau,
\end{equation}
%$$
where $A:=-\Dx+1$ endowed with periodic boundary conditions,
$$
S_A(t):=e^{-\frac{\gamma}{2}t}\frac{\sin(\Lambda(A) t)}{\Lambda(A)},\ \ \ C_A(t):=e^{-\frac{\gamma}{2}t}\cos(\Lambda(A) t)
$$
and $\Lambda(z):=\(z-\frac{\gamma^2}4\)^\frac12$. The corresponding solution semigroup in the energy phase space $\mathcal E$ is then defined via
%$$
\begin{equation}\label{2.sem}
\xi_w(t)=\mathcal S_A(t)\xi_w(0),\ \ \ \mathcal S_A(t):=\(\begin{matrix} C_A(t), & S_A(t)\\ \Dt C_A(t),&\Dt S_A(t)\end{matrix}\).
\end{equation}
%$$
The following result is well-known and can be verified by straightforward calculations.

\begin{lemma}\label{Lem2.con} The operators $\mathcal S_A(t)$ are bounded in $\mathcal E$ and satisfy the following estimate
%$$
\begin{equation}\label{2.contr}
\|\mathcal S_A(t)\|_{\mathcal L(\mathcal E,\mathcal E)}\le Ce^{-\gamma_0t},
\end{equation}
%$$
where the constant $C$ may depend on $\gamma$ and $\gamma_0:=\min\{\gamma/2,1\}$.
\end{lemma}
Furthermore, in the regular case where the measure $\mu$ is absolutely continuous ($\mu(dt)=g(t)dt$ for some $g\in L^1(\tau,T;H)$), the solution of the non-homogeneous equation is given by the Duhamel formula:
$$
\xi_w(t)=\mathcal S_A(t)\xi_w(\tau)+\int_\tau^t\mathcal S_A(t-s)\(\begin{matrix}0\\g(s)\end{matrix}\)\,ds.
$$
The next theorem shows that the analogue of this formula holds in a general case as well.

\begin{theorem}\label{lwmf.thE!}
Let  $\gamma\geq 0$,  $\xi_\tau\in\E$ and the external force $\mu\in M(\tau,T;H)$. Then problem \eqref{lwmf.pr} possesses a unique energy solution $w$ on $[\tau,T]$. This solution satisfies
%$$
\begin{equation}\label{lwmf.DuhM}
\xi_w(t)=\mathcal S_A(t)\xi_\tau+
\int_{s\in[\tau,t)}\mathcal S_{A}(t-s)\(\begin{matrix}0\\\rho_{\mu}(s)\end{matrix}\)|\mu|(ds),\ t\in [\tau,T],
\end{equation}
%$$
where $\rho_\mu \in L^1_{|\mu|}(\tau,T;H)$ is the density of $\mu$ with respect to $|\mu|$ (see \eqref{mu=rho|mu|}).
\par
Furthermore, the following energy estimate holds:
%$$
\begin{equation}\label{lwmf.E-est}
\|\xi_w(t)\|_\E\leq C\left(\|\xi_\tau\|_\E e^{-\gamma_0(t-\tau)}+\int_\tau^te^{-\gamma_0(t-s)}|\mu|(ds)\right),\quad t\in[\tau,T],
\end{equation}
%$$
for some constant $C$ depending only on $\gamma$.
\end{theorem}

\begin{proof} We first note that due to Lemma \ref{Lem2.con}, the function $\xi_w(t)$ is well-defined and belongs to $L^\infty(\tau,t;\E)$ and satisfies  energy inequality \eqref{lwmf.E-est} (here we have implicitly used that $\|\rho_\mu(t)\|_H=1$). The weak left-continuity of $\xi_w(t)$ as well as the fact that it satisfies the initial data also an immediate corollary of formula \eqref{lwmf.DuhM}.
\par
In order to check that it satisfies the equation in the sense of distributions, we expand $\xi_w(t)$ into the Fourier series associated with the eigenfunctions of the operator $A$. Namely,
let $e_i$, $\lambda_i$ be the eigenvectors and the eigenvalues (enumerated in the non-decreasing order) of the operator $A$ and let $P_N$ be the orthoprojector to the linear subspace generated by the first $N$ eigenvectors. We also denote $Q_N:=1-P_N$. Then,
$$
\xi_w(t)=\xi_{P_Nw}(t)+\xi_{Q_N w}(t)
$$
and, due to Lemma \ref{bv.lem.|QNmu|to0} and estimate \eqref{lwmf.E-est},
%$$
\begin{equation}
\lim_{N\to\infty}\|\xi_{Q_Nw}\|_{L^\infty(\tau,T;H)}=0.
\end{equation}
%$$
Thus, it is enough to verify that, for every $N\in\mathbb N$, the function $w_N(t):=P_Nw(t)$ is a distributional solution of an ODE
$$
\Dt^2w_N+\gamma\Dt w_N+Aw_N=P_N\mu.
$$
But this can be done in a straightforward way using the integration by parts
formula   \eqref{bv.ByParts} (with $H=\R^N$) and the properties of the Duhamel integral (we leave the rigorous proof of this to the reader). Thus, the function $\xi_w(t)$ is indeed the desired energy solution.
\par
Finally, let $w_1(t)$ and $w_2(t)$ be two energy solutions. Then, since both of these functions are weakly continuous in $H^1$, their derivatives $\Dt w_i(t)$ are weakly left-continuous and have the same jumps according to formula \eqref{lwmf.j}, we conclude that $\xi_w(t)$ is weak-continuous in $\mathcal E$ where $w(t)=w_1(t)-w_2(t)$. In addition, $w(t)$ solves the homogeneous problem \eqref{lwmf.pr} with $\mu=0$ and zero initial data. It is well-known that such solution is unique, so $w\equiv0$ and the uniqueness is also verified and the theorem is proved.
\end{proof}

\begin{cor}\label{lwmf.cor.cont}
Let assumptions of Theorem \ref{lwmf.thE!} holds then the energy solution $w\in C(\tau,T;H^1)$ and $\Dt w\in C([\tau,T]\setminus \supp\mu_d;H)$, where $\supp\mu_d$ is the support of the discrete part of the measure $\mu$ or equivalently the set of points of discontinuities for $\Phi_\mu(t)$. Moreover, the limits $\xi_w(t+0)$ and $\xi_w(t-0)$ both exist for every $t\in[\tau,T]$ in a strong topology of $\mathcal E$.
\end{cor}

Indeed, this follows immediately from the analogous statement for the finite-dimensional part $\xi_{P_Nw}(t)$ and from the uniform smallness of the function $\xi_{Q_Nw}(t)$ proved in the theorem.

\begin{cor}\label{Cor2.energy} Assume that, in addition, the measure $\mu$ is non-atomic ($\mu(\{t\})=0$ for all $t$). Then, the solution $\xi_w\in C(\tau,T;\mathcal E)$. Moreover, the energy equality holds:
%$$
\begin{equation}\label{2.ener}
\frac12\(\|\xi_w(t)\|_{\mathcal E}^2-\|\xi_w(s)\|_{\mathcal E}^2\)=-\gamma\int_s^t\|\Dt w(\kappa)\|^2_H\,d\kappa+\int_s^t(\Dt w(\kappa), \mu(d\kappa)),
\end{equation}
%$$
for all $[s,t]\subset[\tau,T]$ (Since for non-atomic measures the integrals over $[\tau,t)$ and $[\tau,t]$ coincide we write here and below $\int_\tau^t$ instead of $\int_{\kappa\in[\tau,t)}$).
\end{cor}
Indeed, as usual, identity \eqref{2.ener} is proved first for the finite-dimensional function $\xi_{w_N}(t)$, where it is standard since the function $\xi_{w_N}(t)$ is continuous in time and therefore can be approximated by smooth functions. Then, passing to the limit $N\to\infty$, one gets the desired energy equality for the infinite-dimensional case as well (using the fact that $\xi_{Q_Nw}$ is uniformly small).
\par
\begin{rem} Identity \eqref{2.ener} can be rewritten in the following way:
%$$
\begin{equation}\label{2.eener}
\frac d{dt}\(\frac12\|\xi_w(t)\|^2_{\mathcal E}-\int_\tau^t(\Dt w(s),\mu(ds))\)=-\gamma\|\Dt w(t)\|^2_H.
\end{equation}
In particular, the function $\frac12\|\xi_w(t)\|^2_{\mathcal E}-\int_\tau^t(\Dt w(s),\mu(ds))$ is absolutely continuous in time. However, the energy $\|\xi_w(t)\|^2_{\mathcal E}$ is not absolutely continuous since the singular part of the measure $\mu$ is not assumed to vanish.
\par
The analogue of this formula can be written in the general case, where the discrete part of the measure $\mu$ does not vanish. However, in this case, one should be careful with the integral $\int_\tau^t(\Dt w(\kappa),\mu_d(d\kappa))$ since the function $\Dt w$ has jumps {\it exactly} at the points where $\Phi_\mu(t)$ is discontinuous.
 Moreover, since according to \eqref{eq-sol}, the function $\Dt w-\mu([\tau,t))$ is continuous, the only problematic term is
$\int_{[\tau,t)}(\mu_d([\tau,s)),\mu_d(ds))$. This integral makes sense as a Lebesgue-Stiltjes integral. But the value of the integral thus defined is {\it inconsistent} with the energy identity. Indeed, in our case $\Dt w$ is left-semicontinuous at jump points $t=t_j$ and therefore
$$
\int_{[\tau,t)}(\Dt w(s),\mu_d(ds)))=\sum_{j}\(\Dt w(t_j-0),\mu(\{t_j\}\)_H.
$$
However, arguing in a bit more accurate way (e.g., approximating $\mu_d$ by smooth functions or comparing the values of the energy functional before and after a jump), we see that the correct formula must be
$$
\int_{[\tau,t)}(\Dt w(t),\mu_d(ds)):=\sum_{j}\(\frac12(\Dt w(t_j+0)+\Dt w(t_j-0)),\mu(\{t_j\})\)_H
$$
which corresponds to the choice $\Dt w(t_j):=\frac{\Dt w(t_j+0)+\Dt w(t_j-0)}2$ (see also \cite{Hew}). This gives the following natural interpretation of the problematic integral:
$$
\int_{[\tau,t)}(\Dt w(s),\mu(ds)):=\int_{[\tau,t)}(\Dt w(s)-\mu([\tau,s)),\mu(ds))+\frac12\|\mu([\tau,t))\|_H^2
$$
which is consistent with the energy equality.
 We will return to this in the forthcoming paper.
\end{rem}

We conclude this Section by establishing the Strichartz type estimates for the measure driven wave equation using the approximations of the measure by absolutely continuous ones.

\begin{theorem}\label{lwmf.th.str}
Let  $\gamma\ge0$, the initial data $\xi_\tau\in\E$ and the external force $\mu\in M(\tau,T;H)$. Then the energy solution $w$ to problem \eqref{lwmf.pr} obeys the estimate
%$$
\begin{equation}\label{2.str}
\|w\|_{L^4(\tau,T;L^{12})}\leq C\left(\|\xi_\tau\|_\E+\|\mu\|_{M(\tau,T;H)}\right),
\end{equation}
%$$
where the constant $C$ depends on $\gamma$ and $T-\tau$ but  is independent of $\xi_\tau$ and $\mu$.
\end{theorem}

\begin{proof}
Let $\Phi_\mu(t)$ be the distribution function of $\mu$ given by \eqref{bv.Fmu}. Let $\Phi_n(t)$ denote the smooth approximations of $\Phi_\mu$ constructed  as in Proposition \ref{prop.[L1]*=M} and let us consider the following approximation sequence $w_n$
%$$
\begin{equation}\label{2.app}\begin{cases}
\Dt^2 w_n+\gamma\Dt w-\Dx w_n+\gamma\Dt w_n=\Phi'_n(t),\ t\in[\tau,T],\\
\xi_{w_n}|_{t=\tau}=\xi_\tau.
\end{cases}
\end{equation}
%$$
We note that by construction (see Proposition \ref{prop.[L1]*=M}) we have
%$$
\begin{align}\label{lwmf.MuntoMu}
&\mu_{\Phi_n}\to \mu\  \text{ as }n \to\infty\ \mbox{wealky-star in }M(\tau,T;H),\\
\label{lwmf.Phi'n<|mu|}
&\|\Phi'_n\|_{L^1(\tau,T;H)}\leq \|\mu\|_{M(\tau,T;H)}.
\end{align}
%$$
and, in addition, $\mu_{\Phi_n}([\tau,t))\to\mu([\tau,t))$ for all $t\in[\tau,T]$, see Remark \ref{Cor9.conv}.
Using the standard energy estimate and \eqref{lwmf.Phi'n<|mu|} we see that
%$$
\begin{equation}
\|\xi_{w_n}\|_{L^\infty(\tau,T;\E)}\leq C\left(\|\xi_\tau\|_\E+\|\Phi'_n\|_{L^1(\tau,T;H)})\leq C(\|\xi_\tau\|_\E+\|\mu\|_{M(\tau,T;H)}\right),\ \forall n\in\N.
\end{equation}
%$$
The last estimate together with \eqref{lwmf.MuntoMu} implies that $\xi_{w_n}$ converges to some $\xi_{\bar w}\in L^\infty(\tau,T;\E)$ as $n$ goes to infinity weakly-star in $L^\infty(\tau,T;\E)$. We need to show that $\bar w$ is an energy solution for
problem \eqref{lwmf.pr}. Indeed, arguing in a standard way, we see that $w_n\to \bar w$ strongly in $C(\tau,T;H)$ and, therefore, $\bar w$ is weakly continuous in $H^1$ and $\bar w(\tau)=w_\tau$.
\par
To verify that $\bar w$ is an energy solution, it is enough to pass to point-wise limit at
%$$
\begin{equation}\label{eq-soln}
\Dt w_n(t)=-\int_\tau^t(-\Dx+1)w_n(s)\,ds-\gamma w_n(t)+\mu_{\Phi_n}([\tau,t))+w_\tau'+\gamma w_\tau,\ t\in[\tau,T]
\end{equation}
%$$
and get \eqref{eq-sol}. Thus, $\bar w$ is an energy solution of \eqref{lwmf.pr} and, by the uniqueness, $\bar w=w$.
\par
To obtain the desired Strichartz estimate, we apply Theorem \ref{th.linstr} to equation \eqref{2.app} and get
%$$
\begin{equation}
\|w_n\|_{L^4(\tau,T;L^{12})}\le\left(\|\xi_\tau\|_\E+\|\Phi'_n\|_{L^1(\tau,T;H)}\right)\le C\left(\|\xi_\tau\|_\E+\|\mu\|_{M(\tau,T;H)}\right),\ \forall n\in\N.
\end{equation}
%$$
The last estimate allows us to assume without loss of generality that $w_n$ converges to $w$ as $n\to\infty$  weakly in $L^4(\tau,T;L^{12})$. Weak lower semicontinuity of the norm implies the desired estimate \eqref{2.str} and finishes the proof of the theorem.
\end{proof}
\begin{rem} Since the energy estimate gives us the control of the norm of $w$ in $L^\infty(\tau,T;L^6)$, we can replace the $L^4(L^{12})$-norm in the left-hand side of \eqref{2.str} by any intermediate Strichartz norm, for instance, by the $L^5(L^{10})$ norm.
\end{rem}

\section{The quintic wave equation: well-posedness and dissipativity in the energy norm}\label{s.new1}
In this Section, we  discuss the properties of solutions for our main object of study - the damped quintic wave equation:
%$$
\begin{equation}\label{4.qwmu}
\Dt^2u+\gamma\Dt u+(1-\Dx) u+f(u)=\mu,\ \ \xi_u\big|_{t=\tau}=\xi_\tau
\end{equation}
%$$
on the 3D torus $\Omega=\mathbb T^3$. Since the results presented below are either well-known or straightforward adaptations of well-known results to the case of measure external forces, we restrict ourselves by giving only the brief exposition (more details can be found in \cite{CV,KSZ,plan1,plan2}).
\par
 We assume that $\xi_\tau\in\mathcal E$, $\mu\in M_b(\R,H)$ and the non-linearity $f\in C^2(\R)$ has the following structure:
%$$
\begin{equation}\label{4.f}
f(u)=u^5+h(u),\ \ |h''(u)|\le C(1+|u|^{q}),\ \ q<3,\ \ h(0)=0.
\end{equation}
%$$
We start our exposition by giving the analogue of Definition \ref{def.sol.lwm} of  a weak solution for the non-linear case.
\begin{Def}\label{Def4.solen} A function $u(t)$ such that  $\xi_u(t)\in L^\infty(\tau,T;\E)$ (where $\xi_u(t):=\{u(t),\Dt u(t)\}$) is an  \emph{energy} solution of problem \eqref{4.qwmu} on $[\tau,T]$ if
\par
1) It satisfies the equation in the sense of distributions, i.e.,
 for any test function $\phi\in C^\infty_0((\tau,T)\times \Omega)$, the following equality holds
%$$
\begin{multline}\label{lwmf.w.sol1}
-\int_\tau^T(\Dt u(t),\Dt \phi(t))\,dt+\int_\tau^T(\nabla u(t),\nabla \phi(t))\,dt+\int_\tau^T(u(t),\phi(t))\,dt+\\+
\gamma\int_\tau^T(\Dt w(t),\phi(t))\,dt+\int_\tau^T(f(u(t)),\phi(t))\,dt=\int_\tau^T(\phi(t),\mu(dt)).
\end{multline}
%$$
\par
2) It is left-weakly semicontinuous at every point $t\in[\tau,T]$ as $\mathcal E$-valued function.
\par
3) The initial conditions are satisfied in the following sense:
$$
\xi_u(\tau)=\{u_\tau,u'_\tau\},\ \ \xi_u(\tau+0):=\operatorname{w-lim}_{t\to\tau+}\xi_u(t)=\Big\{u_\tau,u'_\tau+\mu(\{\tau\})\Big\}.
$$
\end{Def}
Analogously to the linear case (see Remarks \ref{Rem3.1} and \ref{Rem3.2}), we may conclude that
$$
\Dt u(t+0)-\Dt u(t)=\mu(\{t\}),\ \ t\in[\tau,T]
$$
and, in particular, the difference $\xi_{u_1-u_2}(t)$ between two energy solutions of \eqref{4.qwmu} (corresponding to the same $\mu$) belongs to $C_w(\tau,T,\E)$. In addition, exactly as in the linear case, equality \eqref{lwmf.w.sol1} is equivalent to
%$$
\begin{equation}\label{4.intdef}
\Dt u(t)=-\int_\tau^t(-\Dx+1)u(s)+f(u(s))\,ds-\gamma u(t)+\mu([\tau,t))+u_\tau'+\gamma u_\tau,\ t\in[\tau,T].
\end{equation}
%$$
 The presence of the non-linear term $f(u)$ does not make any difference here since due to the embedding theorem $H^1\subset L^6$, $f(u)\in L^\infty(\tau,T;H^{-1})$.
\par
The next standard theorem gives the solvability of equation \eqref{4.qwmu} in the class of energy solutions.
\begin{theorem}\label{th.ex1} Let $\xi_\tau\in\mathcal E$, $\mu\in M_b(\R,H)$ and the non-linearity $f$ satisfies \eqref{4.f}. Then, for every $T>\tau$, there exists at least one energy solution $u(t)$ in the sense of the above definition which satisfies the following estimate:
%$$
\begin{equation}\label{est.disen}
\|\xi_u(t)\|_\E\leq e^{-\beta(t-\tau)}Q(\|\xi_\tau\|_\E)+Q(\|\mu\|_{M_b(\tau,T;H)}),\ t\geq\tau,
\end{equation}
where monotone increasing function $Q$, constant $\beta>0$ are independent of $\tau\in\R$, $T$, $\mu$ and $\xi_\tau\in\E$.
%$$
\end{theorem}
\begin{proof} Indeed, let us start with the case where the measure $\mu$ is regular, i.e., $\mu\in L^1(\tau,T;H)$. In this case, the assertion of the theorem is standard: the existence of a solution is obtained, e.g., using the Galerkin approximations, the uniform estimate energy estimate for Galerkin approximations can be deduced just by multiplying the equation by $\Dt u+\beta u$ for some positive $\beta$, and the validity of the energy estimate for the solution of \eqref{4.qwmu} is then established by passing to the limit in the Galerkin approximations, see \cite{CV} and references therein for the details.
\par
Let us now consider the general case where the measure $\mu$ may be singular. In this case, we approximate $\mu$ by regular measures $\mu_n$ using the special approximations constructed in Proposition \ref{prop.[L1]*=M} (see also Remark \ref{Cor9.conv}). Namely, the sequence $\mu_n$ is uniformly bounded in $L^1(\tau,T;L^2)$, weakly star convergent to the measure $\mu$ in $M(\tau,T;H)$ and  $\mu_n([\tau,t))$ converge to $\mu([\tau,t))$ for every $t\in[\tau,T]$. Let $u_n(t)$ be an energy solution of \eqref{4.qwmu} where $\mu$ is replaced by $\mu_n$. Then, due to the uniform energy estimate, we may assume without loss of generality that
$$
\xi_{u_n}\to \xi_{u}\ \text{weakly star in } L^\infty(\tau,T;\E).
$$
Due to the compactness of the embedding
$$
L^\infty(\tau,T;H^1)\cap W^{1,\infty}(\tau,T;H)\subset C(\tau,T;H),
$$
we conclude that $u_n\to u$ strongly in $C(\tau,T;H)$ and, therefore, almost everywhere in $[\tau,T]\times\Omega$. Moreover, since $f(u_n)$ is uniformly bounded in $L^{6/5}((\tau,T)\times\Omega)$, the convergence almost everywhere implies that
$$
f(u_n)\to f(u) \ \text{weakly in}\ L^{6/5}((\tau,T)\times\Omega).
$$
The established convergence allows us to pass to the limit $n\to\infty$ in the equations \ref{lwmf.w.sol1} and establish that the limit function $u$ solves equation \eqref{4.qwmu} in the sense of distributions.
\par
Finally, in order to verify the left-semicontinuity (find the proper representative in the class of equivalence), it is sufficient to pass to the point-wise limit in equation \eqref{4.intdef} for solutions $u_n$ and the theorem is proved.
\end{proof}
The existence of weak energy solutions can be proved analogously not only for quintic non-linearities. The only difference is that the energy space should be properly corrected. Namely, if the non-linearity grows as $u|u|^q$ where $q>5$, one should take $\E:=[H^1_0\cap L^{q+2}]\cap L^2$ as the energy space, see \cite{CV} for details. However, to the best of our knowledge, the uniqueness of such solution is known only if $q\le2$. Moreover, for the quintic case $q=4$, we do not know also whether or not any energy solution satisfies the energy estimate \eqref{est.disen}. In order to overcome this problem, we introduce (following \cite{BG1999,SS1,SS}) the so-called Shatah-Struwe (SS) solutions and utilize the Strichartz estimates.

\begin{Def}\label{Def4.SS} An energy solution $u(t)$ is a Shatah-Struwe solution of problem \eqref{4.qwmu} if, in addition,
%$$
\begin{equation}\label{4.strass}
u\in L^4(\tau,T;L^{12}(\Omega)).
\end{equation}
%$$
We note that, due to Theorem \ref{lwmf.th.str} and the fact that
$$
\|f(u)\|_{L^1(\tau,T;H)}\le C(1+\|u\|^4_{L^4(\tau,T;L^{12})})\|u\|_{L^\infty(\tau,T;H^1)},
$$
we see that, for any Shatah-Struwe solution $u$, $f(u)\in L^1(\tau,T;H)$.
\end{Def}
The next theorem establishes the uniqueness of such solutions.
\begin{theorem}\label{Th4.uni} Let $u_1$ and $u_2$ be two Shatah-Struwe solutions of problem \eqref{4.qwmu} which correspond to different initial data and the same $\mu\in M(\tau,T,H)$. Then, the following estimate holds:
%$$
\begin{equation}\label{4.uni}
\|\xi_{u_1}(t)-\xi_{u_2}(t)\|_{\E}\le e^{C\int_\tau^t(1+\|u_1(s)\|^4_{L^{12}}+\|u_2(s)\|^4_{L^{12}})\,ds}\|\xi_{u_1}(0)-\xi_{u_2}(0)\|_{\E},
\end{equation}
%$$
where $\tau\le t\le T$ and the constant $C$ is independent on $u_1$ and $u_2$. In particular, the Shatah-Struwe solution is unique.
\end{theorem}
\begin{proof}Indeed, let $v(t)=u_1(t)-u_2(t)$. Then, the function $\xi_v(t)$ is weakly-continuous in $\E$ since $\Dt u_1$ and $\Dt u_2$ have the same jumps determined by the discrete part $\mu_d$. This function solves the equation
$$
\Dt^2v+\gamma\Dt v+(1-\Dx) v+v+[f(u_1)-f(u_2)]=0.
$$
Since $f(u_1)-f(u_2)\in L^1([\tau,T],L^2)$ and $\Dt v\in L^\infty([\tau,T],L^2)$, multiplication on $\Dt v$ can be justified in a standard way and gives
$$
\frac12\frac d{dt}\|\xi_v\|^2_{\E}=-(f(u_1)-f(u_2),\Dt v).
$$
Moreover, using again the fact that $|f'(u)|\le C(1+|u|^4)$, the H\"older inequality and the embedding $H^1\subset L^6$, we get
$$
|(f(u_1)-f(u_2),\Dt v)|\le C(1+|u_1|^4+|u_2|^4)|v|,|\Dt v|)\le C(1+\|u_1\|_{L^{12}}^4+\|u_1\|_{L^{12}}^4)\|v\|_{H^1}\|\Dt v\|_{L^2}
$$
and the Gronwall inequality finishes the proof of the theorem.
\end{proof}
The next corollary is crucial for our proof of asymptotic compactness.
\begin{cor}\label{Cor4.energy} Let the assumptions of Theorem \ref{th.ex1} and let, in addition, the measure $\mu$ be non-atomic (i.e., $\mu(\{t\})=0$ for all $t$). Then, for every Shatah-Struwe solution $u$, the energy functional  $t\to \frac12\|\xi_u(t)\|^2_{\E}+(F(u(t)),1)$ is a continuous BV function of time  and the following energy equality holds for all $\tau\le s\le t\le T$:
%$$
\begin{equation}\label{4.energy}
\frac12\|\xi_u(t)\|^2_{\E}+(F(u(t)),1)-\frac12\|\xi_u(s)\|^2_{\E}-(F(u(s)),1)+\gamma\int_s^t\|\Dt u(\kappa)\|^2_{L^2}\,d\kappa=\int_s^t(\Dt u(\kappa),\mu(d\kappa)).
\end{equation}
In particular, $\xi_u\in C(\tau,T;\E)$.
\end{cor}
\begin{proof} Indeed, since $u\in L^5(\tau,T;L^{10})$, the term $f(u)\in L^1(\tau,T;H)$ can be treated as a regular measure. Thus, according to Corollary \ref{Cor2.energy}, we may write
$$
\frac12\(\|\xi_u(t)\|^2_\E-\|\xi_u(s)\|^2_\E\)+\int_s^t(f(u(\kappa)),\Dt u(\kappa))\,d\kappa+\gamma\int_s^t\|\Dt u(\kappa)\|^2_{L^2}\,d\kappa=\int_s^t(\Dt u(\kappa),\mu(d\kappa)).
$$
Since $f(u)\Dt u\in L^1((\tau,T)\times\Omega)$, the term involving the non-linearity is well-defined. Moreover, arguing in a standard way, we get that the function $t\to (F(u(t)),1)$ is absolutely continuous and
$$
(F(u(t)),1)-(F(u(s)),1)=\int_s^t(f(u(\kappa)),\Dt u(\kappa))\,d\kappa.
$$
Thus, the energy equality is proved. The fact that the energy functional is continuous and BV in time follows immediately from this equality. Finally, the fact that $\xi_u\in C(\tau,T;\E)$ follows from the energy equality in a straightforward way using the  energy method. Thus, the corollary is proved.
\end{proof}
We now discuss the existence of Shatah-Struwe solutions.
\begin{prop}\label{Prop.nc} Let the assumptions of Theorem \ref{th.ex1} be satisfied. Then, for every $\xi_u(\tau)\in\mathcal E$, there exists a unique global Shatah-Struwe solution $u(t)$ and this solution satisfies the energy dissipative estimate \eqref{est.disen}.
\end{prop}
The proof of the existence is standard, see \cite{sogge, plan1,plan2} for the details. First, based on the Strichartz estimate \eqref{2.str} for the linear equations and treating the non-linearity as a perturbation, one establishes the {\it local} existence. Then, using the so-called Pohozhaev-Morawetz identity and non-concentration arguments, one establishes that the Strichartz norm cannot blow up and this gives global existence. The presence of the measure $\mu$ in the right-hand side does not produce any essential difficulties as not difficult to check. We will not give a detailed proof here since in the next Section, we give an alternative proof and estimate the Strichartz norm without using the non-concentration arguments.

\section{Quintic wave equation: energy to Strichartz estimates}
\label{s.strna}
As we have already mentioned, the global existence result for Shatah-Struwe solutions based on the non-concentration arguments (and stated in Proposition \ref{Prop.nc}) does not give any control of the Strichartz norm $\|u\|_{L^4(T,T+1;L^{12})}$ in terms of $T$ and the corresponding norms of the initial data and the external forces. In particular, we do not have any control of the behaviour of this norm as $T\to\infty$ which in turn leads to essential problems in the attractor theory, see \cite{KSZ} for the details. The aim of this Section is to estimate this Strichartz norm in terms of the energy norm and the proper norm of the external forces. Since we have already known the dissipative estimate for the energy norm, this result will give us the desired dissipative estimate for the Strichartz norm. Our approach is crucially based on the following result for homogeneous quintic wave equation in the whole space $\R^3$.
\begin{prop}\label{Prop.tao} There exists a monotone increasing function $Q:\R_+\to\R_+$ such that any Shatah-Struwe solution $v(t)$ of the quintic wave equation
%$$
\begin{equation}\label{eq.q}
\Dt v-\Dx v+v^5=0
\end{equation}
in the whole space $\Omega=\R^3$
satisfies the estimate
%$$
\begin{equation}\label{est-str}
\|v\|_{L^4(\R,L^{12}(\R^3))}\le Q\(\|\Dt v(0)\|_{L^2(\R^3)}^2+\|\Nx v(0)\|^2_{L^2(\R^3)}\).
\end{equation}
%$$
\end{prop}
The proof of this estimate can be found in \cite{BG1999} (see also \cite{tao1} for the explicit expression of the function~$Q$).
\par
Clearly, estimate \eqref{est-str} on the whole line $t\in\R$ cannot hold in the case where $\Omega$ is a bounded domain. However, its finite time analogue remains true in the case where $\Omega=\mathbb T^3$.
\begin{cor}\label{Cor.tor} There exists a monotone increasing function $Q:\R_+\to\R_+$ such that any Shatah-Struwe solution $v$ of  quintic wave equation \eqref{eq.q} with periodic boundary conditions satisfies the estimate
%$$
\begin{equation}\label{est-str-tor}
\|v\|_{L^4(0,1;L^{12})}\le Q\(\|\xi_v(0)\|_{\mathcal E}\).
\end{equation}
%$$
\end{cor}
Indeed, this estimate follows immediately from \eqref{est-str} and finite speed propagation result for wave equations, see \cite{sogge}. To the best of our knowledge, the question of validity of \eqref{est-str-tor} for the case of general bounded domains remains open.
\par
We are now ready to state the key result of this Section.
\begin{theorem}\label{th.str-f1}
Let $\Omega=\Bbb T^3$, non-linearity $f$ satisfy \eqref{4.f} and the external force $\mu\in M(0,1;H)$. Then the Shatah-Struwe solution $u$ of problem
%$$
\begin{equation}\label{nhqwv}
%\begin{cases}
\Dt^2u-\Dx u+f(u)=\mu,\ \ \xi_u\big|_{t=0}=\xi_0
%\end{cases}
\end{equation}
%$$
satisfies the following estimate:
%$$
\begin{equation}\label{L5L10.nh}
\|\xi_u\|_{L^\infty(0,1;\E)}+\|u\|_{L^4(0,1;L^{12})}\leq Q(\|\xi_0\|_{\E})+Q(\|\mu\|_{M(0,1;H)}),
\end{equation}
%$$
where monotone nondecreasing function $Q$ is independent of the choice of initial data $\xi_0\in\E$ and $\mu\in M(0,1;H)$.
\end{theorem}
\begin{proof} Let us first suppose that $f(u)=u^5$ (i.e., $h(u)=0$). The general case $h\ne0$ will be considered later.
 To verify the desired estimate, we consider an approximating sequence $\{u_N\}_{N=1}^\infty$ to the solution $u$, where $u_N$ solves problem \eqref{nhqwv} with external force $\mu_N$ instead of $\mu$, and the sequence $\{\mu_N\}_{N=1}^\infty$ of discrete measures is provided by Theorem \ref{bv.th.delta-conv}:
%$$
\begin{equation}
\mu_N=\sum_{k=0}^{N}h_{k,N}\delta_{t_{k,N}},\ h_{k,N}\in H,\ \sum_{k=0}^{N}\|h_{k,N}\|\leq\|\mu\|_{M(0,1;H)}
\end{equation}
%$$
and $0= t_{0,N}<t_{1,N}<\ldots<t_{N,N}=1$. Note that the solution $u_N(t)$ should solve the homogeneous problem for $t\in(t_{k,N},t_{k+1,N})$ and has jumps of time derivative at finitely many points $t=t_{k,N}$:
$$
\Dt u_N(t_{k,N}+0)-\Dt u_N(t_{k,N}-0)=h_{k,N},
$$
so the existence and uniqueness of $u_N$ follows immediately from the analogous result for the {\it homogeneous} problem \eqref{eq.q} and we need not to use Proposition \ref{Prop.nc} here. Moreover, due to Theorem \ref{th.ex1}, we have the uniform energy estimate
%$$
\begin{equation}
\|\xi_{u_N}\|_{L^\infty(0,1;\mathcal E)}\le Q(\|\xi_0\|_{\mathcal E})+Q(\|\mu_N\|_{M(0,1;H)})\le
Q(\|\xi_0\|_{\mathcal E})+Q(\|\mu\|_{M(0,1;H)})
\end{equation}
%$$
for some monotone increasing function $Q$. Thus, passing to a subsequence if necessary and using that $\Phi_{\mu_N}(t)\to\Phi_\mu(t)$ uniformly for all $t$ (due to the special choice of $\mu_N$ explained in Theorem \ref{bv.th.delta-conv}), we may assume that $u_N$ is convergent weakly-star to the weak energy solution $u$ of problem \eqref{nhqwv}, see the proof of Theorem \ref{th.ex1}. Thus, we only need to verify the uniform estimate for the Strichartz norms of solutions~$u_N$. Then passage to the limit $N\to\infty$ will give us the desired estimate for $u$ as well.
\par
Note that we can get the Strichartz estimate for the solution $u_N$ just applying estimate \eqref{est-str-tor} at every time interval $t\in(t_{k,N},t_{k+1,N})$ and using that the energy norm is under the control. However, this is not enough since the obtained estimate will clearly depend on $N$. So we need to proceed in a bit more accurate way.
\par
Let us consider the approximations $u_N^l$, $l=0,\cdots,N-1$ of the solution $u_N$ which solve \eqref{nhqwv} with the same initial data and the external forces
$$
\mu_N^l:=\sum_{k=0}^{l}h_{k,N}\delta_{t_{k,N}}.
$$
Then, on the one hand, due to Theorem \ref{th.ex1},
%$$
\begin{equation}\label{good-en}
\|\xi_{u_N^l}\|_{L^\infty(0,1;\mathcal E)}\le Q(\|\xi_0\|_{\mathcal E})+Q(\|\mu\|_{M(0,1;L^2)})
\end{equation}
%$$
uniformly with respect to $l$ and $N$. On the other hand,
 clearly $u_N^{N}(t)\equiv u_N(t)$ for all $t\in[0,1]$. Moreover,
$$
u_N^{l}(t)=u_N^{l+1}(t),\ \ t<t_{l+1,N}
$$
and the functions $u_N^{l}(t)$ and $u_{N}^{l+1}(t)$, $t\ge t_{l+1,N}$ solve linear homogeneous problem \eqref{eq.q} with the initial data
$$
\xi_{u_N^{l+1}}=\xi_{u_{N}^l}\big|_{t=t_{l+1,N}+0}+\{0,h_{l+1,N}\}.
$$
In particular, due to Corollary \ref{Cor.tor} and estimate \eqref{good-en},
%$$
\begin{equation}\label{good-str}
\|u_N^{l}\|_{L^4(t_{l,N},1;L^{12})}+\|u_N^{l+1}\|_{L^4(t_{l+1,N},1;L^{12})}\le Q(\|\xi_0\|_{\mathcal E})+Q(\|\mu\|_{M(0,1;H)}).
\end{equation}
%$$
Finally, we introduce functions $v_{0}(t):=u_N^0(t)$ and $v_{l+1}(t):=u_{N}^{l+1}(t)-u_N^l(t)$, $l=0,\cdots, N-1$. Then, obviously
%$$
\begin{equation}\label{5.split}
u_N(t)=v_0(t)+\sum_{l=1}^{N}v_l(t),\ t\in[0,1]
\end{equation}
%$$
and the functions $v_{l+1}(t)$, $t>t_{l+1,N}$ solve
%$$
\begin{equation}\label{5.est-dif}
\Dt^2 v_{l+1}-\Dx v_{l+1}=-[f(u_N^{l+1})-f(u_N^l)],\ \ \xi_{v_{l+1}}\big|_{t=t_{l+1,N}+0}=\{0,h_{l+1,N}\}.
\end{equation}
%$$
Note also that $v_{l+1}(t)\equiv0$ for $t<t_{l+1,N}$. To estimate the Strichartz norms of $v_{l+1}$, we use that
$$
|f(u^{l+1}_N)-f(u^l_N)|\le C(|u^{l+1}_N|^4+|u^l_N|^4)|v_{l+1}|
$$
and, therefore, due to H\"older inequality and Sobolev embedding $H^1\subset L^6$, we have
%$$
\begin{equation}\label{5.fest}
\|f(u^{l+1}_N)-f(u^l_N)\|_{H}\le C(\|u^{l+1}_N\|^4_{L^{12}}+\|u^l_N\|^4_{L^{12}})\|\xi_{v_{l+1}}\|_{\mathcal E}.
\end{equation}
%$$
Multiplying now equation \eqref{5.est-dif} by $\Dt v_{l+1}$ and using \eqref{5.fest}, we get
$$
\frac12\frac d{dt}\|\xi_{v_{l+1}}\|^2_{\mathcal E}\le C(\|u^{l+1}_N\|^4_{L^{12}}+\|u^l_N\|^4_{L^{12}})\|\xi_{v_{l+1}}\|_{\mathcal E}^2
$$
and the Gronwall inequality together with the control \eqref{good-str} give
%$$
\begin{equation}\label{energy-v}
\|\xi_{v_{l+1}}\|_{L^\infty(0,1;\E)}\le \(Q(\|\xi_0\|_{\mathcal E})+\|\mu\|_{M(0,1;H)})\)\|h_{l+1,N}\|_{H}.
\end{equation}
%$$
We are now ready to apply the standard Strichartz estimate to the linear equation \eqref{5.est-dif} and get
%$$
\begin{multline}
\|v_{l+1}\|_{L^4(0,1;L^{12})}=\|v_{l+1}\|_{L^4(t_{l+1,N},1;L^{12})}\le\\\le C(\|\xi_{v_{l+1}}(t_{l+1,N}+0)\|_{\mathcal E}+\|f(u_N^{l+1})-f(u_N^l)\|_{L^1(t_{l+1,N},1;H)})\le\\\le
\(Q(\|\xi_0\|_{\mathcal E})+\|\mu\|_{M(0,1;H)})\)\|h_{l+1,N}\|_{H}.
\end{multline}
%$$
Finally, according to \eqref{5.split}, we arrive at
%$$
\begin{multline}
\|u_N\|_{L^4(0,1;L^{12})}\le \(Q(\|\xi_0\|_{\mathcal E})+\|\mu\|_{M(0,1;H)})\)\(1+\sum_{l=0}^{N}\|h_{l,N}\|_{H}\)\le\\\le \(Q(\|\xi_0\|_{\mathcal E})+\|\mu\|_{M(0,1;H)}\)\(1+\|\mu\|_{M(0,1;H)}\).
\end{multline}
%$$
Thus, in the particular case $f(u)=u^5$, the theorem is proved.
\par
We consider now the general case $h(u)\ne0$ which can be derived from the obtained estimate by more or less standard perturbation arguments. We first remind the following simple lemma which can be verified using the convexity arguments (see \cite{vz96} and see also \cite{SADE}).

\begin{lemma}\label{lem.Q}
Let $Q:\R_+\to \R_+$ be a monotone increasing function, $L_1,L_2\in\R_+$ and $\eb\in[0,1]$. Then there exists a smooth monotone increasing function $Q_1:\R_+\to\R_+$ such that
\begin{equation}
Q(L_1+\eb L_2)\leq Q_1(L_1)+\eb Q_1(L_2),
\end{equation}
where $Q_1$ is determined by $Q$ only.
\end{lemma}
We rewrite equation \eqref{nhqwv} in the form
$$
\Dt^2 u-\Dx u+u^5=\mu -h(u)
$$
and apply already proved estimate \eqref{L5L10.nh} on the interval $t\in[0,T]$ where $T\le1$ will be determined later. Then, we have
$$
\|u\|_{L^5(0,T;L^{10})}\le \|u\|_{L^4(0,T;L^{12})}+\|\xi_u\|_{L^\infty(0,T;\mathcal E)}\le Q(\|\xi_0\|_{\mathcal E})+Q(\|\mu\|_{M(0,1;H)}+\|h(u)\|_{L^1(0,T;H)}).
$$
Since the function $h(u)$ has a sub-quintic growth rate, the H\"older inequality gives
$$
\|h(u)\|_{L^1(0,T;L^2)}\le CT^\kappa(1+\|u\|_{L^5(0,T;L^{10})}^5)
$$
for some positive exponent $\kappa$. Inserting this estimate into the previous one and using Lemma \ref{lem.Q}, we arrive at
%$$
\begin{equation}\label{5.pert}
\|u\|_{L^5(0,T;L^{10})}\le Q(\|\xi_0\|_{\mathcal E})+Q(\|\mu\|_{M(0,1;L^2)})+T^\kappa Q(\|u\|_{L^5(0,T;L^{10})}).
\end{equation}
%$$
Important here that the function $Q$ is independent of $T$. Fixing $T=T(\|\xi_0\|_{\mathcal E}+\|\mu\|_{M(0,1;H)})$ to be small enough, we derive from \eqref{5.pert} that
$$
\|u\|_{L^5(0,T;L^{10})}\le Q(\|\xi_0\|_{\mathcal E})+Q(\|\mu\|_{M(0,1;H)})
$$
for some new monotone function $Q$. Since the energy norm of the solution is under the control, we may apply this estimate on the intervals $[T,2T]$, $[2T,3T]$ and so on. This gives us the desired control
$$
\|u\|_{L^5(0,1;L^{10})}\le Q(\|\xi_0\|_{\mathcal E})+Q(\|\mu\|_{M(0,1;H)})
$$
for some monotone increasing function $Q$. Since the $L^1(H)$-norm of $f(u)$ is controlled by the $L^5(L^{10})$-norm of $u$, we may get the control of the $L^4(L^{12})$-norm of $u$ using the Strichartz estimate for the linear equation. Thus, the theorem is proved.
\end{proof}

As corollary of Theorem \ref{th.str-f1} we obtain the desired dissipative Strichartz estimate for the solutions of the nonlinear damped wave equation \eqref{eq.qdw} which is crucial for what follows.
\begin{cor}\label{th.str-f2}
Let the non-linearity $f$ satisfy \eqref{4.f} and the external force $\mu\in M_b(\R;H)$. Then for any $\tau\in\R$ and initial data $\xi_\tau\in \E$ the problem \eqref{eq.qdw} possesses a unique Shatah-Struwe solution $u$ and the following estimate holds
%$$
\begin{equation}
\label{est.dis-str}
\|\xi_u(t)\|_\E+\|u\|_{L^4([t,t+1];L^{12})}\leq e^{-\beta(t-\tau)}Q(\|\xi_\tau\|_\E)+Q(\|\mu\|_{M_b(\R;H)}),\ t\geq\tau,
\end{equation}
%$$
for some constant $\beta>0$ and some monotone nondecreasing function $Q$ which are independent of $\xi_\tau\in\E$, $\mu\in M_b(\R;H)$ and $\tau\in\R$.
\end{cor}
\begin{proof}
The result easily follows if one applies estimate \eqref{L5L10.nh} on $[t,t+1]$ to equation \eqref{eq.qdw}, treating $\gamma\Dt u +u$ as the right hand side, and combines the obtained estimate with dissipative energy estimate \eqref{est.disen} and the estimate of Lemma \ref{lem.Q}.
\end{proof}

\begin{rem}\label{Rem5.str-q} Since the $L^4(L^{12})$-norm of the solution $u$ together with the energy norm allow us to control the $L^1(L^2)$-norm of the non-linearity $f(u)$, applying the Strichartz estimates for the linear equation and treating $f(u)$ as an external force, we get the dissipative estimate for other Strichartz norms of $u$, namely
%$$
\begin{equation}\label{5.str-q}
\|u\|_{L^{2/q}(t,t+1;L^{6/(1-q)})}\le Q(\|\xi_u(0)\|_{\mathcal E})e^{-\alpha t}+Q(\|\mu\|_{M_b(\R,H)}),
\end{equation}
%$$
where $q\in[0,1)$ and the function $Q$ depends on $q$, but is independent on $u$ and $\mu$.
\end{rem}

\begin{rem}\label{rem.Crist-Kiselev} We recall that the Strichartz estimates for non-homogeneous {\it linear} dispersive equations are usually derived from the homogeneous ones using the duality arguments and the so-called Christ-Kiselev lemma (see \cite{tao} and references therein). In contrast to this, the approach suggested in the proof of Theorem \ref{th.str-f1} works directly for \emph{nonlinear} (and even \emph{critical nonlinear}) problems and can be treated as a generalization of Christ-Kiselev lemma to the non-linear case. We believe that this approach will be useful for other dispersive equations as well.   
\end{rem}

\section{Damped wave equation: weak uniform attractors}\label{s.ua}
We start with basic definitions of nonautonomous dynamical systems (adapted to the measure-driven case), for more detailed treatment and recent advances see \cite{CV}, \cite{Zntrc}.

Let us first recall the key definitions and concept related with the attractors theory. We start with the autonomous case

\begin{Def}\label{Def.auto}
Let $\Phi$ be a Hausdorff  topological space and $\mathbb S(t):\Phi\to\Phi$, $t\ge0$ be a semigroup on it. Let also $\mathbb B$ be a family of sets $B\subset \Phi$ satisfying the property: if $B\in\mathbb B$ and $B_1\subset B$ then $B_1\in\mathbb B$. The sets $B\in \mathbb B$ are called bounded.
\par
A set $\mathcal B\in \mathbb B$ is an absorbing set for the semigroup $\mathbb S(t)$ if for any $B\in\mathbb B$ there exists time $T=T(B)$ such that
$$
\mathbb S(t)B\subset\mathcal B,\ \ \ t\ge T.
$$
A set $\mathcal B$ is an attracting set for the semigroup $S(t)$ if for every neighbourhood $\mathcal O(\mathcal B)$ and every $B\in\mathbb B$, there exists $T=T(\mathcal O,B)$ such that
$$
\mathbb S(t)B\subset\mathcal O(\mathcal B),\ \ \ t\ge T.
$$
Finally, a set $\mathcal A$ is a global attractor for the semigroup $S(t)$ if
\par
1) $\mathcal A$ is compact and bounded ($\mathcal A\in\mathbb B$) in $\Phi$;
\par
2) $\mathcal A$ is an attracting set for $S(t)$;
\par
3) $\mathcal A$ is a minimal set which satisfies properties 1) and 2).
\end{Def}
The 3rd property of the global attractor is usually formulated as the strict invariance with respect to $\mathbb S(t)$, but keeping in mind the non-autonomous case, we prefer to state it as minimality, see \cite{CV} for more details. To state the existence result for the autonomous case, we need one more definition.
\begin{Def} The semigroup $\mathbb S(t):\Phi\to\Phi$ is (sequentially) asymptotically compact on a set $B\subset\Phi$ if, for any sequences $t_n\to\infty$ and $x_n\in B$, the sequence $\mathbb S(t_n)x_n$ is precompact in $\Phi$.
\end{Def}

\begin{prop}\label{Prop.auto} Let the semigroup $\mathbb S(t):\Phi\to\Phi$ possess an absorbing set $\mathcal B\in\mathbb B$. Assume also that
\par
1. The topology induced on $\mathcal B$ by the inclusion $\mathcal B\subset\Phi$ is metrizable and complete (i.e., $\mathcal B$ is a complete metric space);
\par
2. The semigroup $\mathbb S(t)$ is asymptotically compact on $\mathcal B$.
\par
Then the semigroup $\mathbb S(t)$ possesses a global attractor $\mathcal A\subset\mathcal B$.
\par
In addition, if the operators $\mathbb S(t)$ are continuous on $\mathcal B$ for every fixed $t$, then the attractor $\mathcal A$ is strictly invariant: $\mathbb S(t)\mathcal A=\mathcal A$ and is generated by all bounded trajectories defined for all $t\in\R$:
%$$
\begin{equation}\label{5.str}
\mathcal A=\mathcal K\big|_{t=0},
\end{equation}
%$$
where $\mathcal K:=\{u:\R\to\Phi\,: \mathbb S(t)u(\tau)=u(t+\tau),\ t\ge0,\ \tau\in\R,\ \cup_{t\in\R}u(t)\in\mathbb B\}$.
\end{prop}
The proof of this proposition is standard and the details can be found in \cite{CV}.
\par
Since we are mainly interested in the non-autonomous equations, we recall below how the above concepts can be extended to the non-autonomous case. The first difference is that the solution operators are no more generate a semigroup, but the so-called dynamical process which is two-parametric family $U(t,\tau)$, $t\ge\tau$ acting in the phase space and satisfying
%$$
\begin{equation}\label{5.process}
U(t,t)=Id,\ \ U(t,\tau)\circ U(\tau,s)=U(t,s),\ t\ge\tau\ge s.
\end{equation}
%$$
The operator $U(t,\tau)$ is understood as a solution operator which maps the in initial data at time moment $\tau$ to the solution at time moment $t$.
\begin{Def}\label{Def.n-auto}
Let $\mathcal E$ be a Hausdorff  topological space and $U(t,\tau):\mathcal E\to\mathcal E$, $t\ge\tau$ be a dynamical process on it. Let also $\mathbb B$ be a family of sets $B\subset \mathcal E$ satisfying the property: if $B\in\mathbb B$ and $B_1\subset B$ then $B_1\in\mathbb B$. The sets $B\in \mathbb B$ are called bounded.
\par
A set $\mathcal B\in \mathbb B$ is a {\it uniformly} absorbing set for the semigroup $S(t)$ if for any $B\in\mathbb B$ there exists time $T=T(B)$ such that
$$
U(t,\tau)B\subset\mathcal B,\ \ \ t-\tau\ge T,\ \ \tau\in\R.
$$
A set $\mathcal B$ is a {\it uniformly} attracting set for the process $U(t,\tau)$ if for every neighbourhood $\mathcal O(\mathcal B)$ and every $B\in\mathbb B$, there exists $T=T(\mathcal O,B)$ such that
$$
U(t,\tau)B\subset\mathcal O(\mathcal B),\ \ \ t-\tau\ge T,\ \ \tau\in\R.
$$
Finally, a set $\mathcal A_{un}$ is a {\it uniform} attractor for the process $U(t,\tau)$ if
\par
1) $\mathcal A$ is compact and bounded in $\mathcal E$;
\par
2) $\mathcal A$ is a uniformly attracting set for $U(t,\tau)$;
\par
3) $\mathcal A$ is a minimal set which satisfies properties 1) and 2).
\par\noindent
In the sequel, $\mathcal E$ will be a Banach space (or even Hilbert space) endowed either by the strong or weak topology. The associated uniform attractor will be referred  as {\it strong} or {\it weak} unform attractor respectively. In both cases, $\mathbb B$ consists of all bounded sets of the Banach space considered.
\end{Def}
The generalization of the concept of asymptotic compactness is also straightforward.

\begin{Def} The process $U(t,\tau):\mathcal E\to\mathcal E$ is {\it uniformly} asymptotically compact on a set $B\subset\mathcal E$ if, for any sequences $t_n,\tau_n\in\R$ such that $t_n-\tau_n\to\infty$ and any sequence $x_n\in B$, the sequence $U(t_n,\tau_n)x_n$ is precompact in $\mathcal E$.
\end{Def}
As well as the following existence result, see \cite{CV} for details.
\begin{prop}\label{Prop.non} Let the process $U(t,\tau):\mathcal E\to\mathcal E$ possess a uniformly absorbing set $\mathcal B\in\mathbb B$. Assume also that
\par
1. The topology induced on $\mathcal B$ by the inclusion $\mathcal B\subset\Phi$ is metrizable and complete (i.e., $\mathcal B$ is a complete metric space);
\par
2. The process $U(t,\tau)$ is uniformly asymptotically compact on $\mathcal B$.
\par
Then the process $U(t,\tau)$ possesses a uniform attractor $\mathcal A_{un}\subset\mathcal B$.
\end{prop}

We now return to our damped wave equation \eqref{eq.qdw}. Since, according to Corollary \ref{th.str-f2}, for every $\tau\in\R$, $\xi_\tau\in\mathcal E:=H^1(\mathbb T^3)\times L^2(\mathbb T^3)$ and any $\mu\in M_b(\R,H)$, this problem possesses a unique Strichartz solution $\xi_u(t)$, so we may introduce a family of dynamical processes $U_\mu(t,\tau)$, $\mu\in M_b(\R,H)$ in the energy phase space. However, since in contrast to the usual case, the trajectories $\xi_u(t)$ may have jumps, we should be a bit accurate in order to preserve the property \eqref{5.process}. In particular, we use here our agreement that the trajectories $\xi_u(t)$ are left-semicontinuous and we may set  
%$$
\begin{equation}
U_\mu(t,\tau)\xi_\tau:=\xi_u(t-0)=\lim_{s\to t-0}\xi_u(s)=\xi_u(t)
\end{equation}
%$$
and we set $\xi_u(\tau-0):=\xi_\tau$. Then, as not difficult to see that the operators $U_\mu(t,\tau)$ thus defined are indeed the dynamical processes in the energy space $\mathcal E$, so we may study their uniform attractors. We fix $\mathbb B$ as a family of bounded (in a usual sense) subsets of our energy space $\mathcal E$ (it is a Banach space, so bounded sets are well-defined). Then, estimate
\eqref{est.dis-str} guarantees the existence of a uniformly attracting set. Moreover, it can be taken in the form
%$$
\begin{equation}\label{5.abs-set}
\mathcal B:=\big\{\xi\in\mathcal E\,: \|\xi\|_{\mathcal E}\le 2Q(\|\mu\|_{M_b(\R,H)})\big\}.
\end{equation}
%$$
Recall that in this Section we are mainly interested in {\it weak} uniform attractors, so we endow the space $\mathcal E$ with the weak topology and denote the obtained locally convex space by $\mathcal E_w$. Since the space $\mathcal E$ is a reflexive Banach space, the absorbing set $\mathcal B$ is compact and metrizable in a weak topology of $\mathcal E_w$, so all of the assumptions of Proposition \ref{Prop.non} are automatically verified and we have proved the following result.

\begin{theorem}\label{Th.wnon-attr} Let the assumptions of Corollary \ref{th.str-f2} hold. Then, for every $\mu\in M_b(\R,H)$ the dynamical process $U_\mu(t,\tau)$ possesses a uniform attractor $\mathcal A_{un}^w$ in the space $\mathcal E_w$ which is called a weak uniform attractor for equation \eqref{eq.qdw}.
\end{theorem}

At the next step, we describe the extension of the key representation formula \eqref{5.str} to the case of uniform attractors. To this end, we will use (following \cite{CV}), the reduction of the dynamical process $U_\mu(t,\tau)$ to a semigroup acting on the extended phase space. To this end, we introduce a  group of shifts acting on the space of measures $M_b(\R,H)$:
$$
(T(s)\mu)(t):=\mu(t+s),\ \ t,s\in\R.
$$
Then, as not difficult to verify, the introduced dynamical processes $U_\mu(t,s)$ satisfies the following translation identity (=cocycle property):
%$$
\begin{equation}\label{5.co}
U_{T(s)\mu}(t,\tau)=U_\mu(t+s,\tau+s), \ \ t\ge\tau\in\R,\ \ s\in\R.
\end{equation}
%$$
In order to fix the proper topology on the space $M_b(\R,H)$, we recall that $M_{loc}(\R, H)$ is a dual space for $C_{00}(\R,H)$, where $C_{00}$ means continuous functions with compact support endowed with the inductive topology. Denote by $M_{loc}^{w^*}(\R,H)$ the space $M_{loc}(\R,H)$ endowed with the associated weak star topology. Then, by Banach-Alaoglu theorem, the unit ball of $M_b(\R,H)$ is compact and metrizable in the topology of $M_{loc}^{w^*}(\R,H)$. We recall that $\mu_n\to\mu$ in this topology if and only if
$$
\lim_{n\to\infty}\int_\R(\phi(s),\mu_n(ds))=\int_\R(\phi(s),\mu(ds))
$$
for every $\phi\in C_{00}(\R,H)$. We are now ready to define the hull of the  measure $\mu\in M_b(\R,H)$ as a closure of all shifts of $\mu$ in the weak-star topology:
%$$
\begin{equation}\label{H(m)}
\Cal H(\mu)=\left[\{T(s)\mu,s\in\R\}\right]_{M_{loc}^{w^*}(\R,H)},
\end{equation}
%$$
where $[\ \cdot\ ]_{M_{loc}^{w^*}(\R,H)}$ means the closure in $M_{loc}^{w^*}(\R,H)$. Obviously, the set $\mathcal H(\mu)$ endowed with the weak-star topology is a compact metric space and the group of shifts $$
T(s):\mathcal H(\mu)\to \mathcal H(\mu),\ \ T(s)\mathcal H(\mu)=\mathcal H(\mu),
$$
acts continuously on $\mathcal H(\mu)$.
\par
Let now $U_\mu(t,\tau):\mathcal E\to\mathcal E$ be a family of dynamical processes associated with damped wave equation \eqref{eq.qdw}. Then, the extended phase space for problem \eqref{eq.qdw} is defined via
$$
\Phi:=\mathcal E\times\mathcal H(\mu)
$$
and the associated autonomous dynamical system on $\Phi$ acts as follows
%$$
\begin{equation}\label{ext.sem}
\Bbb S(t)\{\xi_0,z\}:=\{U_z(t,0),T(t)z\},\ \ \xi_0\in \mathcal E,\ \ z\in\mathcal H(\mu).
\end{equation}
%$$
Indeed, the semigroup property for $\Bbb S(t)$ is an immediate corollary of the translation identity \eqref{5.co}.
\par
The key general idea is to relate the uniform attractor $\mathcal A_{un}$ for the dynamical process $U_\mu(t,\tau)$ constructed above with the global attractor $\mathbb A$ of the extended semigroup $\Bbb S(t)$ and, in particular, to describe the structure of $\mathcal A_{un}$ using the representation \eqref{5.str} for the autonomous case. Namely, we endow the extended phase space $\Phi=\mathcal E\times\mathcal H(\mu)$ with the topology  generated by the embedding $\Phi\subset \mathcal E_w\times M_{loc}^{w^*}(\R,H)$ and fix bounded sets in $\Phi$ as follows: $B\subset \Phi$ is bounded iff $\Pi_1B$ is bounded in $\mathcal E$ (here and below $\Pi_1$ means the projection to the first component of the Cartesian product $\mathcal E\times\mathcal H(\mu)$). Then, due to estimates \eqref{est.dis-str} and the elementary fact that
%$$
\begin{equation}
\|z\|_{M_b(\R,H)}\le \|\mu\|_{M_b(\R,H)},\ \ z\in \mathcal H(\mu),
\end{equation}
%$$
 the set
$$
\mathcal B_{ext}:=\mathcal B\times\mathcal H(\mu),
$$
where $\mathcal B$ is defined by \eqref{5.abs-set},
is a compact metrizable absorbing set for the extended semigroup $\mathbb S(t)$ and, therefore, due to Proposition \ref{Prop.auto}, the semigroup $\Bbb S(t)$ possesses a global attractor $\mathbb A_{ext}$. The next Theorem gives the desired structure of the constructed uniform attractor for the damped wave equation \eqref{eq.qdw}.

\begin{theorem}\label{Th5.cont-attr} Let the assumptions of Theorem \ref{Th.wnon-attr} hold and let, in addition, the maps $(\xi_0,z)\to U_z(t,\tau)\xi_0$ be continuous (in the weak topology) as maps $\mathcal B_{ext}$ to $\E$ for every fixed $t,\tau\in\R$, $t\ge\tau$. Then,
%$$
\begin{equation}\label{6.str1}
\mathcal A_{un}=\Pi_1\mathbb A_{ext}
\end{equation}
%$$
and, moreover,
%$$
\begin{equation}\label{6.str2}
\mathcal A_{un}=\cup_{z\in\mathcal H(\mu)}\mathcal K_z\big|_{t=0},
\end{equation}
%$$
where $\mathcal K_z:=\big\{u\in L^\infty(\R,\mathcal E),\ \ U_z(t,\tau)u(\tau)=u(t), \ t\ge\tau\in\R\big\}$ is the so-called kernel of the process $U_z(t,\tau)$ in the terminology of \cite{CV}.
\end{theorem}
The proof of this result in general setting can be found in \cite{CV}.
\par
Note that, in contrast to the usual case, the continuity assumption is {\it not satisfied} for general $\mu\in M_b(\R,H)$. Namely, the following result holds.

\begin{prop}\label{Prop.cont} Let the assumptions of Theorem \ref{Th.wnon-attr} hold. Then the continuity assumption of Theorem \ref{Th5.cont-attr} hold if and only if
%$$
\begin{equation}\label{6.mes-cont}
1.\ \  z_n(\{\tau\})\equiv 0;\ \ 2. \ \ z_n([\tau,t])\to z([\tau,t]) \ \ \text{weakly in $H$}
\end{equation}
%$$
for every sequence $z_n\in\mathcal H(\mu)$ such that $z_n\to z$ weakly star in $M_{loc}(\R,H)$ and every fixed $t\ge\tau\in\R$.
\end{prop}
\begin{proof} Indeed, let \eqref{6.mes-cont} be satisfied. We need to prove that $U_{z_n}(t,\tau)\xi_n$ is weakly convergent to $U_z(t,\tau)\xi_0$ if $z_n\to z$ in $\mathcal H(\mu)$ and $\xi_n\to\xi_0$ in $\mathcal E_w$. Let $\xi_{u_n}(t):=U_{z_n}(t,\tau)\xi_n$ be the corresponding Shatah-Struwe solutions. Then, due to the uniform dissipative estimate \eqref{est.dis-str}, we may assume without loss of generality that $\xi_{u_n}\to\xi_u$ weakly star in $L^\infty(\tau,t;\mathcal E)$. Thus, we only need to pass to the limit in  \eqref{4.intdef}. Namely, taking into account that $z_n(\{t\})=0$,
this equality reads
%$$
\begin{equation}\label{6.intdef}
\Dt u_n(t)=-\int_\tau^t(-\Dx+1)u_n(s)+f(u_n(s))\,ds-\gamma u_n(t)+z_n([\tau,t])+u_{\tau,n}'+\gamma u_{\tau,n},
\end{equation}
%$$
where $\xi_n:=\{u_{\tau,n},u_{\tau,n}'\}$. Obviously, the limit function $\xi_u(t)$ satisfies equation \eqref{eq.qdw} in the sense of distributions and the passage to the limit in \eqref{6.intdef} is also straightforward due to condition \eqref{6.mes-cont}.
\par
Let us now check the necessity. We first check that $z(\{\tau\})=0$ for all $z\in\mathcal H(\mu)$ is necessary. Indeed, let $z(\{0\})\ne0$ for some $z\in\mathcal H(\mu)$. Since the number of jumps is at most countable, we may assume that $z(\{-1\})=0$. Let us consider a sequence $z_n:=T_{1/n}z$ and $\xi_{u_n}:=U_{z_n}(-1,0)\xi_0$, where $\xi_0\in\mathcal E$. Clearly, $z_n\to z$ as $n\to\infty$ and we may assume without loss of generality that $\xi_{u_n}\to\xi_{\bar u}$ weakly star in $L^\infty(-1,0;\mathcal E)$. Moreover, by the Helly selection theorem, we may also assume that $\xi_{u_n}(t)\to\xi_{\bar u}(t)$ weakly in $\mathcal E$ for almost all $t\in[-1,0]$. Let $\xi_u(t):=U_{z}(-1,t)\xi_0$. Then, two cases a priori possible:
\par
1. $\xi_{\bar u}(t)\ne\xi_{u}(t)$ on a subset of $[-1,0]$ of positive measure. Then, the continuity obviously fails.
\par
2. $\xi_{\bar u}=\xi_{u}$ almost everywhere. Then, passing to the limit in \eqref{6.intdef}, say, in $H^{-2}$, we get
%$$
\begin{multline}
\lim_{n\to\infty}\Dt u_n(0-)-\Dt u(0-)=\lim_{n\to\infty} z_n([-1,0))-z([-1,0))=\\=-\lim_{n\to\infty}z([-1,-1+1/n))+\lim_{n\to\infty}z([-1,1/n))-z([-1,0))=z(\{0\})\ne0
\end{multline}
%$$
and the continuity of $U_z(-1,0)$ fails. Thus, the necessity of the first condition is proved.
\par
The necessity of the second condition can be proved analogously, but even simpler since we need not to shift the measures and may pass to the limit directly in \eqref{6.intdef}. So, the proposition is proved.
\end{proof}
The proved proposition reduces finding necessary and sufficient conditions for the weak continuity of the dynamical process associated with equation \eqref{eq.qdw} to verifying conditions \eqref{6.mes-cont} which are purely measure theoretic and can be completely understood. To state the criterion, we need the following definition.

\begin{Def}\label{Def6.wmes} A measure $\mu\in M_b(\R,H)$ is weak uniformly non-atomic if for every $\psi\in H$ there exists a monotone increasing  function $\omega_\psi:\R_+\to\R_+$ such that
%$$
\begin{equation}\label{6.wun}
\lim_{h\to0}\omega_\psi(h)=0,\ \ \ |(\mu([s,t]),\psi)|\le \omega_\psi(|t-s|)
\end{equation}
%$$
for all $t\ge s\in\R$. The space of such measures is denoted by $M^{wna}_b(\R,H)$.
\end{Def}
Then, the following result holds.

\begin{prop}\label{Prop6.cont} Assumptions \eqref{6.mes-cont} are satisfied if and only if the initial measure $\mu\in M_b(\R,H)$ is weak uniformly non-atomic.
\end{prop}
\begin{proof} Assume that assumptions \eqref{6.mes-cont} hold and let $\psi\in H$ be arbitrary. Consider the function $G:\mathcal H(\mu)\times[0,1]\to\R$ defined by
$$
G(z,\tau):=(z([0,\tau]),\psi).
$$
Then, due to the first condition of \eqref{6.mes-cont}, this function is continuous in $\tau$ for every fixed $z$. On the other hand, due to the second condition of \eqref{6.mes-cont}, it is continuous in $z$ for every fixed $\tau$. Thus, there is a point $\tau_0\in(0,1)$ such that $G$ is jointly continuous at $\{z,\tau_0\}$ for every $z\in\mathcal H(\mu)$ (in a fact, there is a dense set of such points $\tau_0\in[0,1]$, see e.g., \cite{namioka} and references therein). Since $\mathcal H(\mu)$ is compact, we conclude that there exists a monotone increasing function $\omega_\psi:\R_+\to\R_+$ such that
$$
|(z([\tau_0,\tau_0+s]),\psi)|= |G(z,\tau_0+s)-G(z,\tau_0)|\le\omega_\psi(s) \ \text{ for all } z\in\mathcal H(\mu)
$$
and $\lim_{h\to0}\omega_\psi(h)=0$. Using finally that
$$
(T(h)z)([\tau,t])=z([\tau+h,t+h])
$$
and that $T(h)\mathcal H(\mu)=\mathcal H(\mu)$, we deduce \eqref{6.wun}. Thus, conditions \eqref{6.mes-cont} imply that $\mu$ is weakly uniformly non-atomic.
\par
Let now $\mu$ be weakly uniformly non-atomic. Then, as not difficult to see using the Helly selection theorem, see Theorem \ref{bv.th.w*sc} and Corollary \ref{bv.cor2},
%$$
\begin{equation}\label{6.un}
|(z([\tau,t]),\psi)|\le\omega_\psi(t-\tau),\ \ \forall z\in\mathcal H(\mu),
\end{equation}
%$$
where the functions $\omega_\psi$ are the same as in \eqref{6.wun}. Then, the first assumption of \eqref{6.mes-cont} is immediate and the second one is the standard corollary of the Arzela theorem and the proposition is proved.
\end{proof}
Thus, we have proved the following theorem which can be considered as the main result of this Section.
\begin{theorem}\label{Th6.main} Let the assumptions of Theorem \ref{Th.wnon-attr} hold and let, in addition, $\mu\in M^{wna}_b(\R,H)$. Then, the weak uniform attractor $\mathcal A_{un}$ of equation \eqref{eq.qdw} satisfies \eqref{6.str1} and \eqref{6.str2}.
\end{theorem}
Indeed, this is an immediate corollary of Theorem \ref{Th5.cont-attr} and Propositions \ref{Prop.cont} and \ref{Prop6.cont}.
\par
We now give some examples clarifying the posed conditions to the external forces.

\begin{example} We start with the case of regular measures $\mu(t)\in L^p_b(\R,H)$ where $p>1$. Then,
$$
\|\mu([\tau,\tau+h])\|_H=\|\int_\tau^{\tau+h}\mu(t)\,dt\|_H\le \int_\tau^{\tau+h}\|\mu(t)\|_H\,dt \le\(\int_\tau^{\tau+h}1dt\)^{1-1/p}\|\mu\|_{L^p_b}=C|h|^{1-1/p}.
$$
Thus, $\mu\in M_b^{wna}(\R,H)$ (and even strongly uniformly non-atomic) and the theory works. Moreover, in this case
$$
\mathcal H(\mu)\subset L^p_b(\R,H)\subset L^1_b(\R,H),
$$
so, all measures from the hull are regular.
\par
This will be not the case, if we consider the so-called {\it normal} external forces from $L^1_b(\R,H)$ which has been introduced in \cite{Lu} to study the uniform attractors for parabolic equations (see also \cite{Zntrc} for more details), we recall that $\mu\in L^1_b(\R,H)$ is normal if there is a monotone increasing function $\omega:\R_+\to\R_+$ such that $\lim_{h\to0}\omega(h)=0$ and
%$$
\begin{equation}\label{6.snormal}
\int_t^{t+h}\|\mu(t)\|_H\,dt\le \omega(h),\ \ t\in\R.
\end{equation}
%$$
In this case, we still have $\mu\in M_b^{wna}(\R,H)$ (also $\mu\in M_b^{sna}(\R,H)$) and the theory works. However, in this case the hull $\mathcal H(\mu)$ may contains measures with non-zero singular part. According to the Dunford-Pettis theorem, see Section \ref{s.bv}, the condition which guarantees that $\mathcal H(\mu)\subset L^1_b(\R,H)$ is a bit stronger:
$$
\int_A\|\mu(t)\|_H\,dt\le \omega(|A|),
$$
where $A$ is any (Lebesgue) measurable set on $\R$ and $|A|$ stands for the Lebesgue measure.
\par
Condition \eqref{6.normal} can be weakened as follows:
%$$
\begin{equation}\label{6.normal}
\|\int_t^{t+h}\mu(t)\,dt\|_H\le \omega(h),\ \ t\in\R
\end{equation}
%$$
which still guarantees that $\mu\in M_b^{wna}(\R,H)$.
\end{example}
\begin{example}\label{Ex6.ex} We now give two more exotic examples clarifying the nature of weakly non-atomic measures. We start with the scalar measure $\mu\in M_b(\R,\R)$. To this end, we fix a non-negative smooth function $\phi\in C_0^\infty(\R)$ supported on $[0,1]$ such that $\int_\R\phi(t)\,dt=1$ and consider the delta-like sequence $\phi_n(t):=n\phi(nt)$. Finally, we introduce the following function
%$$
\begin{equation}
\mu(t):=\frac12\sum_{n=2}^\infty\(\phi_{n^2}(t-n)-\phi_{n^2}(t-n-\frac 1{n^2})\).
\end{equation}
%$$
Clearly, this function belongs to $L^1_b(\R)$. It is also not difficult to show, that the $n$th term of this function averages to zero. So, particularly, $\mu\in M_b^{wna}(\R)$ and
$$
T(s)\mu\rightharpoondown 0
$$
as $s\to\infty$. On the other hand, the total variation of this measure reads
$$
|\mu|(t):=\frac12\sum_{n=2}^\infty\(\phi_{n^2}(t-n)+\phi_{n^2}(t-n-\frac 1{n^2})\)
$$
and we see that $nth$ term now tends to the $\delta$-function at $t=n$. Particularly,
$$
T(n)|\mu|\rightharpoondown\sum_{n\in\Bbb Z}\delta(t-n)\ne0.
$$
Thus, $|\mu|\notin M_b^{wna}(\R)$, so the assumption \eqref{6.snormal} does not imply \eqref{6.normal} and the class of measure $M_b^{wna}(\R)$ is indeed larger than $M_b^{sna}(\R)$.
\par
The next example is somehow complementary to the previous one and an alternative construction in the infinite-dimensional spaces. Namely, let $H$-be a Hilbert spaces and $\{e_n\}_{n=1}^\infty$ be an orthonormal base in it. Let
$$
\mu(t):=\sum_{n=1}^\infty\phi_n(t-n)e_n.
$$
Then, clearly $\mu\in L^1_b(\R,H)$ and its total variation reads
$$
|\mu|(t):=\sum_{n=1}^\infty\phi_n(t-n).
$$
Thus, taking any $\psi\in H$ and using that $(\psi,e_n)\to0$, we see that $\mu\in M_b^{wna}(\R,H)$. However, its total variation clearly does not belong to this space.
\end{example}
Our last example shows the pathology which may appear in the case where the condition $\mu\in M_b^{wna}(\R,H)$ is violated.
\begin{example}\label{Ex6.bad} Let us consider the first order ODE in the form
%$$
\begin{equation}\label{6.pat}
y'=y-y^3-3+3\arctan t.
\end{equation}
%$$
The example for the hyperbolic equation can be obtained analogously by adding the term $\eb y''(t)$ but the construction become less transparent, so we prefer to deal with the first order equation. In this case, the uniform attractor can be found explicitly. Namely, the external force now is $\mu_0(t)=3\arctan t$ and its hull gives
$$
\mathcal H(\mu_0)=\{-3\}\cup\{+3\}\cup\{\mu_0(t+h),\, h\in\R\}.
$$
Moreover, as not difficult to see, using e.g., the comparison principle, that every complete trajectory $y(t)$ which corresponds to the external force $\mu\in \mathcal H(\mu_0)$, $\mu\ne\pm3$ satisfies
$$
\lim_{t\to-\infty}y(t)=-2,\ \  \lim_{t\to+\infty}y(t)=-1,\ \ y'(t)>0
$$
and, consequently,
$$
\mathcal K_z=[-2,-1],\ \ \ z\in\mathcal H(\mu_0),\ \ z\ne\pm3.
$$
Finally, in the case when $z=-3$ the equation is monotone, so $\mathcal K_z=\{-2\}$ and in the case $z=+3$, we have the autonomous regular attractor $\mathcal K_z=[-1,1]$. Therefore,
$$
\mathcal A_{un}=\cup_{z\in\mathcal H(\mu_0)}K_z=[-2,1].
$$
We now consider the perturbed version of equation \eqref{6.pat}:
%$$
\begin{equation}\label{6.pat1}
y'=y-y^3-3+3\arctan t+\bar\mu(t),\ \ \bar\mu(t):=\frac12\sum_{n=1}^\infty \phi_{Kn}(t-Kn)-\phi_{Kn}(t-Kn-\frac1{Kn}),
\end{equation}
%$$
where $K>0$ is sufficiently big number and $\phi_n(t)$ is the same as in the previous example. Then, since $T(s)\bar\mu\to0$ as $s\to\pm\infty$, the hull of this external force $\mu+\bar\mu$ is similar to the non-perturbed one
$$
\mathcal H(\mu+\bar\mu)=\{+3\}\cup\{-3\}\cup\{\mu(t+h)+\bar\mu(t+h),\ \ h\in\R\}.
$$
Then, using the fact that the impact of the right-hand side $\frac12(\phi_{Kn}(t-Kn)-\phi_{Kn}(t-Kn-\frac1{Kn}))$ to the solution of \eqref{6.pat} is just a spike of size close to one half centered near $t=Kn$ if $K$ is large enough, we see that
$$
\mathcal K_z=[-2,-1/2],\ \ z\in\mathcal H(\mu+\bar\mu),\ \ z\ne\pm3.
$$
Thus, we have
$$
\cup_{z\in\mathcal H(\mu+\bar\mu)}\mathcal K_z=[-2,1].
$$
On the other hand, if we take $y\big|_{t=\tau}=1$ with $\tau>0$ big enough, we get a trajectory which is close to $y(t)=1$ with spikes of size close to one half. This shows that
$$
\mathcal A_{un}=[-2,\frac32]\ne \cup_{z\in\mathcal H(\mu+\bar\mu)}\mathcal K_z.
$$
\end{example}
\begin{rem}\label{Rem6.tr} We recall that the representation formula \eqref{6.str1} plays the fundamental role in the theory of non-autonomous attractors (see e.g. \cite{CV}), so the last example shows that the constructed theory of uniform attractors for general measures $\mu\in M_b(\R,H)$ is not satisfactory and we really need the restriction $\mu\in M_b^{wna}(\R,H)$ to have a reasonable theory.
\par
Up to the moment, the problem of building up a satisfactory attractors theory for general measures $\mu\in M_b(\R.H)$ remains open. The most natural and straightforward idea here is to endow the space $M_b(\R,H)$ with a different topology in which the $U_\mu(t,\tau)$ become continuous in $\mu$. But unfortunately this does not work even in the scalar case. Indeed, we actually need the topology $\Tau$ on the space of measures $M(0,1)$ satisfying two properties:
\par
1) The unit ball in $M(0,1)$ is sequentially compact in $\Tau$.
\par
2) The convergence $\mu_n\to\mu$ in $\Tau$ implies the point-wise convergence of distribution functions $\Phi_{\mu_n}(t)\to\Phi_\mu(t)$ for every fixed $t\in[0,1]$.
\par
But this topology does not exist. Indeed, consider a sequence $\mu_n=\delta(t-1/2)-\delta(t-1/2-1/n)$. This sequence clearly convergent to zero in the weak-star topology and does not converge to zero in $\Tau$ (since $\Phi_{\mu_n}(1/2)=1$ does not converge to zero). Note that the convergence in $\Tau$ plus uniform boundedness of a sequence implies its weak star convergence (due to the Helly theorem). Thus, we should have a subsequence $\mu_{k_n}$ which converges in $\Tau$ to zero which is impossible since $\Phi_{\mu_{n_k}}(1/2)=1$ does not tend to zero. So, we see that the problem is deeper than one might expect.
\par
Alternatively, it seems to us that the problem can be solved by passing from the dynamical process on the initial phase space to the so-called trajectory dynamical system which acts on pieces of trajectories and endowed with the proper space-time topology (e.g., the topology of $L^p_{loc}(\R_+,\mathcal E)$ with $1\le p<\infty$), see \cite{CV} and references therein. We return to this problem in the forthcoming paper.
\end{rem}

\section{Asymptotic compactness  and strong uniform attractors}\label{s.str}

In this Section we would like to address the question of existence of a \emph{strong} uniform attractor $\mathcal A_{un}^s$ for equation \eqref{eq.qdw}. By definition, this is the uniform attractor for the dynamical process $U_\mu(t,\tau)$ associated with this equation and acting in the energy phase space $\mathcal E$ endowed with the strong topology, see Definition \ref{Def.n-auto}. In this Section we always assume that
%$$
\begin{equation}
\mu\in M_b^{wna}(\R,H)
\end{equation}
%$$
and, therefore, the {\it weak} uniform attractor $\mathcal A_{un}^w$ always exists and, due to Theorem \ref{Th6.main}, possesses the description \eqref{6.str2}. It is also not difficult to see, the strong uniform attractor if exists coincides with the weak one:
%$$
\begin{equation}
\mathcal A_{un}^s=\mathcal A_{un}^w=\cup_{z\in\mathcal H(\mu)}\mathcal K_z.
\end{equation}
%$$
Moreover, due to Proposition \ref{Prop.non}, to verify the existence of a strong uniform attractor, we only need to check the asymptotic compactness of the process $U_\mu(t,\tau)$. In a fact, it is more convenient for us to check instead the asymptotic compactness of the extended semigroup $\Bbb S(t):\Phi\to\Phi$ acting on the spaces $\Phi:=\mathcal E\times\mathcal H(\mu)$, where the space $\mathcal E$ is endowed with the strong topology (and $\mathcal H(\mu)$ remains endowed with the weak-star topology). Namely, we will verify that for any sequence of $\tau_n\in\R$ such that $\tau_n\to-\infty$ and any sequences $z_n\in\mathcal H(\mu)$ and $\xi_{\tau_n}\in \mathcal B$, the sequence
%$$
\begin{equation}\label{7.ac}
\{U_{z_n}(0,\tau_n)\xi_{\tau_n}\}_{n=1}^\infty
\end{equation}
%$$
is precompact in $\mathcal E$. Due to the translation identity, this implies the asymptotic compactness of the process $U_\mu(t,\tau)$. Actually, since under our conditions the extended semigroup $\Bbb S(t)$ is weakly continuous on $\Phi$ for every fixed $t\ge0$, one can prove that the asymptotic compactness of the semigroup $\Bbb S(t)$ and the process $U_\mu(t,\tau)$ are equivalent, but we will not use this fact below.
\par
Clearly, the only assumption $\mu\in M^{wna}_b(\R,H)$ is not enough to get the strong asymptotic compactness (see examples in \cite{Zntrc}, in particular, as shown there, $\mu\in L^\infty(\R,H)$ is also not enough for compactness even in the case of linear damped wave equation). In order to state our extra assumptions on $\mu$, following \cite{Zntrc}, we introduce the following classes of external forces.

\begin{Def}\label{Def7.ext} Let $\mu\in M_b(\R,H)$. The measure $\mu$ is called {\it space-regular} if there exists a sequence $\mu_n\in M_b(\R,C^\infty(\Omega))$ such that
%$$
\begin{equation}\label{7.app}
\lim_{n\to\infty}\|\mu_n-\mu\|_{M_b(\R,H)}=0.
\end{equation}
%$$
Analogously, the measure $\mu$ is called {\it time-regular} if there exists a sequence $\mu_n\in C^\infty_b(\R,H)$ such that \eqref{7.app} holds (here and below we identify the measure which is absolutely continuous with respect to the Lebesgue measure with its density).
\end{Def}
The following proposition gives the key property of the introduced classes of functions.

 \begin{prop}\label{Prop7.reg} Let $\mu\in M_b(\R,H)$ be space-regular. Then, for every $k\in \N$ and every $\eb>0$, there exists $\bar\mu=\bar\mu_{\eb,k}\in M_b(\R,H^k)$ such that
 %$$
 \begin{equation}\label{7.app1}
 \|\mu-\bar\mu\|_{M_b(\R,H)}\le\eb.
 \end{equation}
 %$$
 Moreover, every measure from $\mathcal H(\mu)$ is space-regular and, for every $z\in\mathcal H(\mu)$ there exists $\bar z\in\mathcal H(\bar\mu)$  such that
 %$$
 \begin{equation}\label{7.app2}
 \|z-\bar z\|_{M_b(\R,H)}\le\eb.
 \end{equation}
 %$$
 Analogously, let $\mu\in M_b(\R,H)$ be time-regular. Then,  for every $k\in \mathbb N$ and every $\eb>0$, there exists $\bar\mu=\bar\mu_{\eb,k}\in H^k_b(\R,H)$ such that \eqref{7.app1} holds. Moreover, every measure from $\mathcal H(\mu)$ is time-regular and, for every $z\in\mathcal H(\mu)$ there exists $\bar z\in\mathcal H(\bar\mu)$  such that \eqref{7.app2} holds.
 \end{prop}
The proof of this proposition is straightforward and is given in \cite{Zntrc}.

\begin{rem}\label{Rem7.reg} More details on the properties of space or time regular functions can be found in \cite{Zntrc}. For instance, any time-regular measure $\mu$ belongs to $L^1_b(\R,H)$ (this follows, e.g., from the Dunford-Pettis theorem, see Theorem \ref{bv.cor3}). In contrast to this, the space-regular measures may have singular component. It is also known that $\mu$ is simultaneously space and time regular if and only if it is translation compact in $L^1_b(\R,H)$.
\par
The typical examples of space or time regular measures are $\mu\in M_b(\R,H^1)$ or $\mu\in C^\alpha_b(\R,H)$, $\alpha>0$ respectively. Typical example of space non-regular measure is
$$
\mu(t)=\sum_{n=1}^\infty\chi_{[n,n+1)}(t)e_n,
$$
where $\{e_n\}_{n=1}^\infty$ is an orthonormal base in $H$, say, generated by the Laplacian and $\chi_A(t)$ is a characteristic function of the set $A$. The example of time non-regular function is even simpler $\mu(t)=\sin (t^2)$. Combining these two examples, we get a measure
$$
\tilde \mu(t)=\sum_{n=1}^\infty\sin(n^2t)\chi_{[n,n+1)}(t)e_n
$$
which is neither space nor time regular. Nevertheless, $\tilde\mu\in M_b^{wna}(\R,H)$ and as elementary calculations show, gives the strong  asymptotic compactness due to the averaging effects. Thus, the introduced conditions are not necessary for the asymptotic compactness. Unfortunately, the necessary and sufficient conditions are not known so far.
\end{rem}
We are now ready to state and prove the main result of this Section.

\begin{theorem}
\label{Th7.main}
Let the assumptions of Theorem \ref{Th6.main} hold and let, in addition, the external force $\mu$ be time-regular or space regular.  Then the  dynamical processes $U_\mu(t,\tau)$ associated to problem \eqref{eq.qdw} possesses a strong uniform attractor $\Cal A_{un}^s$ which coincides with the weak attractor $\mathcal A_{un}^w$ constructed before and admits representations \eqref{6.str1} and \eqref{6.str2}.
\end{theorem}
\begin{proof} As explained before, we only need to verify the asymptotic compactness of the associated process $U_\mu(t,\tau)$ in a strong topology of $\mathcal E$. To this end, it is sufficient to verify the pre-compactness of the sequence \eqref{7.ac}, where $\tau_n\to-\infty$, $\xi_{\tau_n}$ are taken from the uniformly absorbing set $\mathcal B$ and $z_n\in\mathcal H(\mu)$. We will utilize  the so-called
energy method, see \cite{ball,rosa}, which is based on the following elementary fact: let the sequence $\xi_n\rightharpoondown\xi_\infty$ in a Hilbert space $\mathcal E$ and $\|\xi_n\|_\mathcal E\to\|\xi_\infty\|_{\mathcal E}$ than $\xi_n\to\xi_\infty$ strongly. The proof is divided into two natural steps.
\par
{\it Step 1.} At this step we utilize the weak continuity of the processes $U_z(t,\tau)$ and the existence of weak uniform attractor in order to obtain good description of weak limit points of the sequence \eqref{7.ac}. The arguments given below actually reprove the general representation formula \eqref{6.str2} for the case of equation \eqref{eq.qdw}. Nevertheless, we decide to give these arguments here since they are crucial for our proof of asymptotic compactness.
\par
Without loss of generality we may assume that $z_n\to z\in\mathcal H(\mu)$ (in the associated weak star topology). Let us also introduce the solutions which correspond to this sequence
%$$
\begin{equation}
\xi_{u_n}(t)=U_{z_n}(t,\tau_n)\xi_{\tau_n},\ t\geq\tau_n.
\end{equation}
%$$
Then, due to the dissipative estimate \eqref{est.dis-str} and the fact that $\xi_{\tau_n}$ are uniformly bounded, the sequence $\xi_{u_n}(t)$ satisfies
%$$
\begin{equation}\label{7.bound}
\|\xi_{u_n}(t)\|_{\mathcal E}+\|u_n\|_{L^4(t,t+1;L^{12})}\le C,\ \ t\ge\tau_n.
\end{equation}
%$$
In particular, the sequence \eqref{7.ac} is bounded, so passing to the subsequence if necessary, we may assume that
%$$
\begin{equation}\label{7.weak}
\xi_n:=U_{z_n}(0,\tau_n)\xi_{\tau_n}\rightharpoondown \xi_\infty.
\end{equation}
%$$
for some $\xi_\infty\in\mathcal E$.
Moreover, without loss of generality, we may assume also that
%$$
\begin{equation}
\xi_{u_n}\to\xi_u\ \text{ weakly star in $L^\infty_{loc}(\R,\mathcal E)$ and } u_n\to u \text{ weakly in $L^4_{loc}(\R,L^{12})$}
\end{equation}
%$$
to some function $u$ such that $\xi_u\in L^\infty(\R,\mathcal E)$ and $u\in L^4_b(\R,L^{12})$. Passing to the limit $n\to\infty$ in the sense of distributions in equations \eqref{eq.qdw} for $u_n$, we get in a standard way  (see e.g. \cite{KSZ} for the details) that $u$ is a complete bounded solution of \eqref{eq.qdw} with the right-hand side $z\in\mathcal H(\mu)$ and since $z\in M_b^{wna}(\R,H)$, the function $\xi_u(t)$ has no jumps, so $u(t)$ is a Shatah-Struwe solution for \eqref{eq.qdw} and therefore
$$
\xi_u\in \mathcal K_z.
$$
We need to check now that $\xi_u(0)=\xi_\infty$. To this end, we establish some strong convergences for solutions $u_n(t)$ which will be essentially used in Step 2 below. First we note that $u_n$ is bounded in $L^\infty(\R,H^1)$ and $\Dt u_n$ is bounded in $L^\infty(\R,H)$, so by the compactness arguments,
%$$
\begin{equation}\label{7.fun}
u_n\to u \text{ strongly in }\ C_{loc}(\R,H).
\end{equation}
%$$
The analogous result for $\Dt u_n(t)$ is a bit more delicate since in contrast to the standard case, $\Dt^2 u_n$ are not functions, but measures. To overcome this problem, we derive from \eqref{6.intdef} that
%$$
\begin{multline}
\|\Dt u_n(t)-\Dt u_n(\tau)\|_{H^{-1}}\le \int_\tau^t(\|u_n(s)\|_{H^1}+\|f(u_n(s))\|_{H^{-1}})\,ds+\\+\gamma\|u_n(t)-u_n(\tau)\|_{H^{-1}}+
\|z_n([\tau,t])\|_{H^{-1}}\le C|t-\tau|+\|z_n([\tau,t])\|_{H^{-1}},
\end{multline}
%$$
where we have implicitly used that $\xi_{u_n}(t)$ is bounded in $\mathcal E$ and that
$$
\|f(u_n)\|_{H^{-1}}\le C\|f(u_n)\|_{L^{6/5}}\le C(1+\|u_n\|^6_{L^6})\le C(1+\|u_n\|_{H^1}^6).
$$
Moreover, since $\mu\in M_b^{wna}(\R,H)$, there exists a monotone function $\omega:\R_+\to\R_+$ such that $\lim_{x\to0}\omega(x)=0$ and
%$$
\begin{equation}
\|z([\tau,t])\|_{H^{-1}}\le \omega(|t-\tau|),\ \ z\in\mathcal H(\mu).
\end{equation}
%$$
Thus,
$$
\|\Dt u_n(t)-\Dt u_n(\tau)\|_{H^{-1}}\le C|t-\tau|+\omega(|t-\tau|),
$$
and the functions $\Dt u_n(t)$ are equi-continuous as functions with values in $H^{-1}$. Since they are also bounded as functions in $H$, the Arzela theorem gives us that
%$$
\begin{equation}
\Dt u_n\to\Dt u \ \ \text{strongly in }\ C_{loc}(\R,H^{-1}).
\end{equation}
%$$
Thus, $\xi_{u_n}\to \xi_u$ strongly in $C_{loc}(\R,\mathcal E^{-1})$ and, particularly,
%$$
\begin{equation}\label{7.str}
\xi_n=\xi_{u_n}(0)=U_{z_n}(0,\tau_n)\xi_{\tau_n}\rightharpoondown \xi_u(0)=\xi_\infty.
\end{equation}
%$$

\emph{Step 2.} At this step we verify that $\|\xi_{u_n}(0)\|_{\mathcal E}\to \|\xi_u(0)\|_{\mathcal E}$  by passing to the limit in the appropriate  energy equality. Crucial for this method is the fact that, under the assumption that $\mu\in M_b^{wna}(\R,H)$, any Shatah-Struwe solution of equation \eqref{eq.qdw} satisfies the energy equality, see Corollary \ref{Cor4.energy}. Thus, the validity of taking the scalar product of the equation \eqref{eq.qdw} with $\Dt u$ is justified  and testing this equation with $u$ does not require any extra justification. By this reason, we may multiply (following to \cite{KSZ}) equation \eqref{eq.qdw} for the solution $u_n(t)$ by $\Dt u_n+\delta u_n$ where $\delta>0$ is small enough to get
%$$
\begin{equation}\label{7.en}
\frac{d}{dt}\E(\xi_{u_n})+\delta\E(\xi_{u_n})+B(\xi_{u_n})+\\
\delta\Big((f(u_n),u_n)-(F(u_n),1)\Big)=(z_n,\Dt u_n+\delta u_n),
\end{equation}
%$$
where
%$$
\begin{equation}
 \E(\xi_u)=\frac{1}{2}\|\xi_{u}\|^2_\E+\frac{\delta\gamma}{2}\|u\|^2_H+\delta(\Dt u, u)+(F(u),1),\ F(u)=\int_0^uf(v)\,dv
\end{equation}
%$$
and
%$$
\begin{equation}
B(\xi_u)=\left(\gamma-\frac{3\delta}{2}\right)\|\Dt u\|^2_H-\delta^2(\Dt u,u)+\left(\frac{\delta}{2}-\frac{\gamma\delta^2}{2}\right)\|u\|^2_H+\frac{\delta}{2}\|\nabla u(s)\|^2_H.
\end{equation}
%$$
Multiplying \eqref{7.en} by $e^{\delta t}$ and integrating the obtained identity in time from $\tau_n$ to $0$ we get the energy identity in the following integral form
%$$
\begin{multline}\label{7.inten}
\E(\xi_{u_n})(0)+
\int_{-\infty}^{0}e^{\delta s}\left(B(\xi_{u_n})(s)+\delta\Big((f(u_n(s)),u_n(s))-(F(u_n(s)),1)\Big)\right)\,ds=\\
e^{\delta\tau_n}\E(\xi_{\tau_n})+\int_{-\infty}^0e^{\delta s}(\Dt u_n(s),z_n(ds))+\delta \int_{-\infty}^0e^{\delta s}(u_n(s),z_n(ds)),
\end{multline}
%$$
where, to avoid dependence on $n$ in the lower limit of integration, we  set $\xi_{u_n}(s)\equiv 0$ for $s<\tau_n$.
\par
We want to pass to the  limit $n\to\infty$ in \eqref{7.inten}. To this end, we first note that the weak convergence $\xi_{u_n}(0)\to\xi_{u}(0)$ in $\mathcal E$ and the compactness of the embedding $H^1\subset H$ imply that
%$$
\begin{equation}
\frac{\delta\gamma}{2}\|u(0)\|^2+\delta(\Dt u(0), u(0))=\lim_{n\to\infty}\left(\frac{\delta\gamma}{2}\|u_n(0)\|^2+\delta(\Dt u_n(0), u_n(0))\right).
\end{equation}
%$$
In order to pass to the limit in the terms containing the non-linearity, we recall
that $f(u)$ has a positive coefficient in front of the leading quintic term, see \eqref{4.f}.  Therefore,
%$$
\begin{equation}
1.\ F(s)\geq -C,\ s\in\R,\qquad 2.\ f(s)s-F(s)\geq -C,\ s\in\R,
\end{equation}
%$$
for some $C=C_f$. Moreover, the strong convergence $u_n(0)\to u(0)$ implies the convergence almost everywhere (passing to a subsequence if necessary). This allows to apply the Fatou lemma and get
%$$
\begin{equation}
(F(u(0)),1)\leq\liminf_{n\to\infty}(F(u_n(0)),1).
\end{equation}
%$$
Analogously, using the strong convergence $u_n\to u$ in $C_{loc}(\R,H)$ and the boundedness of $u_n$ in $L^\infty(\R,H^1)$, we arrive at
%$$
\begin{multline}
\int_{-\infty}^0e^{\delta s}\Big((f(u(s)),u(s))-(F(u(s)),1)\Big)\,ds \leq\\
\liminf_{n\to\infty}\int_{-\infty}^0e^{\delta s}\Big((f(u_n(s)),u_n(s))-(F(u_n(s)),1)\Big)\,ds.
\end{multline}
%$$
Next, for small enough $\delta=\delta(\gamma)>0$ the quadratic form $B$ is positive definite and hence is convex and weakly lower semicontinuous, therefore
%$$
\begin{equation}
\int_{-\infty}^0e^{\delta s}B(\xi_u(s))\,ds\leq \liminf_{n\to\infty}\int_{-\infty}^0e^{\delta s}B(\xi_{u_n}(s))\,ds.
\end{equation}
%$$
Let us now look at the right-hand side of \eqref{7.inten}.
Since $\xi_{\tau_n}$ are bounded in $\E$ by the assumption and $\tau_n$ tends to $-\infty$ the first term on the right hand side of  vanishes.
\par
Moreover, since $z_n$ and $u_n$ are bounded in $M_b(\R,H)$ and $L^\infty(\R,H)$ respectively and $u_n\to u$ strongly in $C_{loc}(\R,H)$, we have
%$$
\begin{equation}
\int_{-\infty}^0e^{\delta s}(u(s),z(ds))=\lim_{n\to\infty}\int_{-\infty}^0e^{\delta s}(u_n(s),z_n(ds))\,ds.
\end{equation}
%$$
Here we also used that $z_n\to z$ weakly star in $M_{loc}(\R,H)$ as well as $\mu\in M_b^{wna}(\R,H)$ (in order to guarantee that $z_n\big|_{t\le0}\to z\big|_{t\le0}$ weakly star in $M_{loc}(-\infty,0;H)$).
\par
Up to the moment, we have nowhere used that $\mu$ is time or space regular. This will be essentially used in order to pass to the limit in the second term in the right-hand side of \eqref{7.inten}, namely, to show that
%$$
\begin{equation}\label{7.bad}
\int_{-\infty}^0e^{\delta s}(\Dt u(s),z(ds))=\lim_{n\to\infty}\int_{-\infty}^0e^{\delta s}(\Dt u_n(s),z_n(ds)).
\end{equation}
%$$
Assume for the moment that \eqref{7.bad} is verified and complete the proof of the theorem. Indeed, passing to a subsequence if necessary, we may assume that
$$
\limsup_{n\to\infty}\|\xi_{u_n}(0)\|_{\mathcal E}=\lim_{n\to\infty}\|\xi_{u_n}(0)\|_{\mathcal E}.
$$
Then, taking $\liminf_{n\to\infty}$ from both sides of \eqref{7.inten} and using the inequalities obtained above together with the fact that
$$
\liminf_{n\to\infty}(A_n+B_n)\ge\liminf_{n\to\infty}A_n+\liminf_{n\to\infty}B_n,
$$
we arrive at
%$$
\begin{multline}\label{unEid3}
\limsup_{n\to\infty}\frac{1}{2}\|\xi_{u_n}(0)\|^2_\E+\frac{\delta\gamma}{2}\|u(0)\|^2+\delta(\Dt u(0),u(0))+(F(u(0)),1)+\\
\int_{-\infty}^{0}e^{\delta s}\left(B(\xi_{u})(s)+\delta\Big((f(u(s)),u(s))-(F(u(s)),1)\Big)\right)\,ds\leq\\
\int_{-\infty}^0e^{\delta s}(\Dt u(s)+\delta u(s),z(ds)).
\end{multline}
%$$
On the other hand, since $u$ is \emph{Shatah-Struwe} solution of the limit problem,  it also obeys energy equality
%$$
\begin{multline}\label{unEid4}
\frac{1}{2}\|\xi_{u}(0)\|^2_\E+\frac{\delta\gamma}{2}\|u(0)\|^2+\delta(\Dt u(0),u(0))+(F(u(0)),1)+\\
\int_{-\infty}^{0}e^{\delta s}\left(B(\xi_{u})(s)+\delta\Big((f(u(s)),u(s))-(F(u(s)),1)\Big)\right)\,ds=\\
\int_{-\infty}^0e^{\delta s}(\Dt u(s)+\delta u(s),z(ds)).
\end{multline}
%$$
Combining \eqref{unEid3}, \eqref{unEid4} with weak lower semi continuity of $\|\cdot\|_\E$ we get the chain of inequalities
%$$
\begin{equation}
\limsup_{n\to\infty}\|\xi_{u_n}(0)\|^2_\E\leq\|\xi_{u}(0)\|^2_\E\leq \liminf_{n\to\infty}\|\xi_{u_n}(0)\|^2_\E,
\end{equation}
%$$
that implies the equality
%$$
\begin{equation}
\|\xi_{u}(0)\|^2_\E=\lim_{n\to\infty}\|\xi_{u_n}(0)\|^2_\E,
\end{equation}
%$$
which together with the already proved weak convergence $\xi_{u_n}(0)\rightharpoondown\xi_u(0)$ proves the strong convergence.  Thus, in order to finish the proof of theorem, we only need to verify identity \eqref{7.bad}. This is done in the following lemma.
\begin{lemma} Let $\mu\in M_b^{wna}(\R,H)$ be a measure which is either time or space regular. Assume also that the sequence of functions $\xi_{u_n}\in C_b(\R,\mathcal E)$ be uniformly bounded and that $\xi_{u_n}\to\xi_u$ strongly in $C_{loc}(\R,\mathcal E^{-1})$. Then, equality \eqref{7.bad} holds for every sequence $z_n\in\mathcal H(\mu)$ such that $z_n\to z$ weakly star in $M_{loc}(\R,H)$.
\end{lemma}
\begin{proof}[Proof of the Lemma] Let $\mu$ be time regular. Then, according to Proposition \ref{Prop7.reg}, for every $\eb>0$, there exists $\bar\mu\in H^2_b(\R, H)$ and measures $\bar z_n\in \mathcal H(\bar\mu)$ such that
%$$
\begin{equation}\label{7.close}
\|\mu-\bar\mu\|_{M_b(\R,H)}\le\eb,\ \ \|z_n-\bar z_n\|_{M_b(\R,H)}\le\eb.
\end{equation}
%$$
Moreover, since the hull $\mathcal H(\bar\mu)$ is compact in a weak topology of $H^2_{loc}(\R,H)$, we may also assume that $\bar z_n\rightharpoondown \bar z\in\mathcal H(\bar\mu)$ weakly in $H^2_{loc}(\R,H)$. In particular,
$$
\|z-\bar z\|_{M_b(\R,H)}\le\eb.
$$
Since the functions $\xi_{u_n}(t)$ are bounded in $L^\infty(\R,\mathcal E)$, we have
%$$
\begin{multline}\label{7.close2}
|\int_{-\infty}^0e^{\delta s}(\Dt u_n(s),z_n(ds))-\int_{-\infty}^0e^{\delta s}(\Dt u_n(s),\bar z_n(ds))|+\\+|\int_{-\infty}^0e^{\delta s}(\Dt u(s), z(ds))-\int_{-\infty}^0e^{\delta s}(\Dt u(s),\bar z(ds))|\le C\eb.
\end{multline}
%$$
Thus, we only need to prove that
%$$
\begin{equation}\label{7.conv}
\int_{-\infty}^0e^{\delta s}(\Dt u(s),\bar z(ds))=\lim_{n\to\infty}\int_{-\infty}^0e^{\delta s}(\Dt u_n(s),\bar z_n(ds)).
\end{equation}
%$$
To verify this we utilize the fact that $\bar z_n$ is smooth in time and that $u_n\to u$ strongly in $C_{loc}(\R,H)$, so we may integrate by parts and get
%$$
\begin{multline}
\int_{-\infty}^0e^{\delta s}(\Dt u_n(s),\bar z_n(ds))=(\bar z_n(0),u_n(0))-\int_{-\infty}^0e^{\delta s}(\bar z_n'(s)+\delta \bar z_n(s),u_n(s))\,ds\to\\\to (\bar z(0),u(0))-\int_{-\infty}^0 e^{\delta s}(\bar z'(s)+\delta \bar z(s),u(s))\,ds=\int_{-\infty}^0e^{\delta s}(\Dt u(s),\bar z(ds))
\end{multline}
%$$
and the lemma is proved in the case where $\mu$ is time regular.
\par
Assume now that $\mu$ is space regular. Then, analogously to the time regular case, we may approximate the measure $\mu$ by $\bar \mu\in M_b(\R,H^1)$ and fix $\bar z_n\in H(\bar\mu)$ in such a way that \eqref{7.close} and \eqref{7.close2} hold. And again, the desired convergence would be proved if we check \eqref{7.conv}. However, since we do not assume that $\bar\mu\in M_b^{wna}(\R,H^1)$, this convergence may be broken and we need to proceed in a more accurate way. Namely, let $\beta>0$ be a small number and
%$$
\begin{equation}
\theta_\beta(t):=\begin{cases}1,\ t\le0,\\ 0,\ t\ge\beta,\\ 1-\beta^{-1}t,\ t\in[0,\beta].\end{cases}
\end{equation}
%$$
Then, since $\bar z_n\to\bar z$ weakly star in $M_{loc}(\R,H^1)$ and $\Dt u_n\to\Dt u$ strongly in $C_{loc}(\R,H^{-1})$, for every $\beta>0$, we have
%$$
\begin{equation}\label{7.conv1}
\int_{\R}e^{\delta s}(\theta_\beta(s)\Dt u(s),\bar z(ds))=\lim_{n\to\infty}\int_{\R}e^{\delta s}(\theta_\beta(s)\Dt u_n(s),\bar z_n(ds)).
\end{equation}
%$$
Thus, to prove the convergence, we need to estimate
%$$
\begin{multline}
|\int_{[0,\beta]}e^{\delta s}[(\theta_\beta(s)\Dt u_n(s),\bar z_n(ds))-(\theta_\beta(s)\Dt u(ds),\bar z(s))]|\le\\\le \|\bar z_n\|_{M_b(\R,H^1)}e^{\delta\beta}\|\Dt u_n-\Dt u\|_{C(0,\beta;H^{-1})}+|\int_{[0,\beta]}e^{\delta s}(\theta_\beta(s)\Dt u(s),\bar z_n(ds)-\bar z(ds))|.
\end{multline}
%$$
The first term in the right-hand side tends to zero as $n\to\infty$ and due to \eqref{7.close} the second term satisfies
$$
|\int_{[0,\beta]}e^{\delta s}(\theta_\beta(s)\Dt u(s),\bar z_n(ds)-\bar z(ds))|\le
|\int_{[0,\beta]}e^{\delta s}(\theta_\beta(s)\Dt u(s),z_n(ds)-z(ds))|+C\eb,
$$
where the constant $C$ is independent of $n$. Thus, we only need to prove that
%$$
\begin{equation}\label{7.fconv}
\lim_{\beta\to0}\int_0^\beta e^{\delta s} (\theta_\beta(s)\Dt u(s),z(ds))=0
\end{equation}
%$$
uniformly with respect to all $z\in\mathcal H(\mu)$. Moreover, since $z$ in non-atomic, the function $\Dt u(s)$ is continuous as a function with values in $H$, we only need to prove that
$$
\lim_{\beta\to0}\(\Dt u(0),\int_0^\beta(1-\beta^{-1}s)z(ds)\)=0.
$$
Finally, integration by parts together with the fact that $\mu\in M_b^{wna}(\R,H)$ give
$$
|\(\Dt u(0),\int_0^\beta(1-\beta^{-1}s)z(ds)\)|=|\beta^{-1}\int_0^\beta(\Phi_z(s),\Dt u(0))\,ds|\le
\sup_{s\in[0,\beta]}|(z([0,s]),\Dt u(0))|\le \omega_{\Dt u(0)}(\beta).
$$
Thus, the convergence \eqref{7.fconv} is verified and the lemma is proved. The theorem is also proved.
\end{proof}
\end{proof}

\section{Smoothness of uniform attractors}\label{s.sm}

The aim of this Section is to verify that the uniform attractor $\mathcal A_{un}$ of the damped wave equation \eqref{eq.qdw} is more regular if the external force $\mu\in M_b(\R,H)$ is more regular. We consider two model cases of extra regularity for $\mu$, namely,
%$$
\begin{equation}\label{8.mtime}
\partial_t\mu\in M_b(\R,H)
\end{equation}
%$$
or
%$$
\begin{equation}\label{8.mspace}
\mu\in M_b(\R,H^\alpha)
\end{equation}
%$$
for some (small) positive $\alpha$. The main result of this Section is the following theorem.

\begin{theorem}\label{Th8.main} Let the assumptions of theorem \ref{Th.wnon-attr} hold and let in addition the measure $\mu$ satisfies \eqref{8.mtime} or \eqref{8.mspace}. Then, the dynamical process $U_\mu(t,\tau)$ associated with equation \eqref{eq.qdw} possesses the strong uniform attractor $\mathcal A_{un}^s$ in the phase space $\mathcal E$ (which coincides with the weak uniform attractor $\mathcal A_{un}^w$ constructed in Theorem \ref{Th.wnon-attr}) and this attractor is bounded in the space $\mathcal E^\alpha:=H^{\alpha+1}\times H^{\alpha}$ for some small $\alpha>0$:
%$$
\begin{equation}\label{8.atsm}
\|\mathcal A_{un}\|_{\mathcal E^\alpha}\le C.
\end{equation}
%$$
\end{theorem}
 \begin{rem}\label{Rem8.rem} Note that \eqref{8.mtime} together with the assumption $\mu\in M_b(\R,H)$ implies that $\mu$ is a function of bounded variation with values in $H$
 $$
 \mu\in BV_b(\R,H).
 $$
 In particular, $\mu\in M_b^{wna}(\R,H)$ and therefore the uniform attractor possesses representations \eqref{6.str1} and \eqref{6.str2}. In contrast to this, in the case where \eqref{8.mspace} is satisfied, the measure $\mu$ may contain discrete part and \eqref{6.str2} is not necessarily satisfied.
 \end{rem}
To prove the theorem, we split the solution $u$ into three parts
%$$
\begin{equation}\label{8.3}
u(t)=\theta(t)+v(t)+w(t),
\end{equation}
%$$
where $\theta(t)$ solves the linear wave equation
%$$
\begin{equation}\label{8.lin}
\Dt^2\theta+\gamma\Dt\theta+(1-\Dx)\theta=\mu,\ \ \xi_\theta\big|_{t=\tau}=0,
\end{equation}
%$$
the function $v$ solves the following auxiliary nonlinear problem
%$$
\begin{equation}\label{8.dec}
\Dt^2 v+\gamma\Dt v+(1-\Dx)v+f(v)+Lv=0,\ \ \xi_v\big|_{t=\tau}=\xi_u\big|_{t=\tau},
\end{equation}
%$$
where $L>0$ is a sufficiently big number, and the reminder $w$ solves the following problem with zero initial conditions:
%$$
\begin{equation}\label{8.reg}
\Dt^2w+\gamma\Dt w+(1-\Dx)w+[f(\theta+v+w)-f(v)]=Lv,\ \ \xi_w\big|_{t=\tau}=0.
\end{equation}
%$$
We need to obtain good estimates for every of three functions $\theta$, $v$ and $w$. We start with the simplest case of $\theta$ which satisfies the linear equation.
\begin{lemma}\label{Lem8.theta} Let the above assumptions hold and let $\mu$ satisfies either \eqref{8.mspace} or \eqref{8.mtime}. Then the solution $\theta$ of equation \eqref{8.lin} satisfies
%$$
\begin{equation}\label{8.thetagood}
\|\xi_{\theta}(t)\|_{\mathcal E^\alpha}+\|\theta\|_{L^4([t,t+1],H^{\alpha,12})}\le C\|\mu\|_W,
\end{equation}
%$$
where $\alpha>0$ is small enough and the symbol $W$ means the space $BV_b(\R,H)$ (if \eqref{8.mtime} is satisfied) or $M_b(\R,H^\alpha)$ (if \eqref{8.mspace} is satisfied).
\end{lemma}
\begin{proof} Indeed, in the case of conditions \eqref{8.mspace}, estimate \eqref{8.thetagood} is an immediate corollary of Theorem \ref{lwmf.thE!} and estimate \eqref{2.str} applied to the function $\bar\theta:=(-\Dx+1)^{\alpha/2}\theta$ and the elliptic regularity.
\par
Let now assumption \eqref{8.mtime} be satisfied. Then, differentiating equation \eqref{8.lin} in time and denoting $\bar\theta:=\Dt\theta$, we get
$$
\Dt^2\bar\theta+\gamma\Dt\bar\theta-(\Dx-1)\bar\theta=\Dt\mu,\ \ \xi_{\bar\theta}\big|_{t=\tau}=\{0,\mu(\tau)\}.
$$
Since, $\mu(\tau)$ is well-defined and $\|\mu(\tau)\|_{H}\le C\|\mu\|_W$, we may apply Theorem \ref{lwmf.thE!} and estimate \eqref{2.str} to this equation and get
$$
\|\xi_{\bar\theta}(t)\|_{\mathcal E}+\|\bar\theta\|_{L^4([t,t+1],L^{12})}\le C\|\mu\|_W,\ \ t\ge\tau.
$$
Using now that $(-\Dx+1)\theta=-\Dt\bar\theta(t)-\gamma\bar\theta+\mu(t)$, we derive that the function $\xi_{\theta}(t)$ is bounded in $\mathcal E^1$. Finally, using that $H^2\subset H^{3/4,12}$, we see that estimate \eqref{8.thetagood} is satisfied at least for $\alpha\le3/4$. Of course, the bound $\alpha\le3/4$ is artificial and can be easily removed, but the validity of \eqref{8.thetagood} for some small positive $\alpha$ is enough for our purposes. Thus, the lemma is proved.
\end{proof}
At the next step we show that the function $v(t)$ decays exponentially as $t\to\infty$.
\begin{lemma}\label{Lem8.vgood} Let the above assumptions hold. Then, the solution $v(t)$ satisfies the following estimate:
%$$
\begin{equation}\label{8.vgood}
\|\xi_v(t)\|_{\mathcal E}+\|v\|_{L^4(t,t+1;L^{12})}\le Q(\|\xi_u(\tau)\|_{\mathcal E})e^{-\delta(t-\tau)},\ \ t\ge\tau,
\end{equation}
%$$
where the positive constant $\delta$ and the monotone function $Q$ are independent of $t$, $\tau$ and $v$.
\end{lemma}
\begin{proof} Multiplying equation \eqref{8.dec} by $\Dt v+\delta v$, where $\delta>0$ is small enough and arguing in a standard way, we obtain the analogue of the identity \eqref{7.en}, where $z_n=0$ and the non-linearity $f$ is replaced by $f_L(u):=f(u)+Lu$. Since the non-linearity $f$ satisfies \eqref{4.f}, one can verify that, for a sufficiently large $L$,
$$
F_L(u)\ge0,\ \ f_L(u).u-F_L(u)\ge0
$$
and \eqref{7.en} reads
$$
\frac d{dt}\mathcal E(\xi_v)+\delta\mathcal E(\xi_v)\le0.
$$
Applying the Gronwall inequality and using that
$$
\frac14\|\xi_v\|^2_{\E}\le\mathcal E(\xi_v)\le Q(\|\xi_v\|_{\E}),
$$
we end up with the desired estimate for the energy norm of $v$. To get the control of the Strichartz norm, we apply energy to Strichartz estimate \eqref{L5L10.nh} to equation \eqref{8.dec} and get
$$
\|v\|_{L^4(t,t+1;L^{12})}\le Q(\|\xi_v(t)\|_{\E})\le Q(\|\xi_v(\tau)\|_{\E}).
$$
Next we utilize again the fact that $f(0)=0$ which together with the fact that $f$ has no more than quintic growth rate gives us the control
$$
\|f_L(u)\|_{L^1(t,t+1;H)}\le C(1+\|u\|_{L^4(t,t+1;L^{12})}^4)\|\xi_v\|_{L^\infty(t,t+1;\E)}\le Q(\|\xi_u(\tau)\|_{\E})e^{-\delta(t-\tau)}.
$$
Finally, treating the term $f_L(u)$ as an external force and  applying the Strichartz estimate to the obtained linear equation, we arrive at the desired decaying Strichartz estimate for $v$ and finish the proof of the lemma.
\end{proof}
We now ready to treat the most complicated $w$-component of the solution $u$. We do this in two steps: at the first step we get an exponentially growing in time estimate which will be refined at the second step.
\begin{lemma}\label{Lem8.bad} Let the above assumptions hold. Then the solution $w(t)$ satisfies the following estimate:
%$$
\begin{equation}\label{8.bad}
\|\xi_w(t)\|_{\E^\alpha}+\|w\|_{L^4(\tau,t;H^{\alpha,12})}\le \(Q(\|\xi_u(\tau)\|_{\E})+Q(\|\mu\|_W)\)e^{K(t-\tau)},
\end{equation}
%$$
where $\alpha\in(0,\frac25)$ is sufficiently small positive exponent and the monotone functions $K=K(\|\xi_u(\tau)\|_\E+\|\mu\|_W)$ and $Q$ are independent of $t$, $\tau$, and of the concrete choices of  $u$ and $\mu$.
\end{lemma}
\begin{proof} We treat the non-linearity in equation \eqref{8.reg} as an external force and apply the $\E^\alpha$ energy and Strichartz estimate to this linear equation to get
%$$
\begin{multline}\label{8.main}
\|\xi_w(t)\|_{\E^\alpha}+\(\int_\tau^te^{-4\delta(t-s)}\|w(s)\|^4_{H^{\alpha,12}}\,ds\)^{1/4}\le\\\le C\int_\tau^te^{-\delta(t-s)}\(\|f(\theta+v+w)-f(v)\|_{H^\alpha}+L\|v(s)\|_{H^1}\)\,ds,
\end{multline}
%$$
where $\delta>0$ is a sufficiently small number, see \eqref{1.main}.
To estimate the non-linear term we use the key inequality \eqref{A2.f} which gives
%$$
\begin{multline}\label{8.rough}
\|f(\theta+v+w)-f(v)\|_{H^\alpha}\le \\\le C\(1+\|\theta+w\|_{L^{12}}+\|v\|_{L^{12}}\)^{4-\alpha}\(1+\|\theta+w\|_{H^1}+\|v\|_{H^1}\)^{\alpha}
\|\theta+w\|_{H^{1+\alpha}}^{1-\alpha}\|\theta+w\|^\alpha_{H^{\alpha,12}}
\end{multline}
%$$
and 
$$
\|\theta+w\|_{H^{1+\alpha}}^{1-\alpha}\|\theta+w\|^\alpha_{H^{\alpha,12}}\le \|\theta\|_{H^{1+\alpha}}^{1-\alpha}\|\theta\|^\alpha_{H^{\alpha,12}}+
\|w\|_{H^{1+\alpha}}^{1-\alpha}\|w\|^\alpha_{H^{\alpha,12}}+
\|\theta\|_{H^{1+\alpha}}^{1-\alpha}\|w\|^\alpha_{H^{\alpha,12}}+\|w\|_{H^{1+\alpha}}^{1-\alpha}\|\theta\|^\alpha_{H^{\alpha,12}}.
$$
Using the H\"older inequality together with the control \eqref{8.thetagood} for the $\theta$-component, we arrive at
%$$
\begin{multline}\label{8.huge}
\int_\tau^te^{-\delta(t-s)}\(1+\|\theta+w\|_{L^{12}}+\|v\|_{L^{12}}\)^{4-\alpha}\(1+\|\theta+w\|_{H^1}+\|v\|_{H^1}\)^{\alpha}
\|w\|_{H^{1+\alpha}}^{1-\alpha}\|w\|^\alpha_{H^{\alpha,12}}\,ds\le\\\le C\(\int_\tau^te^{-\delta(t-\tau)\frac{1-\alpha}{1-\alpha/4}}(1+\|\theta+w\|^4_{L^{12}}+\|v\|^4_{L^{12}})\right.\times\\\times\left.
(1+\|\theta+w\|_{H^1}+\|v\|_{H^1})^{\frac\alpha{1-\alpha/4}}
\|\xi_w(s)\|_{\E^\alpha}^{\frac{1-\alpha}{1-\alpha/4}}\,ds\)^{1-\alpha/4}
\(\int_\tau^te^{-4\delta(t-s)}\|w(s)\|^4_{H^{\alpha,12}}\,ds\)^{\alpha/4}\le\\\le
\frac18\(\int_\tau^te^{-4\delta(t-s)}\|w(s)\|^4_{H^{\alpha,12}}\,ds\)^{1/4}+
C\(\int_\tau^te^{-\delta'(t-\tau)}l_{\theta,v,w}(s)
\|\xi_w(s)\|_{\E^\alpha}^{\frac{1-\alpha}{1-\alpha/4}}\,ds\)^{\frac{1-\alpha/4}{1-\alpha}},
\end{multline}
%$$
where $\delta'=\delta\frac{1-\alpha}{1-\alpha/4}$ and
$$
l_{\theta,v,w}(s):=(1+\|\theta(s)+w(s)\|^4_{L^{12}}+\|v(s)\|^4_{L^{12}})
(1+\|\theta(s)+w(s)\|_{H^1}+\|v(s)\|_{H^1})^{\frac\alpha{1-\alpha/4}}.
$$
Estimation of three other terms in \eqref{8.rough} which contain $H^{1+\alpha}$ and $H^{\alpha,12}$ norms of $\theta$ is analogous, but even simpler due to the control \eqref{8.thetagood}. 
According to the already obtained estimates, we have
%$$
\begin{equation}\label{8.lbad}
\int_{t}^{t+1}l_{\theta,v,w}(s)\,ds\le Q_1=Q_1(\|\xi_u(0)\|_{\E}+\|\mu\|_{W}),\ \ t\ge\tau
\end{equation}
%$$
and inserting \eqref{8.huge} into the right-hand side of \eqref{8.main}, we arrive at
%$$
\begin{multline}\label{8.gron1}
\|\xi_w(t)\|_{\E^\alpha}+\(\int_\tau^te^{-4\delta(t-s)}\|w(s)\|^4_{H^{\alpha,12}}\,ds\)^{1/4}\le\\\le
C\(\int_\tau^te^{-\delta'(t-\tau)}l_{\theta,v,w}(s)
\|\xi_w(s)\|_{\E^\alpha}^{\frac{1-\alpha}{1-\alpha/4}}\,ds\)^{\frac{1-\alpha/4}{1-\alpha}}+Q_2,
\end{multline}
%$$
where the constant $Q_2$ depends only on $\|\xi_u(0)\|_{\E}$ and $\|\mu\|_W$. Introducing $Y(t):=\|\xi_w(t)\|_{\E^\alpha}^{\frac{1-\alpha}{1-\alpha/4}}$ and taking power $\frac{1-\alpha}{1-\alpha/4}$ from both sides of \eqref{8.gron1}, we finally get
%$$
\begin{equation}\label{8.gr}
Y(t)\le Q+C\int_\tau^te^{-\delta'(t-s)}l_{\theta,v,w}(s)Y(s)\,ds
\end{equation}
%$$
for some new constant $Q$ depending on $\xi_u(0)$ and $\|\mu\|_W$.
The Gronwall inequality applied to this estimate together with \eqref{8.lbad} give the desired estimate \eqref{8.bad} and finish the proof of the lemma.
\end{proof}
We now state (following \cite{ZCPAA2004}) a  corollary of the obtained estimates which is crucial for what follows.
\begin{cor}\label{Cor8.split} Let the above assumptions hold and let $\xi_u(\tau)\in\mathcal B$ where $\mathcal B$ is a uniform absorbing set for equation \eqref{eq.qdw}. Then, for every $\eb>0$ there exists a splitting of the solution $w(t)$ of problem \eqref{8.reg} $w(t)=\bar w(t)+\tilde w(t)$ such that
%$$
\begin{equation}\label{8.small}
\int_s^t\|\tilde w(\kappa)\|^4_{L^{12}}\,ds+\int_s^t\|\xi_{\tilde w}(\kappa)\|_{\E}^{\frac\alpha{1-\alpha/4}}\,d\kappa\le C_\eb+\eb(t-s),\ \ t\ge s\ge\tau
\end{equation}
%$$
and
%$$
\begin{equation}\label{8.smooth}
\|\xi_{\bar w}(t)\|_{\E^\alpha}+\|\bar w\|_{L^4(t,t+1;H^{\alpha,12})}\le C_\eb,\ \ t\ge\tau,
\end{equation}
%$$
where the constant $C_\eb$ depends only on $\eb$ (and is independent of $t,s,\tau$ and $\xi_u(\tau)\in\mathcal B$). Moreover,
%$$
\begin{equation}\label{8.bound}
\|\xi_{\bar w}(t)\|_{\E}+\|\bar w\|_{L^4(t,t+1;L^{12})}+\|\xi_{\tilde w}(t)\|_{\E}+\|\tilde w\|_{L^4(t,t+1;L^{12})}\le C,\ \ t\ge\tau,
\end{equation}
%$$
where $C$ is independent also of $\eb$.
\end{cor}
\begin{proof} Note that, due to estimates \eqref{8.vgood} and \eqref{8.thetagood}, it is sufficient to construct the desired splitting $u(t)=\bar u(t)+\tilde u(t)$ only. To do this we fix $T=T(\eb)$ (actually $T\sim\frac1\eb$) and construct splitting \eqref{8.3} at $\tau_0=\tau$, $\tau_1=\tau+T$, $\tau_2=\tau+2T$, etc. Namely, denote by $\theta_n(t)$, $v_n(t)$, $w_n(t)$ the solutions of problems
\eqref{8.lin}, \eqref{8.dec} and \eqref{8.reg} respectively where the initial time moment $\tau$ is replaced by $\tau+nT$ and define
%$$
\begin{equation}
\tilde u(t):=v_n(t),\ \ \bar u(t):=\theta_n(t)+w_n(t),\ \  t\in[\tau+nT,\tau+(n+1)T).
\end{equation}
%$$
Then, as elementary calculations based on \eqref{8.vgood} show, the function $\tilde u(t)$ satisfies \eqref{8.small} and \eqref{8.bound}. In turn, estimates \eqref{8.thetagood} and \eqref{8.bad} together with the dissipative estimate for the solution $u(t)$ guarantee that the function $\bar u$ satisfies \eqref{8.smooth} and \eqref{8.bound}. Finally, in order to obtain the desired splitting of $w$, we just need to fix
$$
\tilde w(t)=\tilde u(t)-v(t),\ \ \bar w(t)=\bar u(t)-\theta(t)
$$
and the corollary is proved.
\end{proof}
We are now ready to refine Lemma \ref{Lem8.bad}.
\begin{lemma}\label{Lem8.good1} Let the above assumptions hold and let $\xi_u(\tau)\in\mathcal B$. Then the solution $w$ of problem \eqref{8.reg} satisfies
%$$
\begin{equation}\label{8.wgood}
\|\xi_{w}(t)\|_{\E^\alpha}+\|w\|_{L^4(t,t+1;H^{\alpha,12})}\le C,
\end{equation}
%$$
where the constant $C$ is independent of $t$, $\tau$ and $\xi_u(\tau)\in\mathcal B$.
\end{lemma}
\begin{proof} We refine estimates \eqref{8.rough} and \eqref{8.huge} using the result of Corollary \ref{Cor8.split}. To this end, we first note that we may assume without loss of generality that $f'(0)=0$. Indeed, the extra term $|f'(0)|\|w(t)+\theta(t)\|_{H^\alpha}$ is easily controllable by the obtained before estimates in the energy norm $\E$. Next, we write the difference $f(\theta+v+w)-f(v)$ as follows
%$$
\begin{multline}\label{8.ssplit}
|f(\theta+v+w)-f(v)|\le |f(\theta+v+w)-f(v+\tilde w)|+|f(v+\tilde w)-f(v)|\le\\\le |f(\theta+v+w)-f(v+\tilde w)|+|\int_0^1f'(v+\kappa\tilde w)\,d\kappa\, w|+|\int_0^1f'(v+\kappa\tilde w)\,d\kappa\, \bar w|.
\end{multline}
%$$
The first term in the right-hand side of \eqref{8.ssplit} is controlled exactly as in \eqref{8.rough}
%$$
\begin{multline}
\|f(\theta+v+w)-f(v+\tilde w)\|_{H^\alpha}\le C(1+\|\theta+\bar w\|_{L^{12}}+\|v+\tilde w\|_{L^{12}})^{4-\alpha}\times\\\times
(1+\|\theta+\bar w\|_{H^1}+\|v+\tilde w\|_{H^1})^{4-\alpha}\|\theta+\bar w\|_{H^{1+\alpha}}^{1-\alpha}\|\theta+\bar w\|_{H^{\alpha,12}}^{\alpha}\le\\\le C_\eb(1+\|\theta\|_{H^{\alpha,12}}^4+\|v\|_{L^{12}}^4+\|\bar w\|_{H^{\alpha,12}}^4+\|\tilde w\|_{L^{12}}^4).
\end{multline}
%$$
The third term is estimated analogously using \eqref{est.main}:
%$$
\begin{multline}
\|\int_0^1f'(v+\kappa\tilde w)\bar w\,d\kappa\|_{H^\alpha}\le C(1+\|v+\tilde w\|_{L^{12}}+\|v\|_{L^{12}})^{4-\alpha}\times\\\times
(1+\|v+\tilde w\|_{H^1}+\|v\|_{H^1})^{4-\alpha}\|\bar w\|_{H^{1+\alpha}}^{1-\alpha}\|\bar w\|_{H^{\alpha,12}}^{\alpha}\le\\\le C_\eb(1+\|v\|_{L^{12}}^4+\|\tilde w\|_{L^{12}}^4+\|\bar w\|_{H^{\alpha,12}}^4).
\end{multline}
%$$
Thus, due to estimates \eqref{8.smooth}, \eqref{8.thetagood},\eqref{8.vgood} and \eqref{8.bound}, we have
%$$
\begin{multline}
\int_\tau^t e^{-\delta(t-s)}\|f(\theta(s)+v(s)+w(s))-f(v(s)+\tilde w(s))\|_{H^\alpha}\,ds+\\+\int_\tau^t e^{-\delta(t-s)}\|\int_0^1f'(v(s)+\kappa\tilde w(s))\bar w(s)\,d\kappa\|_{H^\alpha}\, ds\le C_\eb.
\end{multline}
%$$
So, we only need to estimate the second term in the right-hand side of \eqref{8.ssplit}. To this end, we utilize that $f'(0)=0$ and apply estimate \eqref{est.fp0} to get
$$
\|\int_0^1f'(v(s)+\kappa\tilde w(s)) w(s)\,d\kappa\|_{H^\alpha}\le C(1+\|v\|_{L^{12}}+\|\tilde w\|_{L^{12}})^{4-\alpha}(\|v\|_{H^1}+\|\tilde w\|_{H^1})^{\alpha}\|w\|_{\E^\alpha}^{1-\alpha}\|w\|_{H^{\alpha,12}}^\alpha
$$
and arguing as in \eqref{8.rough}, we get
%$$
\begin{multline}
\int_\tau^t e^{-\delta(t-s)}\|\int_0^1f'(v(s)+\kappa\tilde w(s))\bar w(s)\,d\kappa\|_{H^\alpha}\,ds\le\\\le
\frac 12\(\int_\tau^te^{4\delta(t-s)}\|w(s)\|^4_{H^{\alpha,12}}\,ds\)^{1/4}+\(\int_\tau^t e^{-\delta'(t-s)}l(s)
\|\xi_w(s)\|_{\E^\alpha}^{\frac{1-\alpha}{1-\alpha/4}}\,ds\)^{\frac{1-\alpha/4}{1-\alpha}},
\end{multline}
%$$
where
%$$
\begin{multline}
l(s):=C(1+\|v(s)\|_{L^{12}}+\|\tilde w(s)\|_{L^{12}})^4(\|v(s)\|_{H^1}+\|\tilde w(s)\|_{H^1})^{\frac{\alpha}{1-\alpha/4}}\le\\\le C(\|v(s)\|_{L^{12}}^4+\|\tilde w(s)\|^4_{L^{12}}+\|\xi_{\tilde w}(s)\|_{\E}^{\frac\alpha{1-\alpha/4}}+\|\xi_{v}(s)\|_{\E}^{\frac\alpha{1-\alpha/4}}).
\end{multline}
%$$
Important that the constant $C$ here is independent of $\eb$. Therefore, due to \eqref{8.small} and \eqref{8.vgood}, we have
%$$
\begin{equation}\label{8.exp}
\int_s^tl(\kappa)\,d\kappa\le C_\eb+\eb(t-s),\ \ t\ge s\ge\tau
\end{equation}
%$$
and inserting the obtained estimates to inequality \eqref{8.main}, we finally get that the function $Y(t)=\|\xi_w(t)\|_{\E^{\alpha}}^{\frac{1-\alpha}{1-\alpha/4}}$ satisfies the refined analogue of \eqref{8.gr}
%$$
\begin{equation}\label{8.fine}
Y(t)+\(\int_\tau^te^{-4\delta(t-s)}\|w(s)\|^4_{H^{\alpha,12}}\,ds\)^{1/4}\le C_\eb+ \int_\tau^te^{-\delta'(t-s)}l(s)Y(s)\,ds.
\end{equation}
%$$
To derive the desired estimate \eqref{8.wgood} from \eqref{8.fine}, we need the following version of the Gronwall lemma.
\begin{lemma}\label{Lem8.gr} Let the function $Y\in C_{loc}([\tau,\infty))$ satisfies
$$
Y(t)\le C+\int_\tau^te^{-\delta'(t-s)}l(t)Y(s)\,ds,\ \ \ t\ge\tau
$$
for some  constants $C$ and $\delta$ and non-negative function $l(t)\ge0$ such that $l\in L^1_{loc}([\tau,\infty))$. Then, the following estimate holds:
%$$
\begin{equation}\label{8.gr1}
Y(t)\le C\(1+\int_\tau^t e^{-\delta'(t-s)+\int_s^tl(\kappa)\,d\kappa}l(s)\,ds\).
\end{equation}
%$$
\end{lemma}
The proof of this lemma follows word by word to the proof of the usual Gronwall lemma and by this reason is omitted.
Applying estimate \eqref{8.gr1} to inequality \eqref{8.fine} and using \eqref{8.exp} (with the parameter $\eb$ fixed in such a way that $\eb<\delta'$), we derive the desired estimate \eqref{8.wgood} and finish the proof of the lemma.
\end{proof}
Now we are ready to complete the proof of the main theorem.
\begin{proof}[Proof of Theorem \ref{Th8.main}] Indeed, according to estimates \eqref{8.thetagood}, \eqref{8.vgood} and \eqref{8.wgood}, the set
$$
\mathcal B_\alpha:=\{\xi\in\E^\alpha\,:\, \|\xi\|_{\E^\alpha}\le R\}
$$
is a compact uniformly attracting set for the process $U_\mu(t,\tau)$ in $\E$ if $R$ is large enough. Thus, the process $U_\mu(t,\tau)$ is uniformly asymptotically compact and possesses a uniform attractor $\mathcal A_{un}$ in the strong topology of $\E$. Moreover, $\mathcal A\subset \mathcal B_\alpha$. Thus, Theorem \ref{Th8.main} is proved.
\end{proof}
The next corollary gives the global well-posedness and dissipativity of the process $U_\mu(t,\tau)$ in the higher energy space $\E^\alpha$.

\begin{cor}\label{Cor8.sm} Let the assumptions of Theorem \ref{Th8.main} hold and let, in addition, $\xi_u(\tau)\in\mathcal E^\alpha$. Then, the solution $u$ of equation \eqref{eq.qdw} satisfies $\xi_u(t)\in\E^\alpha$ for all $t\ge\tau$ and the following estimate holds:
%$$
\begin{equation}\label{8.usm}
\|\xi_u(t)\|_{\E^\alpha}+\|u\|_{L^4(t,t+1;H^{\alpha,12})}\le Q(\|\xi_u(\tau)\|_{\E^\alpha})e^{-\delta(t-\tau)}+Q(\|\mu\|_{W}),
\end{equation}
%$$
where the constant $\delta>0$ and the monotone function $Q$ are independent of $t$, $\tau$, $\mu$ and $u$.
\end{cor}
Indeed, the proof of this estimate is based on the result of Corollary \ref{Cor8.split} and can be obtained analogously to the derivation of estimate \eqref{8.wgood} (and even simpler since we may take $v(t)=0$ and put the initial conditions directly to the $w$ component). By this reason, we left the detailed proof of this corollary to the reader.

\begin{rem} Note that Theorem \ref{Th8.main} and Corollary \ref{Cor8.sm} are formally proved under the assumption $\alpha<\frac25$ and we do not know how to obtain more regularity of the attractor $\mathcal A_{un}$ {\it in one step} even in the case where $\mu$ and $f$ are smooth. For instance, it would be interesting to get $\E^1$ regularity without the usage of fractional spaces. The problem is related to the restriction on the exponent $\alpha$ in the key Lemma \ref{Lem.v4wHa}. However, the higher regularity can be easily obtained in several steps using the standard bootstrapping arguments. Moreover, the most difficult step is exactly the first one: to obtain the $\E^\alpha$ regularity of solutions and since the non-linearity is no more critical in $\E^\alpha$, one can use the linear decomposition to improve further the regularity. Namely, in \eqref{8.3} we may take $v=0$ and $\xi_{\theta}\big|_{t=\tau}=\xi_u\big|_{t=\tau}$. For instance, if $\alpha>\frac18$, we have
$$
\|f(u)\|_{L^1(t,t+1;H^1)}\le C(1+\|u\|_{L^4(t,t+1;H^{\alpha,12})}^4)\|\xi_u\|_{L^\infty(t,t+1;\E^\alpha)}
$$
and, therefore, we need only one extra step to get the $\E^1$-regularity of the attractor.
\end{rem}

\section{Appendix 1: BV-functions and vector measures}
\label{s.bv}
In this Appendix we recall a number of more or less standard results concerning functions of bounded variation (BV-functions) with values in Banach spaces and the associated measures which are used throughout of the paper. We restrict ourselves to consider only the case where these functions are with values in a separable Hilbert space $H$, see \cite{Bg1,Bg2,Di,Ma} and references therein for more detailed exposition.

\begin{Def}\label{bv.def1} A function $\Phi:[a,b]\to H$ is BV if
\begin{equation}
\Var_a^b(\Phi,H):=\sup\left\{\sum_{i=1}^N\|\Phi(t_j)-\Phi(t_{j-1})\|_H\right\}<\infty,
\end{equation}
where the supremum is taken over all finite $N$ and all partitions $a=t_0<t_1<\cdots<t_N=b$ of the segment $[a,b]$.
\end{Def}
 We also recall the elementary properties of the introduced variation:
\par
1) $\Var_a^b(\Phi,H)=\Var_a^c(\Phi,H)+\Var_c^b(\Phi,H)$ if $a<c<b$;
\par
2) $\Var_a^b(\Phi_1+\Phi_2,H)\le\Var_a^b(\Phi_1,H)+\Var_a^b(\Phi_2,H)$;
\par
3) $\Var_a^b(\alpha\Phi,H)=|\alpha|\Var_a^b(\Phi,H)$;
\par
4) The function $\varphi(t):=\Var_a^t(\Phi,H)$ is monotone nondecreasing and $\Var_x^y(\Phi,H)=\varphi(y)-\varphi(x)$;
\par
5) The function $\Phi(t)$ is continuous (left/right semicontinuous) at $t=t_0$ if and only if the same is true for $\varphi(t)$.
\par
Note that if $\Phi$ is BV then $\varphi$ is a {\it scalar nondecreasing and bounded} function, so it is continuous up to at most countable set of points (indeed, number of jumps of $\varphi$ which are larger than $1/n$, for every $n$, must be finite). Therefore the fifth property guarantees that the function $\Phi(t)$ also has at most countable number of discontinuities. Furthermore, due to monotonicity, right/left limits $\varphi(t+0)$ and $\varphi(t-0)$ exist for all $t\in[a,b]$, that, by simple arguments, implies existence of right/left limits at every point of $[a,b]$ for the original function $\Phi$.
\begin{Def}
We denote by $V_0(a,b;H)$ the Banach space all BV-functions $\Phi$ on $[a,b]$ with values in $H$ such that $\Phi(a)=0$ and $\Phi(t)$ is left-semicontinuous at every \emph{internal} point of $[a,b]$ (that is $\Phi$ may have a jump at $b$). The norm in this space is given by
$$
\|\Phi\|_{V_0}:=\Var_a^b(\Phi,H).
$$
\end{Def}

As usual,  every $\Phi\in V_0$, defines a vector-valued  measure $\mu_\Phi$ on the semiring generated by intervals of $[a,b]$ via
%$$
\begin{equation}\label{mu-phi}
\begin{aligned}
\mu_\Phi((s,t]):=\Phi(t+0)-\Phi(s+0),\ \ \mu_\Phi([s,t]):=\Phi(t+0)-\Phi(s-0),\\
 \mu_\Phi((s,t)):=\Phi(t-0)-\Phi(s+0),\ \ \mu_\Phi([s,t)):=\Phi(t-0)-\Phi(s-0),
\end{aligned}
\end{equation}
%$$
where $a\leq s\leq t\leq b$, and we use notations $(t,t]=[t,t)=(t,t)=\emptyset$, $[t,t]={t}$ and $\Phi(b+0):=\Phi(b)$, $\Phi(a-0):=0$. As usual, the left-continuity assumption implies $\sigma$-additivity of $\mu_\Phi$ on the algebra generated by intervals. Moreover, this measure can be extended in a unique way to the $\sigma$-algebra of Borel sets of $[a,b]$ and gives a $\sigma$-additive vector  measure $\mu_\Phi$ of finite total variation $|\mu_\Phi|([a,b])<\infty$. We recall that for any Borel set $A\subset [a,b]$,
%$$
\begin{equation}\label{def.|mu|}
|\mu_\Phi|(A):=\sup\left\{\sum_{n=1}^\infty\|\mu_\Phi(A_n)\|_{H}\right\},
\end{equation}
%$$
where the supremum is taken over all countable disjoint Borel partitions of $A$. It is also known that $|\mu_\Phi|$ is a scalar positive $\sigma$-additive measure generated by the function $\varphi(t):=\Var_a^t(\Phi,H)$, that is formulas \eqref{mu-phi} are valid if one substitutes $\mu_\Phi$ by $|\mu_\Phi|$ on the left and $\Phi$ by $\phi$ on the right. In particular, we have
%$$
\begin{equation}
\label{bv.|mu|=var}
|\mu_\Phi|([a,b])=\varphi(b+0)-\varphi(a-0)=\Var_a^b(\Phi,H).
\end{equation}
%$$
Vice-versa, with every $\sigma$-additive $H$-valued measure $\mu$ defined on Borel $\sigma$-algebra of $[a,b]$ with finite total variation $|\mu|([a,b])<\infty$ we can associate a BV function $\Phi_\mu:[a,b]\to H$ from $V_0([a,b];H)$ given by
%$$
\begin{equation}
\label{bv.Fmu}
\Phi_\mu(t)=\begin{cases}\mu([a,t)),\ a\le t< b,\\
\mu([a,b]),\ t=b.
\end{cases}
\end{equation}
%$$
Thus, there is an isomorphism between the space $M(a,b;H)$ of Borel measures \emph{with bounded total variation} endowed with the norm
$$
\|\mu\|_{M(a,b,H)}:=|\mu|([a,b])=\Var_a^b(\Phi_\mu,H)
$$
 and the space $V_0(a,b;H)$.

\begin{comment}
\begin{rem}\label{bv.rem.BVm}
We note that every finite \emph{scalar} signed measure $\mu$ on $[a,b]$ automatically has finite total variation. That is in the \emph{scalar} case the class of Lebesgue-Stieltjes measures does not constitute any special subclass of finite measures, rather this notion draws attention to the fact that the measure under consideration is generated by some BV-function.

However, this is not the case for \emph{vector-valued} measures. For instance one can consider the following example
\begin{equation}
\mu:\Cal B\to L^2([0,1]),\ \quad \mu(A)=\chi_A(t),\ \text{for all }A\subset\Cal B,
\end{equation}
where $\Cal B$ is Borel $\sigma$-algebra on $[0,1]$. Then, obviously, $\mu$ is a finite vector-valued measure. But since
\begin{equation}
\|\mu(A)\|_{L^2([0,1])}=|A|^\frac{1}{2},
\end{equation}
it is easy to see that
\begin{equation}
|\mu|([0,1])=+\infty.
\end{equation}
\end{rem}
\end{comment}

Furthermore, for every $\mu\in M(a,b;H)$ and every $\mu$-measurable function $f:[a,b]\to H$ such that $\int_a^b\|f(t)\|_H|\mu|(dt)<\infty$, the Lebesgue integral $\int_{[a,b]}(f(t),\mu(dt))_H$ is well-defined. As usual, it is first defined on simple functions
$$
f(t)=\sum_{j=1}^Nc_j\chi_{A_j}(t),
$$
where $c_j\in H$ and $\mu$-measurable sets $A_j$ form a disjoint partition of $[a,b]$, via
%$$
\begin{equation}
\int_{[a,b]}(f(t),\mu(dt))_H=\sum_{j=1}^N(c_j,\mu(A_j))_H.
\end{equation}
%$$
Then it can be extended in a standard way to any integrable function $f:[a,b]\to H$, see \cite{Ma} for the details.
\par
On the other hand, for every $\mu\in M(a,b;H)$, we can also consider Riemann-Stieltjes integral $\int_a^b(f(t),d\,\Phi(t))$  as a limit of Riemann integral sums
%$$
\begin{equation}
\int_a^b(f(t),d\Phi_\mu(t))_H:=\lim_{\Delta t\to0}\Big((f(c_n),\Phi_\mu(b-0)-\Phi(t_{n-1}))_H+\sum_{j=1}^{n-1}(f(c_j),\Phi_\mu(t_{j})-\Phi_\mu(t_{j-1}))_H\Big),
\end{equation}
%$$
where the limit is taken over all partitions $a=t_0<t_1<\cdots<t_n=b$, points $c_j\in[t_{j-1},t_j)$ and $\Delta t:=\max_{j}{|t_j-t_{j-1}|}$ (the first term in the Riemann sum is properly modified in order to preserve the additivity of the integral). It is well-known that Riemann-Stieltjes integral exists at least for every continuous function $f\in C(a,b;H)$ and when exists it coincides with the Lebesgue (Lebesgue-Stieltjes) integral, namely
%$$
\begin{equation}
\int_{[a,b)}(f(t),\mu(dt))_H=\int_a^b(f(t),d\Phi_\mu(t))_H.
\end{equation}
%$$
In addition, by additivity of the  Lebesgue integral, for every segment $[x,y]\subset[a,b]$, we have
%$$
\begin{multline}
\int_{[x,y]}(f(t),\mu_\Phi(dt))_H=\int_{\{y\}}(f(t),\mu_\Phi(dt))_H+\int_{[x,y)}(f(t),\mu_\Phi(dt))_H=\\=
\left(f(y),\mu_\Phi(\{y\})\right)_H+\int_x^y(f(t),d\Phi(t))_H=(f(y),\Phi(y+0)-\Phi(y-0))+\int_x^y(f(t),d\Phi(t))_H.
\end{multline}
%$$
at least for continuous functions $f:[a,b]\to\R$.

We now recall that, by the standard properties of the Lebesgue integral, we have the following inequality:
%$$
\begin{equation}\label{bv.|Im|<I|m|}
\left|\int_{[a,b]}(f(t),\mu_\Phi(dt))_H\right|\leq \int_{[a,b]}\|f(t)\|\,|\mu_\Phi|(dt).
\end{equation}
%$$
In particular, for $f\in C(a,b;H)$ it reads
%$$
\begin{equation}
\left|\int_{[a,b]}(f(t),\mu_\Phi(dt))_H\right|\leq \|f\|_{C(a,b;H)}|\mu_\Phi|([a,b])=\|f\|_{C(a,b;H)}\Var_a^b(\Phi,H),
\end{equation}
%$$
 Thus, for every $\Phi\in M(a,b;H)$, the linear functional
 $$
 L_\Phi(f):=\int_{[a,b]}(f(t),\mu_\Phi(dt))_H
 $$
 is a bounded linear functional on $C(a,b;H)$ and $\|L_\Phi\|\leq \Var_a^b(\Phi,H)$. Actually, analogously to the scalar case (see \cite{KF}), the following version of Riesz-Representation Theorem holds, see \cite{Di} for details.

\begin{theorem} \label{bv.C*=M}
Let $H$ be a separable Hilbert space and $[a,b]\subset\R$. Then for any linear continuous functional $L\in (C(a,b;H))^*$ there exists a function $\Phi\in V_0(a,b;H)$, such that
%$$
\begin{equation}
L(f)=\int_{[a,b]}(f(t),\mu_\Phi(dt))_H,\ \text{for all }f\in C(a,b;H),
\end{equation}
%$$
and $\|L\|=\Var_a^b(\Phi,H)$. In other words
%$$
\begin{equation}
(C(a,b;H))^*=M(a,b;H)=V_0(a,b;H).
\end{equation}
%$$
\end{theorem}
We now recall the concept of absolute continuity and related Radon-Nikodym theorem for vector measures. For simplicity, we will consider the case of a Hilbert space only where the Radon-Nikodym property is always satisfied.
\begin{Def}
Let $H$ be a  Hilbert space and $[a,b]\subset\R$. A measure $\mu \in M(a,b;H)$ is absolutely continuous with respect to a scalar Borel measure $\nu \in M(a,b;\R)$, $\nu\geq 0$ if the \emph{scalar} measure $|\mu|$ given by \eqref{def.|mu|} is absolutely continuous with respect to $\nu$. The latter means that $|\mu|(A)=0$ for every Borel set $A$ such that $\nu(A)=0$. We say that $\mu$ is absolutely continuous ($\mu\in M^{ac}(a,b;H)$) if it is absolutely continuous with respect to the Lebesgue measure on $[a,b]$.
\end{Def}
The vector valued analogue of the Radon-Nikodym theorem then reads, see, e.g., \cite{Ma}.

\begin{theorem}[Radon--Nikodym]\label{th.v-RN}
Let $H$ be a separable Hilbert space, a segment $[a,b]\subset\R$ and a measure $\mu\in M(a,b;H)$. Then the measure $\mu$ is absolutely continuous with respect to a scalar positive Borel measure $\nu\in M(a,b;\R)$, if and only if there exists a function $\rho\in L^1_\nu(a,b;H)$ such that
%$$
\begin{equation}\label{th.v-RN1}
\mu(A)=\int_A\rho(s)\nu(ds),
\end{equation}
%$$
for every Borel set $A\subset [a,b]$ and the integral in the RHS is understood as a Bohner integral. Furthermore, we have
%$$
\begin{equation}\label{th.v-RN2}
|\mu|(A)=\int_A\|\rho(s)\|_H\nu(ds)
\end{equation}
%$$
and, at least for continuous functions $f:[a,b]\to H$,
%$$
\begin{equation}\label{9.simpler}
\int_A(f(t),\mu(dt))_H=\int_A(f(t),\rho(t))_H\nu(dt)
\end{equation}
%$$
for every Borel set $A\subset [a,b]$.
\end{theorem}

There are two particular cases of this theorem which are of our particular interest. The first one is when $\nu=|\mu|$. Clearly, that every measure $\mu\in M(a,b;H)$ is absolutely continuous with respect to $|\mu|$. Therefore, in the case of a \emph{separable} Hilbert space $H$, we can apply the Radon-Nikodym Theorem
and conclude that there exists such a function $\rho_\mu\in L^1_{|\mu|}(a,b;H)$ such that
%$$
\begin{equation}\label{mu=rho|mu|}
\mu(A)=\int_A\rho_\mu(s)|\mu|(ds).
\end{equation}
%$$
Moreover, from \eqref{th.v-RN2} we see that
$$
\|\rho_\mu(s)\|_H=1,
$$
$|\mu|$-almost everywhere.
\par
The above formulas allow us to express vector-valued measures in terms of  scalar measures and integrable functions. In particular, based on \eqref{mu=rho|mu|}, we derive the following approximation result which is crucial for our study of  measure driven PDEs.

\begin{lemma}\label{bv.lem.|QNmu|to0}
Let $H$ be a separable Hilbert space with an orthonormal basis $\{e_i\}_{i=1}^\infty$, a segment $[a,b]\subset\R$ and a measure $\mu\in M(a,b;H)$. And let $P_N$ be an orthonormal projector on the subspace generated by the first vectors $\{e_i\}_{i=1}^N$, $Q_N=Id-P_N$. Then
\begin{equation}
\lim_{N\to\infty}|Q_N\mu|([a,b])=\lim_{N\to\infty}|\mu-P_N\mu|([a,b])=0.
\end{equation}
\end{lemma}
\begin{proof}
Indeed, applying the projector $Q_N$ to \eqref{mu=rho|mu|} and using \eqref{th.v-RN2} we find
\begin{equation}
\lim_{N\to\infty}|Q_N\mu|([a,b])=\lim_{N\to\infty}\int_a^b\|Q_N\rho_\mu(s)\|\,|\mu|(ds)=0,
\end{equation}
by Lebesgue dominated convergence theorem.
\end{proof}
The next more standard particular case is when $\nu$ is a Lebesgue measure on $[a,b]$. In this case, absolutely continuous measures can be characterised via the analogous property of the corresponding distribution functions.

\begin{Def}
A function $\Phi\in V_0(a,b;H)$ is absolutely continuous, $\Phi\in AC(a,b;H)$, iff for every $\eb>0$ there exists $\delta>0$ such that for every  finite sequence of pairwise disjoint sub-intervals $[x_l, y_j]$ of $[a,b]$ satisfying the condition
$$
\sum_j|x_j-y_j|<\delta,\ \ \text{we have}\ \ \sum_j\|\Phi(x_j)-\Phi(y_j)\|_H<\eb.
$$
Then, as not difficult to see, a measure $\mu\in M_{ac}(a,b;H)$ if and only if its distribution function $\Phi_\mu\in AC(a,b;H)$. We recall that, by the definition
$$
\Phi_\mu(t):=\mu([a,t)),\ t\in[a,b).
$$
\end{Def}

As a corollary of the Radon-Nikodym theorem, similar to the scalar case, the following standard result holds.

\begin{theorem}\label{th.AC}
A function $\Phi:[a,b]\to H$ is absolutely continuous if and only if  there exists a density $g(t)\in L^1(a,b;H)$ such that the following equality holds
%$$
\begin{equation}\label{bv.NLform}
\Phi(t)=\Phi(a)+\int_a^tg(s)\,ds,\ t\in[a,b].
\end{equation}
%$$
Furthermore, for every $\Phi\in AC(a,b;H)$ written in the form \eqref{bv.NLform} we have the identity
%$$
\begin{equation}
\Var_a^b(\Phi,H)=\int_a^b\|g(t)\|_H\,dt.
\end{equation}
%$$
\end{theorem}

\begin{rem}\label{rem.AC}
By the properties of the Bohner integral, every $\Phi\in AC(a,b;H)$ is differentiable a. e. on $[a,b]$ and $\Phi'(t)=g(t)$, where $g$ is from \eqref{bv.NLform}. Moreover, this point-wise derivative coincides with the distributional derivative of $\Phi$.
\par
Recall also  that the point-wise derivative $\Phi'\in L^1(a,b;H)$ exists for any BV function $\Phi$ with values in a separable Hilbert space, but the analogue of Newton-Leibnitz formula \eqref{bv.NLform} holds if and only if $\Phi$ is absolutely continuous. Exactly as in the scalar case any BV function can be uniquely decomposed into three parts
$$
\Phi(t)=\Phi_d(t)+\Phi_{sing}(t)+\Phi_{ac}(t),
$$
where $\Phi_d(t)$ is a discrete part (step function), $\Phi_{sing}(t)$ is a singular part (continuous, but satisfying $\Phi_{sing}'(t)=0$ a.e.) and an absolutely continuous part which satisfies the Newton-Leibnitz formula
$$
\Phi_{ac}(t)-\Phi_{ac}(a)=\int_{a}^t\Phi'(s)\,ds=\int_a^t\Phi'_{ac}(s)\,ds
$$
and the analogous decomposition holds for the associated measures.
\end{rem}

\begin{rem}
Recall also the standard  integration by parts formula
%$$
\begin{equation}\label{bv.ByParts}
\int_{[x,y]}(f(t),\mu_{\Phi}(dt))_H+\int_{[x,y]}(f'(t),\Phi(t))_H\,dt=(f(y),\Phi(y+0))_H-(f(x),\Phi(x))_H
\end{equation}
%$$
which holds for every $f\in AC(a,b;H)$, $\Phi\in V_0(a,b;H)$ and every $[x,y]\subset[a,b]$. Of course, if $x=b$ the quantity $\Phi(b+0)$ should be substituted by $\Phi(b)$. Actually, this formula remains true (after replacing the second term in the LHS by $\int_{[x,y]}(\Phi(t), \mu_f(dt))_H$) if the function $f$ is BV with zero discrete part ($f_d(t)\equiv0$). However, it is not straightforward when both $f$ and $\Phi$ have non-zero discrete parts, since an extra accuracy is required to define properly the integral of the step function with respect to the Dirac measure, see e.g., \cite{Hew}. It is also worth to mention that, according to \eqref{bv.ByParts}, the distributional derivative of the function $\Phi_\mu(t)$ is exactly the measure $\mu$:
$$
\Phi_\mu'(t)=\mu(t).
$$
\end{rem}

We now discuss the relations between the space $L^1(a,b;H)$ and $M(a,b;H)$. For any $g\in L^1(a,b;H)$, we define the distribution function
$$
\Phi_g(t):=\int_a^tg(s)\,ds.
$$
Then, obviously, the function $\Phi_g$ is absolutely continuous and, therefore, the associated measure $\tilde \mu_g:=\mu_{\Phi_g}$ is also absolutely continuous. Therefore, due to the Radon-Nikodym theorem
$$
\tilde\mu_g(A)=\int_Ag(t)\,dt \ \text{ and }\ \ |\tilde \mu|(A)=\int_A\|g(t)\|_H\,dt
$$
for every measurable $A\subset[a,b]$. In particular,
$$
\|\tilde\mu_g\|_{M([a,b],H)}=|\tilde\mu_g|([a,b])=\int_a^b\|g(t)\|_H\,dt=\|g\|_{L^1(a,b;H)}.
$$
Thus, the map $g\to\tilde\mu_g$ is an isometric embedding of the space $L^1(a,b;H)$ into the space $M(a,b;H)$ and the range of this linear operator is exactly the space of absolutely continuous measures. This allows us to identify the integrable functions $g\in L^1(a,b;H)$ with regular (absolutely continuous measures).

The advantage of this embedding is that $M(a,b;H)$ is dual to separable Banach space $C(a,b;H)$ and, consequently, its unit ball $B_M$ is weakly-star compact, so after this embedding, the unit ball $B_{L^1}$ becomes weakly-star pre-compact and (since this topology is metrizable on a unit ball), we can naturally identify  weak-star limit points of bounded sequences in $L^1(a,b;H)$ with vector measures of \emph{finite total variation}. Namely, the following statement holds.

\begin{prop}\label{prop.[L1]*=M} Let $H$ be a separable Hilbert space, $[a,b]\subset\R$ and $B_M$ be a unit ball in the space $M(a,b;H)$ endowed with total variation norm $\|\mu\|_{M(a,b;H)}=|\mu|([a,b])$. Then we have
%$$
\begin{equation}
B_M = [B_M\cap M_{ac}([a,b];H)]^{w^*}=[B_{L^1}]^{w^*},
\end{equation}
%$$
where $[\ \cdot\ ]^{w^*}$ means the  closure in the weak-star topology of $M(a,b;H)$.
\end{prop}
\begin{proof} Let $\mu=\mu_{\Phi}\in B_M$. We approximate the distribution function $\Phi$ by the standard mollification procedure $\Phi_n=\theta_n*\bar\Phi$, where the positive kernels $\theta_n$ approximate the $\delta$-function and
%$$
\begin{equation}
\bar\Phi(t)=\begin{cases}
	0,\ t<a;\\
	\Phi(t),\ t\in[a,b];\\
	\Phi(b),\ t>b.
            \end{cases}
\end{equation}
%$$
Then, obviously, $\Var_a^b(\Phi,H)=\Var_{-\infty}^\infty(\bar \Phi,H)$ and
%$$
\begin{equation}
\Var_a^b(\Phi_n,H)\le\Var_{-\infty}^\infty(\Phi_n,H)\le \Var_{-\infty}^\infty(\bar \Phi,H)=\Var_a^b(\Phi,H).
\end{equation}
%$$
Thus, $\mu_{\Phi_n}\in B_M$ and since they are smooth, $\mu_{\Phi_n}\in B_{L^1}$. Moreover, without loss of generality we may assume that $\Phi(t)$ is left/right continuous at the endpoints $t=a$ and $t=b$ (otherwise, we subtract the corresponding endpoint $\delta$-measures and approximate them separately using the one-sided approximating sequences).

Let $f\in C(a,b;H)$ be arbitrary. We need to prove that
%$$
\begin{equation}
\int_a^b(f(t),\mu_{\Phi_n}(d\,t)-\mu_{\Phi}(d\,t))_H\to0\ \text{as }n\to\infty.
\end{equation}
%$$
Since the variations of $\Phi_n$ are \emph{uniformly} bounded with respect to $n$, it is enough to verify the convergence for $f\in C^1(a,b;H)$. In this case, we may integrate by parts to get
%$$
\begin{equation}
\int_a^b(f(t),\mu_{\Phi_n}(d\,t)-\mu_{\Phi}(d\,t))_H=
(f(b),\Phi_n(b)-\Phi(b))_H-\int_a^b(f'(t),\Phi_n(t)-\Phi(t))_H\,dt.
\end{equation}
%$$
Then, from construction of $\Phi_n$, it is easy to see that $\Phi_n(t)$ tends to $\Phi(t)$ at all points of continuity of $\Phi$.
This fact implies convergence of $(f(b),\Phi_n(b))_H$ to $(f(b),\Phi(b))_H$ and together with Lebesgue dominated convergence theorem we also have convergence to $0$ of the integral on the right. This proves the proposition.
\end{proof}

\begin{rem}\label{Cor9.conv} Without loss of generality, we may assume also that $\Phi_n(t)\to\Phi(t)$ for all  $t\in[a,b)$ (including the jump points). Indeed, using the fact that $\Phi_n(t)$ are continuous and choosing the left-sided kernels $\theta_n(t)$ (i.e., such that $\operatorname{supp}\theta_n\subset(-\infty,0]$), we get the convergence $\Phi_n(t)\to\Phi(t)$ for all $t\in[a,b)$. Thus, we may assume that $\mu_{\Phi_n}([a,t))\to \mu_{\Phi}([a,t))$ for all $t\in[a,b]$.
\end{rem}

As the next step, we  recall the  characterization of weak-star convergence in the case of \emph{scalar signed} measures (see \cite{Bg2}, Proposition 8.1.8\footnote{In \cite{Bg2} the author, following the tradition coming from Probability Theory, uses the notion of weak convergence of measures which coincides with the notion of weak-star convergence we use here, following the terminology from Functional Analysis.}) which is usually referred as the Helly selection theorem.

\begin{theorem}\label{bv.th.w*sc}
A sequence of \emph{scalar signed} measures $\mu_n$ on the segment $[a,b]\subset \R$ converges weakly-star in $M(a,b;\R)$ to a measure $\mu$ precisely when
\par
1. $\sup_n\|\mu_n\|_{M(a,b;\R)}<\infty$;
\par
2. Every subsequence $\Phi_{n_k}$ in the sequence of distribution functions $\Phi_{\mu_n}$ of the measures $\mu_n$ contains a further subsequence $\Phi_{n_{k_m}}$ convergent to $\Phi_\mu$ everywhere except of at most countable set depending on the subsequence $\Phi_{n_{k_m}}$.
\par
In the case when measures $\mu_n,\ \mu$ are nonnegative the second condition can be changed to
\par
2'. The whole sequence $\Phi_{\mu_n}$ converges to the function $\Phi_\mu$ at continuity points of $\Phi_{\mu}$.
\end{theorem}

\begin{rem}\label{bv.rem.2}
We would like to emphasise that weak-star convergence of \emph{signed} measures, in contrast to the case of nonnegative measures, does not imply point-wise convergence of the corresponding distribution functions on a dense set. In this respect we mention the example from \cite{Bg2}. On segment $[0,1]$ we can consider the sequence of measures $\mu_n=\delta_{x_n}-\delta_{y_n}$, where the sequence of segments $[x_n,y_n]$ is formed from gliding segments $[k2^{-m},(k+1)2^{-m}]$, where $k\in\overline{0,2^{m}-1}$ for each $m\in \N$. It is easy to see that $\mu_n$ converges to $0$ weakly-star in $M(0,1;\R)$, but distribution functions $\Phi_{\mu_n}(t)=\chi_{[x_n,y_n)}(t)$ does not converge at any point of interval $(0,1)$. Thus, the operator $\mu\to\Phi_\mu(t)$ considered as an operator from $M(a,b;H)$ with weak star topology to $\R$ is not (sequentially) {\it continuous} for any fixed~$t$. By this reason, the solution operator $\mu\to u(t)$ even for the simplest equation
$$
\frac d{dt}u=\mu, \ \ u(a)=0
$$
is {\it not continuous} with respect to the weak star convergence on measures. This discontinuity makes the corresponding attractors theory essentially more delicate.
\end{rem}

The next theorem gives the analogue of the Helly selection theorem for vector measures and some further useful properties of the weak star convergence in $M(a,b;H)$.

\begin{theorem}\label{bv.th.w*M(H)}
Let $H$ be a separable Hilbert space, segment $[a,b]\subset\R$ and   $\mu_n\in M(a,b;H)$ be a sequence of vector measures.   Let also functions $\Phi_n(t),\ \Phi(t)\in V_0(a,b;H)$ be the corresponding distribution functions. Then
\par
1. The sequence $\mu_n$ is weakly-star convergent in $M(a,b;H)$ to a measure $\mu \in M(a,b;H)$ if and only if it is bounded:
$$
\|\mu_n\|_{M(a,b;H)}\le C
$$
and every subsequence of  $\Phi_{n_k}(t)$ of the sequence $\Phi_n$ contains a further subsequence $\Phi_{n_{k_m}}(t)$ which is weakly convergent in $H$ to $\Phi(t)$ at all point of $[a,b]$ with the exception of at most countable subset  depending on the choice of the subsequence $\Phi_{n_{k_{m}}}$.
\par
2. Let $\mu_n$ be weakly-star convergent to $\mu$. Then, for every segment $[x,y]\subset (a,b)$, the following inequality holds:
%$$
\begin{equation}\label{bv.thw*H2}
|\mu|([x,y])\le\liminf_{\delta\to0}\liminf_{n\to\infty}|\mu_n|([x-\delta,y+\delta]).
\end{equation}
%$$
This inequality also holds when $x$ and $x-\delta$ are substituted by $a$ or $y$ and $y+\delta$ are substituted by $b$.
\end{theorem}

\begin{proof}
\emph{1.} Indeed, as not difficult to see, the weak star convergent of vector measures $\mu_n\to \mu$ is equivalent to the weak star convergence of {\it scalar signed measures} $\mu_{n,h}\to\mu_h$ for every fixed $h\in H$, where
%$$
\begin{equation}
\mu_{n,h}(A):=\int_A(h,\mu_n(dt))_H=(h,\mu_n(A))_H,\ \mu_h(A):=\int_A(h,\mu(dt))_H=(h,\mu(A))_H.
\end{equation}
%$$
From the definitions of $\mu_{n,h},\ \mu_h$ and \eqref{bv.Fmu} we also see that
%$$
\begin{equation}\label{bv.3}
\Phi_{n,h}(t):=\Phi_{\mu_{n,h}}(t)=(h,\Phi_n(t))_H,\quad \Phi_h(t):=\Phi_{\mu_h}(t)=(h,\Phi(t))_H.
\end{equation}
%$$
Then,  the first assertion is a standard corollary of Theorem \ref{bv.th.w*sc} applied to measures $\mu_{n,h}$ and the fact that $H$ is separable.

\emph{2.}
To prove the second assertion, we first note that
$$
|\mu|([x,y])\le |\mu|([x-\eb,y+\eb])
$$
for every $\eb>0$. This inequality together with the fact that the limit distribution $\Phi(t)$ is continuous everywhere except of at most countable set, shows that it is sufficient  to consider the case where $x$ and $y$ are the points of continuity of the limit function $\Phi$.
\par
Let us  fix $\eb>0$ and fix a continuous function $f_{\eb}(t)$ on $[x,y]\subset(a,b)$ with norm one such that
%$$
\begin{equation}
\int_{[x,y]} (f_\eb(t),\mu_\Phi(dt))\ge \Var_x^y(\Phi,H)-\eb.
\end{equation}
%$$
Then, due to continuity of $\Phi$ at $x$ and $y$, we may extend $f_\eb$ to a continuous function on $[a,b]$ (which we also denote by $f_\eb$) without extending its norm in such a way that
%$$
\begin{equation}
\supp f_\eb\subset [x-\delta,y+\delta]\subset(a,b),
\end{equation}
%$$
where $\delta \leq\delta_0(\eb)$ is small enough and
%$$
\begin{equation}
\int_{[a,b]} (f_\eb(t),\mu_\Phi(dt))\ge \Var_x^y(\Phi,H)-2\eb.
\end{equation}
%$$
Thus, we obtain
%$$
\begin{multline}
|\mu|([x,y])=\Var_x^y(\Phi,H)\le \int_{[a,b]} (f_\eb(t),\mu_\Phi(dt))+2\eb=\lim_{n\to\infty}\int_{[a,b]} (f_\eb(t),\mu_{\Phi_n}(dt))+2\eb\leq\\\leq \liminf_{n\to\infty}\Var_{x-\delta}^{y+\delta}(\Phi_n,H)+2\eb=\liminf_{n\to\infty}|\mu_n|([x-\delta,y+\delta])+2\eb,
\end{multline}
%$$
and passing to the limit $\eb\to0$, we get the desired inequality.

The case when $x$ and $x-\delta$ equal $a$ or $y$ and $y+\delta$ equal $b$ can be considered analogously.
\end{proof}
\begin{rem} Note that in general the weak-star convergence $\mu_n\to\mu$ in $M(a,b;H)$ {\it does not imply} the weak-star convergence of $\mu_n$ in $M(x,y;H)$ if $[x,y]$ is a proper subinterval of $[a,b]$. By this reason, the naive estimate
$$
\|\mu\|_{M(x,y;H)}:=|\mu|([x,y])\le\liminf_{n\to\infty}|\mu_n|([x,y])
$$
may be not true and hence the second $\liminf$ in \eqref{bv.thw*H2} is essential. On the other hand, the sequence $\mu_n$ is bounded (and since precompact) in $M(x,y;H)$, so passing to a subsequence, we may assume that $\mu_n\to\bar\mu$ weakly star in $M(x,y;H)$. However, even in this case we cannot get that $\mu=\bar \mu$. Instead, we may only prove that
$$
\bar \mu=\mu+h_1\delta_x+h_2\delta_y
$$
for some $h_1,h_2\in H$ depending of the choice of a subsequence. For instance, the sequence $\mu_n:=\delta_{\frac12-\frac1{2n}}\in M(0,1;\R)$ is convergent weakly star to $\mu=\delta_{\frac12}$ in this space. Let $[x,y]:=[1/2,1]$. Then the restrictions of $\mu_n$ to $M(1/2,1;\R)$ vanish and therefore $\bar\mu=0$. So, $\mu=\bar\mu+\delta_{\frac12}$.
\end{rem}
We now introduce the so-called uniformly non-atomic sets of measures which allow us to overcome the discontinuity problem mentioned in Remark \ref{bv.rem.2}.

\begin{Def}\label{Def9.una} The sets $C\subset M(a,b;H)$ is strongly uniformly non-atomic if there exists a monotone increasing continuous function $\omega:\R_+\to\R_+$ and that $\lim_{z\to0}\omega(z)=0$ and
%$$
\begin{equation}
|\mu|([x,y])\le \omega(|x-y|) \text{ for all $x,y\in[a,b]$ and $\mu\in C$. }
\end{equation}
%$$
Analogously, $C$ is weakly uniformly non-atomic if for every $\psi\in H$ there exists a monotone increasing function $\omega_\psi:\R_+\to\R_+$ satisfying $\lim_{z\to0}\omega_\psi(z)=0$ such that
%$$
\begin{equation}
|(\mu([x,y]),\psi)|\le \omega_\psi(|x-y|) \text{ for all $x,y\in[a,b]$ and $\mu\in C$. }
\end{equation}
%$$
\end{Def}

\begin{cor}\label{bv.cor2}
Let $H$ be a separable Hilbert space, $[a,b]\subset\R$ and  let a sequence of measures $\mu_n\in M(a,b;H)$  be  weakly-star convergent  to a measure $\mu\in M(a,b;H)$.
\par
 Assume that the sequence $\mu_n$ is strongly uniformly non-atomic. Then the limit measure $\mu$ is also non-atomic and the distribution functions $\Phi_{\mu_n}(t)$ (resp. $\Phi_\mu(t)$) of $\mu_n$ (resp. $\mu$) satisfy  the inequalities
%$$
\begin{equation}
\label{bv.norm}
\|\Phi_{\mu_n}(x)-\Phi_{\mu_n}(y)\|_H\le \omega(|x-y|),\ \ \|\Phi_\mu(x)-\Phi_\mu(y)\|_H\le\omega(|x-y|)
\end{equation}
%$$
for all $n$ and $x,y\in[a,b]$.
\par
Assume that the sequence $\mu_n$ is weakly uniformly non-atomic.
Then the limit measure $\mu$ is also non-atomic and the distribution functions $\Phi_{\mu_n}(t)$ (resp. $\Phi_\mu(t)$) of $\mu_n$ (resp. $\mu$) satisfy  the inequalities
%$$
\begin{equation}\label{wbv.norm}
|(\Phi_{\mu_n}(x)-\Phi_{\mu_n}(y),\psi)|\le \omega_\psi(|x-y|),\ \ |(\Phi_\mu(x)-\Phi_\mu(y),\psi)|\le\omega_\psi(|x-y|)
\end{equation}
%$$
for all $n$, $\psi\in H$ and $x,y\in[a,b]$.
\par
 In both cases,  for every $\psi\in H$, the scalar distribution functions $(\Phi_{\mu_n}(\cdot),\psi)$ converge to $(\Phi_\mu(\cdot),\psi)$ in $C[a,b]$.
\end{cor}
\begin{proof} Indeed, the first inequality of \ref{bv.norm} follows from the inequality
$$
\|\Phi_{\mu_n}(x)-\Phi_{\mu_n}(y)\|\le |\mu_n|(|x-y|)\le\omega(|x-y|),
$$
the second one is an immediate corollary of \eqref{bv.thw*H2} and the convergence in $C[a,b]$ follows from the Arzela theorem and the Helly selection theorem stated before. The case of weakly uniformly non-autonomous measures is treated analogously.
\end{proof}

In particular, if the sequence $h_n\in B_{L^1}$ is such that
%$$
\begin{equation}
\limsup_{n\to\infty}\int_{x}^y\|h_n(t)\|_H\,dt\le \omega(|x-y|),
\end{equation}
%$$
then the weak-star limit measure $\mu_{\Phi}$ is strongly non-atomic.
\par
Thus, under the assumptions of the above corollary, the discrete contribution of the BV-function $\Phi$ vanishes. The next corollary gives the condition which guarantees that its singular part also vanishes. To this end we need one more definition

\begin{Def}{\label{bv.uaci}}
A sequence of functions $\{h_n\}_{n=1}^\infty\subset L^1([a,b];H)$ is equi-integrable if
%$$
\begin{equation}
\int_{A}\|h_n(t)\|_H\,dt\le \omega(|A|),
\end{equation}
%$$
for any Borel set $A\subset[a,b]$ (here $|A|$ stands for the Lebesgue measure of $A$ and  $\omega:\R_+\to\R_+$ is a monotone increasing continuous function which does not depend on $n$ and such that $\lim_{x\to0}\omega(z)=0$).
\end{Def}
The next statement is a version of the Dunford-Pettis theorem for vector measures, see \cite{Bg1} for more details.
\begin{theorem}\label{bv.cor3}
Let $H$ be a separable Hilbert space, $[a,b]\subset\R$, and a sequence of measures $\mu_n$ be convergent weakly-star in $M(a,b;H)$ to a measure $\mu$. Let also   the corresponding distributions $\Phi_{\mu_n}(t)$  be absolutely continuous, so  $\{\Phi_{\mu_n}'\}_{n=1}^\infty\subset L^1(a,b;H)$. Then, $\Phi_{\mu_n}'$ are convergent weakly in $L^1(a,b;H)$ if and only if they are equi-integrable. In this case, the limit measure $\mu$  is  absolutely continuous and
%$$
\begin{equation}\label{bv.4}
\lim_{n\to\infty}\int_a^b(\Phi'_{\mu_n}(t),\phi(t))_H\,dt=\int_a^b(\Phi_\mu'(t),\phi(t))_H\,dt,\quad \forall\phi\in L^\infty(a,b;H).
\end{equation}
%$$
\end{theorem}
\begin{comment}
\begin{proof}
Since $H$ is separable we can use vector-valued Lusin's Theorem and conclude that for every $\phi \in L^\infty([a,b];H)$ and for any $\eb>0$ there exist a compact set $K_\eb\subset[a,b]$ and a function $\phi_\eb\in C([a,b];H)$ such that $|[a,b]\setminus K_\eb|\leq \eb$ and
%$$
\begin{equation}\label{bv.5}
\phi_\eb(t)=\phi(t)\ \text{for all } t\in K_\eb,\quad \|\phi_\eb\|_{C([a,b];H)}\leq \|\phi\|_{L^\infty([a,b];H)}.
\end{equation}
%$$
Using \eqref{bv.5} we have
%$$
\begin{multline}\left|\int_a^b(h_n(t),\phi(t))_H\,dt- \int_a^b(h(t),\phi(t))_H\,dt\right|\leq \\\leq
\left|\int_a^b(h_n(t)-h(t),\phi_\eb(t))_H\,dt\right|+\left|\int_a^b(h_n(t),\phi(t)-\phi_\eb(t))_H\,dt\right|+
\left|\int_a^b(h(t),\phi(t)-\phi_\eb(t))_H\,dt\right|\leq\\\leq
\left|\int_a^b(h_n(t)-h(t),\phi_\eb(t))_H\,dt\right|+2\|\phi\|_{L^\infty([a,b];H)}\left(\int_{[a,b]\setminus K_\eb}\|h_n(t)\|_H\,dt+\int_{[a,b]\setminus K_\eb}\|h(t)\|_H\,dt\right)\leq\\\leq
 \left|\int_a^b(h_n(t)-h(t),\phi_\eb(t))_H\,dt\right|+R(\eb),
\end{multline}
%$$
where, since the sequence $\{h_n\}_{n=1}^\infty$ has uniformly absolutely continuous integrals,  function $R(\eb)$ does not depend on $n$, and $R(\eb)$ tends to $0$ as $\eb$ goes to $0$. Just obtained estimate together with \eqref{bv.4} implies the desired result.
\end{proof}
\end{comment}
We conclude this Section by one more result related to the approximation of measures by delta-measures which plays important role in the proof of one of the main results of the work, \mbox{Theorem \ref{th.str-f1}}.

\begin{theorem}\label{bv.th.delta-conv}
Let $H$ be a separable Hilbert space, $[a,b]\subset\R$ and $M(a,b;H)$ be the space of measures with finite total variation and $\mu\in M(a,b;H)$. Then there exists a sequence $\mu_n$ of discrete measures
%$$
\begin{equation}
\mu_n=\sum_{k=0}^{n}\delta_{t_{k,n}}h_{k,n},
\end{equation}
%$$
where $h_{k,n}\in H$ for all $k\in\overline{0,n}$ such that
%$$
\begin{equation}\label{9.mesbound}
\|\mu_n\|_{M(a,b;H)}=\sum_{k=0}^{n}\|h_{k,n}\|_{H}\le\|\mu\|_{M(a,b;H)}
\end{equation}
%$$
and $\Phi_{\mu_n}(t)\to\Phi_\mu(t)$ strongly in $H$ as $n\to\infty$ and uniformly with respect to all $t\in[a,b]$. In particular, $\mu_n\to\mu$ weakly star in $M(a,b;H)$.
\end{theorem}

\begin{proof} We first note that, without loss of generality, we may assume that the measure $\mu$ is non-atomic (i.e., that $\Phi_\mu\in C(a,b;H)$). Indeed, in a general case,  we may split the measure $\mu$ on a discrete and non-atomic part: $\mu=\mu_d+\mu_{cont}$, where
$$
\mu_d=\sum_{n=1}^\infty h_n\delta_{t_n},\ \ \|\mu_d\|_{M}=\sum_{n=1}^\infty\|h_n\|_H\le\|\mu\|_M
$$
and $\|\mu\|_{M}=\|\mu_d\|_M+\|\mu_{cont}\|_M$. By these reasons, we may consider $\mu_d$ and $\mu_{cont}$ separately. In addition, the desired approximation for $\mu_d$ can obviously be chosen by the following expression:
$$
\mu_{d,n}:=\sum_{k=1}^n h_k\delta_{t_k}.
$$
Thus, we assume from now on that $\mu=\mu_{cont}$ and $\Phi_\mu\in C([a,b],H)$. Let us set
%$$
\begin{equation}
t_{k,n}:=a+(b-a)\frac kn,\ \text{where } n\in \N,\ k\in \overline{0,n},
\end{equation}
%$$
and define the sequence of measures $\{\mu_n\}_{n=1}^\infty$ as follows
%$$
\begin{equation}
\mu_n=\sum_{k=0}^{n-1}\mu\Big([t_{k,n},t_{k+1,n})\Big)\delta_{t_{k,n}}.
\end{equation}
%$$
Then, by the construction, $\|\mu_n\|_M\le\|\mu\|_M$ and,  for any fixed $t\in[a,b]$
$$
\|\Phi_\mu(t)-\Phi_{\mu_n}(t)\|=\|\mu([a,t])-\mu([a,t_{k_0,n}]))\|=\|\Phi_\mu(t)-\Phi_\mu(t_{k_0,n})\|
$$
where $k_0$ is the largest $k$ such that $t_{k,n}<t$. Since $|t-t_{k_0,n}|\le\frac1n$ and $\Phi_{\mu}$
is uniformly continuous, we have the uniform convergence $\Phi_{\mu_n}\to\Phi_\mu$. The weak star convergence is an immediate corollary of this uniform convergence and the theorem is proved.
\end{proof}

\begin{rem}
Although approximation of measures by sums of delta-measures  is a standard technical result which can be immediately obtained, say, from Krein-Millman theorem, the convergence of $\mu_n$ to $\mu$ in the weak star topology only is not sufficient for our purposes due to the problems mentioned in Remark \ref{bv.rem.2}. In contrast to the usual weak-star convergence, the result presented above has an extra important property that $\Phi_{\mu_n}\to\Phi_{\mu}$ point-wise and even uniform in the strong topology of $H$. This allows us to overcome the above mentioned problem. In particular, this uniform convergence implies that $\mu_n(\{t\})\to\mu(\{t\})$ strongly in $H$ for every $t\in[a,b]$.
\end{rem}

\section{Appendix 2: Key estimates in  fractional Sobolev spaces }\label{s.ap}
In this Appendix, we prove the key inequality for the $H^\alpha$-norm of the difference $f(w+v)-f(v)$ in terms of the proper norms of the functions $v$ and $w$. To this end, we need the following fractional Leibnitz rule.

\begin{theorem}[Kato-Ponce inequality] \label{th.KP}
Let $\alpha>0$ and constants $r$, $p_1$, $q_1$, $p_2$, $q_2\in(1,\infty)$ are such that
%$$
\begin{equation*}
\frac{1}{r}=\frac{1}{p_1}+\frac{1}{q_1}=\frac{1}{p_2}+\frac{1}{q_2}.
\end{equation*}
%$$
Assume also that the functions $v(x)$, $w(x)$ on  $d$-dimensional torus $\T^d$ ($d\in\N$) satisfy
%$$
\begin{equation*}
v\in H^{\alpha,p_1}(\T^d)\cap L^{p_2}(\T^d),\qquad w\in L^{q_1}(\T^d)\cap H^{\alpha,q_2}(\T^d).
\end{equation*}
%$$
Then the product $vw\in H^{\alpha,r}(\T^d)$ and the following inequality holds
%$$
\begin{equation}
\|vw\|_{H^{\alpha,r}}\le C\left(\|v\|_{H^{\alpha,p_1}}\|w\|_{L^{q_1}}+\|v\|_{L^{p_2}}\|w\|_{H^{\alpha,q_2}}\right),
\end{equation}
%$$
for some positive constant $C=C(\alpha,r,p_1,q_1,p_2,q_2)$.
\end{theorem}
For the proof of this theorem see e.g., \cite{Chem}.
\par
We apply this inequality to verify the following estimate.

\begin{lemma}\label{Lem.v4wHa}
Let $\alpha\in(0,2/5)$ and functions $v$ and $w$ be such that
%$$
\begin{equation}\label{A2.vw}
v\in L^{12}(\T^3)\cap H^1(\T^3),\quad w\in H^{\alpha,12}(\T^3)\cap H^{1+\alpha}(\T^3).
\end{equation}
%$$
Assume also that the function $h\in C^1(\R)$ satisfies
%$$
\begin{equation}\label{ap.h}
|h'(v)|\leq C(1+|v|^3),\ \text{for all }v\in\R,
\end{equation}
%$$
for some constant $C>0$. Then $h(v)w\in H^{\alpha}(\T^3)$ and the following estimate holds:
%$$
\begin{equation}\label{est.main}
\|h(v)w\|_{H^\alpha}\le C_\alpha\left(1+\|v\|_{L^{12}}^{4-\alpha}\right)\left(1+\|v\|_{H^1}^\alpha\right)
\|w\|^{1-\alpha}_{H^{1+\alpha}}\|w\|^{\alpha}_{H^{\alpha,12}},
\end{equation}
%$$
for some positive constant $C_\alpha$.
\end{lemma}
\begin{proof}
We apply Kato-Ponce inequality to the function $h(v)w$ with the following exponents
%$$
\begin{equation}
r=2,\ p_1=\frac{12}{4+5\alpha},\ q_1=\frac{12}{2-5\alpha},\ p_2=\frac{12}{4+\alpha},\ q_2=\frac{12}{2-\alpha},
\end{equation}
%$$
that gives
%$$
\begin{multline}\label{ap.0}
\|h(v)w\|_{H^\alpha}\leq C_\alpha\left(\|h(v)\|_{H^{\alpha,\frac{12}{4+5\alpha}}}\|w\|_{L^{\frac{12}{2-5\alpha}}}+
\|h(v)\|_{L^\frac{12}{4+\alpha}}\|w\|_{H^{\alpha,\frac{12}{2-\alpha}}}\right)\leq\\
2C_\alpha\|h(v)\|_{H^{\alpha,\frac{12}{4+5\alpha}}}\|w\|_{H^{\alpha,\frac{12}{2-\alpha}}},
\end{multline}
%$$
for some $C_\alpha$, where we have used continuous embeddings
%$$
\begin{equation}\label{ap.emb1}
H^{\alpha,\frac{12}{4+5\alpha}}\subset L^\frac{12}{4+\alpha}(\T^3),\quad H^{\alpha,\frac{12}{2-\alpha}}(\T^3)\subset L^{\frac{12}{2-5\alpha}}(\T^3).
\end{equation}
%$$
The resulting terms can be estimated by standard interpolation inequalities. Indeed, using growth assumption \eqref{ap.h} we derive
%$$
\begin{multline}\label{ap.1}
\|h(v)\|_{H^{\alpha,\frac{12}{4+5\alpha}}}\le
 C_\alpha\|h(v)\|^{1-\alpha}_{L^3}\|h(v)\|^\alpha_{H^{1,\frac43}}\leq \\
C_\alpha\left(1+\|v\|^{4}_{L^{12}}\right)^{1-\alpha}\left(1+\left(1+\|v\|^3_{L^{12}}\right)
\|v\|_{H^1}\right)^\alpha\\
C_\alpha\left(1+\|v\|^{4-\alpha}_{L^{12}}\right)\left(1+\|v\|^\alpha_{H^1}\right),
\end{multline}
%$$
for some $C_\alpha$. Also we have
%$$
\begin{equation}\label{ap.2}
\|w\|_{H^{\alpha,\frac{12}{2-\alpha}}}\le C_\alpha\|w\|^{1-\alpha}_{H^{\alpha,6}}\|w\|^\alpha_{H^{\alpha,12}}\leq C_\alpha\|w\|^{1-\alpha}_{H^{1+\alpha}}\|w\|^\alpha_{H^{\alpha,12}},
\end{equation}
%$$
for some $C=C_\alpha$.
Collecting \eqref{ap.0}, \eqref{ap.1}, \eqref{ap.2} we complete the proof.
\end{proof}

\begin{cor}\label{CorA2.0}Let assumptions of Lemma \ref{Lem.v4wHa} be satisfied and in addition $h(0)=0$. Then
%$$
\begin{equation}\label{est.fp0}
\|h(v)w\|_{H^\alpha}\le C_\alpha(1+\|v\|_{L^{12}})^{4-\alpha}\|v\|_{H^1}^\alpha
\|w\|^{1-\alpha}_{H^{1+\alpha}}\|w\|^{\alpha}_{H^{\alpha,12}},
\end{equation}
%$$
for some positive constant $C_\alpha$.
\end{cor}
\begin{proof}
Indeed, in this case,
$$
|h(v)|\le C|v|(1+v^3)
$$
and, therefore,
$$
\|h(v)\|_{H^{1,4/3}}\le C\|v(1+v^3))\|_{L^{4/3}}+\|\Nx v(1+v^3)\|_{L^{4/3}}\le C(1+\|v\|_{L^{12}}^3)\|v\|_{H^1}
$$
which gives the desired estimate. Corollary \ref{CorA2.0} is proved.
\end{proof}
\begin{cor} Let the function $f\in C^2(\R)$ satisfy
$$
|f''(u)|\le C(1+|u|^3)
$$
and let the functions $v$ and $w$ satisfy \eqref{A2.vw} for some $\alpha\in[0,\frac25)$. Then, the following estimate holds:
%$$
\begin{equation}\label{A2.f}
\|f(v+w)-f(v)\|_{H^\alpha}\le C(1+\|v\|_{L^{12}}+\|w\|_{L^{12}})^{4-\alpha}(1+\|v\|_{H^1}+\|w\|_{H^1})^{\alpha}
\|w\|_{H^{1+\alpha}}^{1-\alpha}\|w\|_{H^{\alpha,12}}^\alpha.
\end{equation}
%$$
\end{cor}
Indeed, according to \eqref{est.main} applied to $h=f'(v+sw)$, we have
%$$
\begin{multline}
\|f(v+w)-f(v)\|_{H^\alpha}\le\int_0^1\|f'(v+sw)w\|_{H^\alpha}\,ds\le\\\le
C(1+\|v\|_{L^{12}}+\|w\|_{L^{12}})^{4-\alpha}(1+\|v\|_{H^1}+\|w\|_{H^1})^{\alpha}
\|w\|_{H^{1+\alpha}}^{1-\alpha}\|w\|_{H^{\alpha,12}}^\alpha.
\end{multline}
%$$
\begin{rem} The restriction $\alpha<\frac25$ in Lemma \ref{Lem.v4wHa} is essential. Indeed, it is easy to see that estimate \eqref{est.main} fails for $\alpha=1$. On the other hand, using slightly more sharp interpolation inequality
$$
\|w\|_{L^\infty}\le C\|w\|_{H^{1+\frac25}}^{1-\frac25}\|w\|_{H^{\frac25,12}}^{\frac25},
$$
we see that estimate \eqref{est.main} remains true for $\alpha=\frac25$ as well. We expect that it fails for $\alpha>\frac25$ although the rigorous proof of this fact is out of scope of the paper.
\end{rem}

\end{document}